\documentclass{aomart}
\usepackage{amsmath, amscd, amssymb}

\newcommand{\vect}{\overrightarrow}
\newtheorem{theorem}{Theorem}[section]

\newtheorem{lemma}[theorem]{Lemma}
\newtheorem{prop}[theorem]{Proposition}
\newtheorem{problem}[theorem]{Problem}
\newtheorem{corollary}[theorem]{Corollary}
\newtheorem{fact}[theorem]{Fact}
\newtheorem{conj}[theorem]{Conjecture}
\newtheorem{reduction}[theorem]{Reduction}

\newcommand{\kriva}{{\mathcal A}}
\newcommand{\krivb}{{\mathcal B}}
\newcommand{\krivd}{{\mathcal D}}

\theoremstyle{theorem}
\newtheorem{thm}[theorem]{Theorem}
\newtheorem{lem}[theorem]{Lemma}
\newtheorem*{thmun}{Theorem}
\newtheorem{forefthm}[theorem]{Theorem}

\newcommand{\rittmonoid}{\operatorname{RM}}
\newcommand{\linequi}{\operatorname{LE}}
\newcommand{\skewmonoid}{\operatorname{ST}}
\newcommand{\skequi}{\operatorname{SE}}
\newcommand{\borgarmonoid}{\operatorname{BG}}

\theoremstyle{definition}
\newtheorem{definition}[theorem]{Definition}
\newtheorem{nota}[theorem]{Notation}

\theoremstyle{definition}
\newtheorem{Def}[theorem]{Definition}

\theoremstyle{remark}
\newtheorem{Rk}[theorem]{Remark}

\newcommand{\id}{\operatorname{id}}
\newcommand{\sthat}{\hspace{.1cm}| \hspace{.1cm}}

\newcommand{\ord}{\operatorname{ord}}

\newcommand{\Ga}{{\mathbb G}_a}
\newcommand{\Gm}{{\mathbb G}_m}

\renewcommand{\AA}{{\mathbb A}}
\newcommand{\CC}{{\mathbb C}}

\newcommand{\NN}{{\mathbb N}}
\newcommand{\PP}{{\mathbb P}}
\newcommand{\QQ}{{\mathbb Q}}

\newcommand{\UU}{{\mathbb U}}
\newcommand{\ZZ}{{\mathbb Z}}

\newcommand{\per}{{\operatorname{inv}}}

\newcommand{\cO}{{\mathcal O}}

\newcommand{\fm}{{\mathfrak m}}

\newcommand{\ACFA}{\operatorname{ACFA}}

\newcommand{\alg}{{\rm alg}}

\title{Invariant varieties for polynomial dynamical systems}
\author{Alice Medvedev \and Thomas Scanlon}
\thanks{During the writing of this paper Medvedev was partially supported by NSF FRG DMS-0854839 while
Scanlon was partially supported by NSF grants DMS-0450010, DMS-0854839 and DMS-1001550 and a Templeton Infinity grant}
\address{University of California, Berkeley \\
Department of Mathematics \\
Evans Hall \\
Berkeley, CA 94720-3840}
\email{alice@math.berkeley.edu}

\address{University of California, Berkeley \\
Department of Mathematics \\
Evans Hall \\
Berkeley, CA 94720-3840}
\email{scanlon@math.berkeley.edu}

\begin{document}
\begin{abstract}

We study algebraic dynamical systems (and, more generally, $\sigma$-varieties) $\Phi:\AA^n_\CC \to \AA^n_\CC$ given by coordinatewise univariate polynomials by refining an old theorem of Ritt on compositional identities amongst polynomials.  More precisely,  we find a nearly
canonical way to write a polynomial as a composition of ``clusters'' from which one may easily
read off possible compositional identities.
  Our main result is an explicit description of the (weakly) skew-invariant varieties, that is, for a fixed field automorphism $\sigma:\CC \to \CC$ those algebraic varieties $X \subseteq \AA^n_\CC$ for which  $\Phi(X) \subseteq X^\sigma$.  As a special case, we show that if $f(x) \in \CC[x]$ is a polynomial of degree at
least two which is not conjugate to a monomial, Chebyshev polynomial or a negative Chebyshev polynomial, and
$X \subseteq \AA^2_\CC$ is an irreducible curve which is invariant under the action of $(x,y) \mapsto (f(x),f(y))$ and
projects dominantly in both directions, then $X$ must be the graph of a polynomial which commutes with
$f$ under composition.  As consequences, we deduce a variant of a conjecture of Zhang on the existence of rational points with Zariski dense forward orbits and a strong form of the dynamical Manin-Mumford conjecture for liftings of the Frobenius.

 We also show that in models of $\ACFA_0$, a disintegrated set defined by $\sigma(x) = f(x)$ for a polynomial $f$
 has Morley rank one and is usually strongly minimal, that model theoretic
 algebraic closure is a locally finite closure operator on the nonalgebraic points of this
  set unless the skew-conjugacy class of $f$ is defined over a fixed field of a power of $\sigma$, and that nonorthogonality between two such sets is definable in families if the skew-conjugacy class of
  $f$ is defined over a fixed field of a power of $\sigma$.

\end{abstract}

\maketitle

\section{Introduction}
Let $f_1,\ldots,f_n \in \CC[x]$ be a finite sequence of polynomials over the complex numbers and let $\Phi:\AA^n_\CC \to \AA^n_\CC$ be the map $(x_1,\ldots,x_n) \mapsto (f_1(x_1),\ldots,f_n(x_n))$ given by applying the polynomials coordinatewise.   We aim to explicitly describe those algebraic varieties $X \subseteq \AA^n_\CC$ which are invariant under $\Phi$.  To do so, we solve a more general problem.  We fix a field automorphism $\sigma:\CC \to \CC$, describe those algebraic varieties $X \subseteq \AA^n_\CC$ which are (weakly)
skew-invariant in the sense that $\Phi(X) \subseteq X^\sigma$, and recover the solution to the initial problem by taking $\sigma$ to be the identity map.

We consider this more general problem of classifying the skew-invariant varieties in order to import some techniques from the model theory of difference fields and because we are motivated by some fine structural problems in the model theory of difference fields.   Recall that a difference field is a field $K$ equipped with a distinguished endomorphism $\sigma:K \to K$.  The theory of difference fields, expressed in the first-order language of rings expanded by a unary function symbol for the endomorphism, admits a model companion, $\ACFA$, the models of which we call \emph{difference closed}, and it is the rich structure theory of the definable sets in difference closed fields developed in~\cite{CH-ACFA1} which we employ.

In~\cite{Med} the first author refined the trichotomy theorems of~\cite{CH-ACFA1,CHP} for sets
 defined by formulas of the form $\sigma(x) = f(x)$ where $f$ is a rational function showing
 that they are \emph{disintegrated}, or what is sometimes called \emph{trivial}, unless $f$ is covered by an isogeny of  algebraic groups in the sense that there is a one-dimensional algebraic group $G$, an isogeny $\phi:G \to G^\sigma$, and a dominant rational function $\pi:G \to \PP^1$ with $f \circ \pi = \pi^\sigma \circ \phi$.  In this context, disintegratedness is a very strong property which
asserts that all algebraic relations amongst solutions to disintegrated equations are reducible to \emph{binary} relations.  This consequence and the fact that
the dynamical systems
arising from isogenies are well-understood reduce the problem of describing general $\Phi$-skew-invariant varieties to that of describing skew-invariant curves in the affine plane.

Thus, the bulk of the technical work in this paper concerns the problem of describing those affine plane curves $C \subseteq \AA^2_\CC$ which are $(f,g)$-skew-invariant when $f$ and $g$ are disintegrated polynomials in the sense of the previous paragraph. It is not hard to reduce this problem to describing triples $(h, \pi, \rho)$ of polynomials satisfying $f \circ \pi = \pi^\sigma \circ h$ and $g \circ \rho = \rho^\sigma \circ h$ (see Proposition~\ref{curvetocomposition}).
Possible compositional identities involving polynomials over $\CC$ were explicitly classified by Ritt in~\cite{Ritt} and Ritt's work has been given a conceptually cleaner presentation and has been refined to give a very sharp answer to the question of which quadruples of polynomials $(a,b,c,d)$ in $\CC[x]$ satisfy $a \circ b = c \circ d$  
in~\cite{MZ}.

Our combinatorial analysis of the ingredients of Ritt's theorem
yields a refinement that is in some ways weaker and in other ways stronger than the ones in~\cite{MZ}.
Applying our refinements of Ritt's theorem to the compositional equations involving $f$, $g$, $h$, $\pi$, and $\rho$, we explicitly describe all $(f,g)$-skew-invariant plane curves in terms of a decomposition of $f$ as a compositional product.

We should say a few words as to what we mean by weaker and stronger.   Ritt's theorem asserts that
any one decompositions of a polynomial over $\CC$ may be obtained from any other decomposition
via a finite sequence of explicit identities, or what we call \emph{Ritt swaps}.  From this
theorem one might expect that it would be a routine matter to put a polynomial into a standard form as a composition of indecomposable polynomials.
However, ambiguity as to the character of certain polynomials may be introduced through compositions with linear polynomials.  A central
part of  our argument (as well as of~\cite{MZ}) consists of characterizing exactly how compositional identities involving the
special polynomials appearing in Ritt's theorem and linear polynomials may hold.  While the individual steps in these calculations
are very easy, pinning down all of the possibilities requires an exhaustive analysis.   From this point, our results on
canonical forms diverge. While the formalism of~\cite{MZ} is well suited to studying decompositions of compositional powers, it is not
well adapted to the problem of describing possible skew-invariant varieties.  We  discuss the comparison between our theorems on
polynomial decompositions and those from~\cite{MZ} in detail in the body of the paper.

Our key technical
innovation is the notion of a clustering of a decomposition whereby the various compositional
factors are grouped, or ``clustered'', according to their combinatorial properties, for example,
compatible Chebyshev polynomials are clustered together.  Clusterings are not
canonical, but using some invariants computed from clusterings we may read off properties of possible
compositional identities.  Specifically, with Theorem~\ref{swapsferries} we show that the number and
types of the clusters (see Definition~\ref{one-cluster-def}) appearing in a clustering of a decomposition, as well
as the location of the ``gates'' (see Definition~\ref{gate-def}) are invariants of a polynomial, independent
of a choice of decomposition.

In every reasonable sense, for almost every pair of polynomials $(f,g)$ there are
no $(f,g)$-skew-invariant curves other than products of the form $\{ \xi \} \times
\AA^1$ or $\AA^1 \times \{ \zeta \}$ where $f(\xi) = \sigma(\xi)$ (respectively,
$g(\zeta) = \sigma(\zeta)$).   Indeed, even if $f = g = g^\sigma$, in most cases,
the only additional $(f,f)$-invariant curves
are graphs of iterates of $f$ and their converse relations.  For instance, it
is easy to see that this holds for $f$ indecomposable by using our reformulation of
the existence of an $(f,g)$-skew-invariant curve in terms of compositional identities
$f \circ \pi = \pi^\sigma \circ h$ and $g \circ \rho = \rho^\sigma \circ h$.

More generally, there are four basic sources for skew-invariant curves. Some come from (skew) iteration.
If $f$ is any polynomial and $g = f^{\sigma^n}$, then the graph of $f^{\lozenge n} := f^{\sigma^{n-1}} \circ f^{\sigma^{n-2}} \circ \cdots \circ f$ is $(f,g)$-skew-invariant.
In particular, when $f = f^\sigma$ is fixed by $\sigma$, the graphs of iterates of $f$ (and their converse relations) are $(f,f)$-invariant.  If $f$ is polynomial of degree at least two, then the set of linear polynomials $L$
which skew commute with $f$ in the sense that $f \circ L = L^\sigma \circ f$ is finite, but sometimes is
nontrivial. The curve defined by $y = L(x)$ is necessarily $(f,f)$-skew-invariant.
When $f$ is expressible as a nontrivial compositional product, $f = a \circ b$ , then considering what we call a \emph{plain skew-twist} of $f$, $g := b^\sigma \circ a$, we see that the graph of $b$ is $(f,g)$-skew-invariant. While all possible plain skew-twists can be easily read off from one
expression of $f$ as a composition of indecomposable polynomials, it takes more work to characterize the possible sequences of plain skew-twists
which originate from $f$.
Finally, it can happen that graphs of monomial identities (or their conjugates via some linear change of variables) may be $(f,g)$-invariant.
For example, if $f(x) = x \cdot (1+x^3)^2$ and $g(y) = y \cdot (1+y^2)^3$, then the curve defined by $y^2 = x^3$ is $(f,g)$-invariant.
Our primary task is to prove a precise version of the assertion that these examples exhaust the possibilities for skew-invariant curves.

Our characterization of the invariant varieties appears as a combination of Theorem~\ref{coarsestructurethm}
with Theorem~\ref{realchar}. Given a regular map
$\Phi:\AA^N \to \AA^N$ of the form $(x_1,\ldots,x_N) \mapsto (f_1(x_1),\ldots,f_N(x_N))$ where each $f_i$ is a nonconstant polynomial,
the coordinates may be partitioned according the  trichotomy theorem for difference fields. That is,
$\Phi$ may be realized as a Cartesian product of three maps of this form
where in the first map each polynomial $f_i$ is linear, for the second map each $f_i$ is linearly conjugate to either a power
map, Chebyshev polynomial or negative Chebyshev polynomial, and for the third map each $f_i$ is disintegrated. Then, the $\Phi$-skew-invariant
varieties are products of the skew-invariant varieties for each of these three components.   It is a routine matter to classify
the skew-invariant varieties for sequences of linear polynomials.  It follows from
the theory of one-based groups (or a straightforward degree computation) that the skew-invariant subvarieties for sequences of power maps and
Chebyshev polynomials come from algebraic tori.  Moreover, a skew-invariant variety for sequences of power maps and Chebyshev polynomials
may be further decomposed into products of skew-invariant subvarieties for the subsequences consisting of power maps and Chebyshev polynomials
of the same degree.  We collect all of these observations with Theorem~\ref{coarsestructurethm}.

The most complicated class of skew-invariant varieties appear as skew-invariant subvarieties of $(\AA^N,\Phi)$ where for some sequence of disintegrated polynomials $f_1,\ldots,f_N$ the map $\Phi$ takes the form $(x_1,\ldots,x_N) \mapsto (f_1(x_1),\ldots,f_N(x_N))$.
Using disintegratedness, we see that any $\Phi$-skew-invariant variety must be a component of the intersection of
pullbacks of $(f_i,f_j)$-skew-invariant varieties ranging over all pairs $(i,j)$ with $1 \leq i < j \leq n$.  Thus, the classification of $\Phi$-skew-invariant
varieties reduces to the case that $N = 2$.  With Theorem~\ref{realchar} we clarify the sense in which such invariant
curves must come from compositions of skew twists, monomial equations and graphs of twisted iterates.  In particular, we give
very tight bounds on the degrees of the monomial equations which might appear in terms of some refined degrees (which are bounded
by the degree of the polynomial in the usual sense) of the indecomposable polynomials appearing in some complete decomposition of $f_1$.
The following is an especially notable special case.

\begin{thmun}[Theorem~\ref{hhprop}]
Let $f(x) \in \CC[x]$ be a polynomial of degree at least two which is not conjugate to a monomial,
a Chebyshev polynomial or a negative Chebyshev polynomial.
Let $N \in \ZZ_+$ be a positive integer and $X \subseteq \AA^N_\CC$ be
an irreducible subvariety of affine $N$-space over the complex numbers which is
invariant under the coordinatewise
action of $f$.   Then $X$ is defined by equations over
the form $x_i = g(x_j)$ and $x_k = \xi$ where $g$ is a polynomial which commutes
with $f$ and $\xi$ is a fixed point of $f$.   Moreover, $g$ takes the form $L \circ h^{\circ m}$ for
some $m \in \NN$ where $h^{\circ \ell} = f$ for some $\ell$ and $L$ is a linear polynomial
which commutes with a compositional power of $h$.
\end{thmun}

We apply our results on skew-invariant varieties to address problems of two different kinds.
We prove variants of two conjectures of Zhang~\cite{Zh} on the arithmetic of dynamical
systems. We also pin down definable structure on, and definable relations between, sets
defined by $\sigma(x) = f(x)$ for some polynomial $f$ in $\ACFA_0$.

Zhang conjectured that if $\phi:X \to X$ is a polarizable dynamical system over some number field $K$, then there is a
point $a \in X(K^{\alg})$ whose forward orbit under $\phi$ is Zariski dense~(Conjecture 4.1.6 of \cite{Zh}).  We consider
  a situation inspired by Zhang's conjecture, but which is at one level more general in that we drop the polarizability hypothesis and
  strengthen the conclusion in that one need not pass to the algebraic closure to find the desired point with a Zariski dense orbit,
  but in another sense more special in that the map $\phi$ is assumed to be given by a sequence of univariate polynomials.  Let us note
  here a somewhat special case of our Theorem~\ref{Zhangconjdense}.

\begin{thmun}
If $K$ is any field of characteristic zero and $\Phi:\AA^n_K \to \AA^n_K$ is given by a sequence
of univariate polynomials each of degree at least two, then there is a point $a \in \AA^n(K)$ with
a Zariski dense $\Phi$-forward orbit.
\end{thmun}

In fact, we prove a somewhat stronger result in which some of the $f_i$s are allowed to be linear.

In another direction we prove a refined version of Zhang's Manin-Mumford conjecture for dynamical systems lifting a Frobenius.   Zhang
conjectured that if $\phi:X \to X$ were a polarized dynamical system over $\CC$ and $Y \subseteq X$ were an irreducible closed
subvariety for which the $\phi$-preperiodic points lying on $Y$ were Zariski dense in $Y$, then $Y$ must itself be a $\phi$-preperiodic
variety~(Conjecture 1.2.1 of~\cite{Zh}).  Counterexamples to this statement have been advanced and the conjecture
itself has been reformulated~\cite{GhTuZh}.  As with our theorem on the density of rational orbits, our dynamical Manin-Mumford theorem is both
more and less general than what is predicted by the corrected dynamical Manin-Mumford conjecture.  We do not require polarizability, though
we do consider only periodic points and dynamical systems arising from liftings of the Frobenius.   Our precise statement is
given as Theorem~\ref{frobpolyper}.    Let us mention a special case here.

\begin{thmun}
Suppose that $q$ is a power of a prime $p$ and that $f(x) \in \ZZ[x]$ is a polynomial of degree $q$ for which $f(x) \equiv x^q \mod p \ZZ[x]$ but $f$ is not conjugate to $x^q$ itself, the $q^\text{th}$ Chebyshev polynomial or a negative Chebyshev polynomial,
then any irreducible variety $X \subseteq \AA^n_\CC$
containing a Zariski dense set of $n$-tuples of $f$-periodic
points must be defined by finitely many equations of the forms $x_i = \zeta$ for some $f$-periodic point $\zeta$ and  $x_j = L \circ \alpha^{\circ m}(x_k)$ for some $m \in \NN$ where $\alpha^{\circ N} = f$ for some $N$ and $L$ commutes with some compositional power of $\alpha$.
\end{thmun}

In the case of differential fields, Hrushovski and Itai showed that there are model complete theories
of differential fields other than the theory of differentially closed
fields~\cite{HI}.  It is still open whether or not there are model complete theories
of difference fields other than $\ACFA$, but if  there were some formula $\theta(x)$
defining in a difference closed field a set of $D$-rank one having only finitely many
algebraic realizations such that for every other formula $\eta(y,z)$ the set of
parameters $\{ b :  \theta(x) \text{ is nonorthogonal to } \eta(y,b) \}$
were definable, then one could produce a new model complete difference field by omitting
the nonalgebraic types in  $\theta$.  Towards this goal, we prove Theorem \ref{hruit-thm}.

\begin{thmun}
  For a nonconstant polynomial $f$, the set of polynomials $g$ with $({\AA}^1,
  g) \not \perp ({\AA}^1, f)$ is definable if and only if  $f$ is not skew-conjugate to $f^{\sigma^n}$
  for every positive integer $n \in \ZZ_+$.
\end{thmun}

A byproduct of this analysis is an explicit characterization of the algebraic closure
operator on disintegrated sets defined by $\sigma(x) = f(x)$, and the observation that this set is strongly
minimal unless $f(x)$ is skew-conjugate to $x^k \cdot u(x)^n$ for some polynomial $u$ and some $n > 1$, and in any case has Morley rank one if $f$ is disintegrated (Theorem \ref{finiteRM}).

This paper is organized as follows. In Section~\ref{coarsestructure}
we lay out our notation and begin our analysis by reducing the
problem to that of describing skew-invariant curves for pairs of disintegrated polynomials.
We then convert this problem to one about compositional identities and lay out the problem
in terms of certain monoid actions.   The technical
work on compositional identities is spread over the next four sections.  In Section~\ref{section41}
we identify all of the possible ways in which one Ritt polynomial may be linearly related to
another Ritt polynomial.  In Section~\ref{clustersec} we study clusterings of decompositions ending the
section with a theorem on invariants of polynomials computed from these clusterings.  In Section~\ref{canformsec}
we intensify the study of the monoid actions producing canonical forms.  With Section~\ref{puttingtogether}
we complete the technical work converting the results on clusterings and monoid actions to a
characterization of the skew-invariant curves for actions of pairs of polynomials.
In Section~\ref{Applications} we conclude with three applications of our results to definability
of orthogonality, Zhang's conjecture on the density of dynamical orbits, and
a version of the dynamical Manin-Mumford conjecture for Frobenius lifts.

We thank M. Zieve for sharing a preliminary version of~\cite{MZ} and for discussing issues around compositional identities of polynomials and rational functions.  We thank the referee for subjecting our manuscript to a thorough review and suggesting numerous improvements.

\section{Coarse structure of skew-invariant varieties}
\label{coarsestructure}
In this section we compare the formalism of algebraic dynamical systems and of $\sigma$-varieties to
establish the relevance of the model theory of difference fields to the study of algebraic dynamics.  We then
interpret the fundamental trichotomy theorem for difference fields in terms of skew-invariant varieties.   Using
this interpretation, we reduce the analysis of skew-invariant varieties for maps given by coordinatewise
actions of univariate polynomials on affine space to that of describing the skew-invariant curves in
$\AA^2$ for pairs of disintegrated polynomials.
We close out this section by recalling Ritt's theorem on polynomial
decompositions in detail and by formalizing that theorem in
terms of actions of various monoids.  In so doing, we convert the
problem of describing invariant varieties into questions about
canonical forms for these monoid actions.

\subsection{Algebraic dynamics and $\sigma$-varieties}
\label{ADsect}

\begin{definition}
A \emph{difference field} $(L,\sigma)$ is a field $L$ given together with a distinguished field endomorphism $\sigma:L \to L$.
The \emph{fixed field} of a difference field $(L,\sigma)$ is the subfield
$\operatorname{Fix}(\sigma) := \{ a \in L ~:~ \sigma(a) = a \}$.
\end{definition}

If $X$ is an algebraic variety over the difference field $(L,\sigma)$, then the $\sigma$-transform of $X$, $X^\sigma$, is the
base change of $X$ to $L$ via $\sigma$.   More concretely, if $X$ is a closed subvariety of some
affine space, then $X^\sigma$ is defined by the same equations as $X$ but with $\sigma$ applied to the coefficients.
At the level of rational points, $a \in X(K) \leftrightarrow \sigma(a) \in X^\sigma(K)$.
The $\sigma$-transform
gives a endofunctor of the category of algebraic varieties over $L$.  That is, if $f:X \to Y$ is a morphism of varieties over
$L$, then there is a uniquely defined morphism $f^\sigma:X^\sigma \to Y^\sigma$ of varieties over $L$ where again, concretely,
on affine charts the polynomials defining $f^\sigma$ are the images under $\sigma$ of the polynomials defining $f$.

\begin{definition}
Following Pink and R\"{o}ssler~\cite{PR}, a \emph{$\sigma$-variety} over the difference field $(L,\sigma)$ is a pair $(X,f)$ where
$X$ is an algebraic variety over $L$ and $f:X \to X^\sigma$ is a dominant morphism from $X$ to its $\sigma$-transform $X^\sigma$.  A
morphism of $\sigma$-varieties $\alpha:(X,f) \to (Y,g)$ is given by a morphism of varieties $\alpha:X \to Y$ for which
$\alpha^\sigma \circ f = g \circ \alpha$.

$$
\begin{CD}
X @>{f}>> X^\sigma \\
@V{\alpha}VV   @VV{\alpha^\sigma}V \\
Y @>{g}>> Y^\sigma
\end{CD}
$$

We say that two $\sigma$-varieties $(X,f)$ and $(Y,g)$ are \emph{skew-conjugate} if
they are isomorphic as $\sigma$-varieties.
\end{definition}

In particular, univariate polynomials $f$ and $g$ (which give $\sigma$-varieties on $\AA^1$) are skew-conjugate if there is a linear polynomial $\alpha$ such that $g = \alpha^{\sigma} \circ f \circ \alpha^{-1}$.

\begin{definition}
An \emph{algebraic dynamical system} over a field $K$ is a pair $(X,f)$ consisting of an algebraic variety $X$ over $K$ and a dominant
regular map $f:X \to X$.  A morphism $\alpha:(X,f) \to (Y,g)$ of algebraic dynamical systems is given by a regular map $\alpha:X \to Y$
for which $\alpha \circ f = g \circ \alpha$.
\end{definition}

The algebraic dynamical system $(X,f)$ over $K$ may be regarded as a $\sigma$-variety over $(K,\operatorname{id}_K)$.

An algebraic dynamical system $(X,f)$ gives rise to a monoid action of $\NN$ on $X$ via iteration of $f$.  We define $f^{\circ n}$
by recursion on $n$ with $f^{\circ 0} := \operatorname{id}_X$ and $f^{\circ (n+1)} := f \circ f^{\circ n}$.  For
a rational point $a \in X(K)$ we define the forward orbit of $a$ under $f$ to be
$\cO_f(a) := \{ f^{\circ n} (a) ~:~ n \in \NN \}$.  The point $a$ is said to be \emph{periodic} if
$f^{\circ n}(a) = a$ for some $n \in \ZZ_+$ and to be \emph{pre-periodic} if $\cO_f(a)$ is finite (or, equivalently,
if $f^{\circ n}(a)$ is periodic for some $n \in \NN$).

If $(X,f)$ is an algebraic dynamical system and $Y \subseteq X$ is a subvariety, then we say that $Y$ is an $f$-invariant variety if
$f(Y) = Y$, or what is the same thing, that $f(Y(K))$ is Zariski dense in $Y$ when $K$ is an algebraically closed field.  We say that
$Y$ is weakly $f$-invariant if $f(Y) \subseteq Y$.
If $(X,f)$ is an algebraic dynamical system and $a \in X(K)$ is any point, then the Zariski closure of $\cO_f(a)$ is
a weakly $f$-invariant subvariety of $X$.   Thus, $\cO_f(a)$ is Zariski dense in $X$ if an only if for no
$n \in \NN$ does $f^{\circ n}(a)$ lie on a (possibly reducible) proper $f$-invariant variety.   In this
way Zhang's conjecture on the existence of algebraic points with Zariski dense forward orbits
(see Section~\ref{ddo}) may be understood as an assertion that there are few $f$-invariant varieties.

In the more general context of a $\sigma$-variety, iteration need not give rise to maps from $X$ back to itself, but it still makes sense.
For a $\sigma$-variety $(X,f)$ over $(L,\sigma)$
we define the skew-iteration of $(X,f)$ by recursion on $n$ setting
 $f^{\lozenge 0} := \operatorname{id}_X$ and
$f^{\lozenge (n+1)} := f^{\sigma^n} \circ f^{\lozenge n}$.
Observe that $(X,f^{\lozenge n})$ is a $\sigma^n$-variety over
$(L,\sigma^n)$.

To distinguish Cartesian powers from $f^{\lozenge n}$ and $f^{\circ n}$, we sometimes
write $f^{\times n}$ for the map $f^{\times n}:X^{\times n} \to (X^\sigma)^{\times n}$ given by
$(x_1,\ldots,x_n) \mapsto (f(x_1),\ldots,f(x_n))$.

For $(X,f)$ a $\sigma$-variety over the difference field $(L,\sigma)$, a sub-variety $Y \subseteq X$ is weakly $f$-skew-invariant if $f(Y) \subseteq Y^\sigma$.   The subvariety $Y$ is
$f$-skew-invariant  if $f(Y) = Y^\sigma$, or, equivalently, if
$(Y,f \upharpoonright Y)$ is a $\sigma$-variety.   A weakly skew-invariant variety need not be skew-invariant, but there is a naturally associated maximal $f$-skew-invariant subvariety.

\begin{Def}
Let $(X,f)$ be a $\sigma$-variety and $Y \subseteq X$ a subvariety of $X$.   The
$f$-skew-invariant part of $Y$ is the subvariety
$$
Y_\per := \bigcap_{n=0}^\infty (f^{\lozenge n} (Y))^{\sigma^{-n}} \text{ .}
$$
\end{Def}

\begin{prop}
\label{maxinv}
If $(X,f)$ is a $\sigma$-variety and $Y \subseteq X$ is a subvariety of $X$, then $Y_\per$ is the
maximal $f$-skew-invariant subvariety of $Y$.  If $Y$ is weakly $f$-skew-invariant, then
$Y_\per = (f^{\lozenge n}(Y))^{\sigma^{-n}}$ for $n \gg 0$.
\end{prop}
\begin{proof}
From its definition, we have $Y_\per^\sigma = (\bigcap_{n=0}^\infty (f^{\lozenge n}(Y))^{\sigma^{-n}})^\sigma
\subseteq (\bigcap_{n=1}^\infty (f^{\lozenge n}(Y))^{\sigma^{-n}})^\sigma =
\bigcap_{m=0}^\infty f^{\lozenge m+1} (Y)^{\sigma^{-m}} = f (\bigcap_{m=0}^{\infty} f^{\lozenge m}(Y)^{\sigma^{-m}})
= f(Y_\per)$.  As $\dim(Y_\per^\sigma) \geq \dim(f(Y_\per))$ and the number of components of $Y_\per^\sigma$
is at least that of $f(Y_\per)$, we conclude that $f(Y_\per) = Y_\per^\sigma$.  On the other hand, if
$Z \subseteq Y$ were $f$-skew-invariant, then for every $n$ we would have
$Z = f^{\lozenge n}(Z)^{\sigma^{-n}} \subseteq f^{\lozenge n}(Y)^{\sigma^{-n}}$.  Hence, $Z \subseteq Y_\per$ so
that $Y_\per$ is the maximal $f$-skew-invariant subvariety of $Y$.

If $Y$ were weakly $f$-skew-invariant, then the intersection defining $Y_\per$ would be an intersection over
 a decreasing chain, and, thus, equal to $f^{\lozenge n}(Y)^{\sigma^{-n}}$ for $n \gg 0$ by
 Noetherianity.
\end{proof}

\begin{prop}
\label{pushpullinv}
If $\pi:(X,f) \to (Y,g)$ is a map of $\sigma$-varieties, $Z \subseteq X$ is a subvariety of $X$ and $W \subseteq Y$
is a subvariety of $Y$, then $\pi(Z_\per) = \pi(Z)_\per$ and $\pi^{-1}(W_\per) = \pi^{-1}(W)_\per$.
\end{prop}

\begin{proof}
From the fact that $\pi^\sigma \circ f = g \circ \pi$, we conclude that
$\pi^{\sigma^n} \circ f^{\lozenge n} = g^{\lozenge n} \circ \pi$ for any $n \in \NN$.  Thus,
from the definition of the skew-invariant part we have
$\pi(Z)_\per = \bigcap_{n=0}^\infty (g^{\lozenge n}(\pi(Z)))^{\sigma^{-n}}
= \bigcap_{n=0}^\infty (\pi^{\sigma^n} \circ f^{\lozenge n} (Z))^{\sigma^{-n}} =
\pi (\bigcap_{n=0}^\infty (f^{\lozenge n}(Z))^{\sigma^{-n}}) = \pi(Z_\per)$, as claimed.

Likewise, $\pi^{-1}(W)_\per = \bigcap_{n=0}^\infty (f^{\lozenge n} \pi^{-1} W)^{\sigma^{-n}}
= \bigcap_{n=0}^\infty ((\pi^{\sigma^n})^{-1} g^{\lozenge n} (W))^{\sigma^{-n}}
= \pi^{-1} \bigcap_{n=0}^\infty (g^{\lozenge n}(W))^{\sigma^{-n}} = \pi^{-1}(W_\per)$.
\end{proof}

Let us note that if $(X,f)$ is a $\sigma$-variety over $(K,\operatorname{id}_K)$, then $(X,f)$ is simply an algebraic dynamical system  over $K$,
$f^{\lozenge n} =  f^{\circ n}$ for each $n \in \NN$, and a subvariety $Y \subseteq X$ is $f$-skew-invariant just in case it is
$f$-invariant.  However, if we start with an algebraic dynamical system
$(X,f)$ over some field $K$ and then form the base change $(X,f)_L$ to some difference
field $(L,\sigma)$ where $\sigma \upharpoonright K = \operatorname{id}_K$, the notions of an $f$-invariant subvariety and of an
$f$-skew-invariant variety need not coincide.  Moreover, there are algebraic dynamical systems
$(X,f)$ and $(Y,g)$ which are non-isomorphic as algebraic dynamical systems
(and remain so after any field extension), but which become isomorphic as $\sigma$-varieties after an appropriate base extension.
For example, $(\AA^1,x \mapsto x+1)$ and $(\AA^1,\operatorname{id})$ are clearly not isomorphic as algebraic dynamical systems, but after
base change to a difference field containing a solution to the difference equation $\sigma(b) = b+1$, they become isomorphic as
$\sigma$-varieties.

\subsection{Model theory of difference fields}
\label{mtprelim}

In this section we translate some of the fundamental theorems on the model theory of difference fields to a more geometric language.
The reader can find a more thorough treatment of these connections in~\cite{CH-AD1}.  All of the
theorems on the model theory of difference fields which we require can be found in~\cite{CH-ACFA1}.

A difference field $(K,\sigma)$
is \emph{difference closed} if it is existentially closed in the class of difference fields.  That is, if a finite system of
difference equations and inequations over $K$ has a solution in some difference field extending $(K,\sigma)$, then it
already has a solution in $(K,\sigma)$.  By successively adjoining solutions to such systems of difference equations and
  inequations, one sees that every difference field embeds into a difference closed field.
The class of difference closed fields is axiomatized by three schemata of axioms expressible in the language of difference fields, the
language of rings augmented by a unary function symbol for the distinguished endomorphism.  It is obvious that the first schema is
given by a first-order sentence.  A routine argument expresses the second as a countable list of sentences.  However, the
last schema requires absolute irreducibility of a variety to be a first-order property of the coefficients of
the defining equations.  This is attained by bounding the degrees for the ideal membership problem in polynomial rings.

\begin{fact}[Theorem 1.1 of~\cite{CH-ACFA1}]
A difference field $(K,\sigma)$ is difference closed if and only if
\begin{enumerate}
\item $\sigma$ is an automorphism of $K$,
\item $K$ is algebraically closed, and
\item for any irreducible affine variety $X$ defined over $K$ and any irreducible subvariety $Y \subseteq X \times X^\sigma$
for which the two projections $Y \to X$ and $Y \to X^\sigma$ are dominant, there is a point $a \in X(K)$ with $(a,\sigma(a)) \in Y(K)$.
\end{enumerate}
\end{fact}

From the axioms for difference closed fields, we see that if $(K,\sigma)$ is a difference closed field, $X$ is an irreducible
variety over $K$, $\Gamma \subseteq X \times X^\sigma$ is an irreducible subvariety of $X \times X^\sigma$ for which the
two projections $\Gamma \to X$ and $\Gamma \to X^\sigma$ are dominant and we define  the $(K,\sigma)$ points of
$(X,\Gamma)$ by
$$
(X,\Gamma)^\sharp(K,\sigma) := \{ a \in X(K) ~:~ (a,\sigma(a)) \in \Gamma(K) \} \text{ ,}
$$
then $(X,\Gamma)^\sharp(K,\sigma)$ is Zariski dense in $X$.  In particular, if $(X,f)$ is an irreducible
$\sigma$-variety over a difference closed field $(K,\sigma)$, then
$$
\{ a \in X(K) ~:~ f(a) = \sigma(a) \} = (X,\Gamma(f))^\sharp(K,\sigma)
$$
is
Zariski dense in $X$ where $\Gamma(f)$ is the graph of $f$.   Moreover, an absolutely irreducible
subvariety $Y \subseteq X$ is $f$-skew-invariant
if and only if $Y(K) \cap (X,\Gamma(f))^\sharp(K,\sigma)$ is Zariski dense in $Y$.  In this sense, we see that there
are enough rational points defined over a difference closed field to reflect the geometry of $\sigma$-varieties, or even of
algebraic dynamical systems regarded as $\sigma$-varieties obtained by base change from the fixed field.
We use this observation to
translate results from the structure theory of definable sets in difference closed fields to the language of algebraic dynamical systems and $\sigma$-varieties.

Let us recall the notion of orthogonality, specialized to the case of $\sigma$-varieties.   That $\sigma$-varieties
of different characters (\emph{eg} those coming from group actions versus those unrelated to groups) are
orthogonal is the first step in the reduction of the study of skew-invariant varieties in general to the special case
of skew-invariant curves in the plane.

\begin{definition}
Two  absolutely irreducible $\sigma$-varieties $(X,f)$ and $(Y,g)$ over a difference field $(K,\sigma)$
are \emph{almost orthogonal}, written $(X,f) \perp^a_K (Y,g)$, if every
$(f,g)$-skew-invariant subvariety of $X \times Y$ is a finite union of products of components of $f$-skew-invariant and
$g$-skew-invariant varieties. If for every extension of difference fields $(L,\sigma) \supseteq (K,\sigma)$, we
have $(X_L,f) \perp^a_L (Y_L,g)$, then  $(X,f)$ and $(Y,g)$ are \emph{orthogonal}, written $(X,f) \perp (Y,g)$,
\end{definition}

\begin{Rk}
\label{orthocheat}
What we are calling (almost) orthogonality is usually called \emph{full (almost) quantifier-free orthogonality} in the model theory literature.
The subtler notions of orthogonality for types, while present in the background, are not
directly relevant to the problems we consider here.   In fact,
nonorthogonality of $(X,f)$ and $(Y,g)$ need not imply that some type in $(X, f)^\sharp$
is nonorthogonal to some type in $(Y,g)^\sharp$, as the $(f,g)$-skew-invariant subvariety
witnessing nonorthogonality in our sense may have no sharp points at all.
Our notion does correspond to eventual non-orthogonality: if $(X,f)$ and $(Y,g)$ are non-orthogonal
in our sense, then for some $m$,
 some type in $(X, f^{\lozenge m})^{\sharp}$,
where $(X, f^{\lozenge m})$ is regarded as a $\sigma^m$-variety,
is nonorthogonal to some type
in $(Y, g^{\lozenge m})^{\sharp}$, even in the reduct to $\sigma^m$.
We return to issues around quantifier elimination in Section~\ref{newsectionstrmin}.
\end{Rk}

\begin{Rk}
The distinction between almost orthogonality and orthogonality is real.  For example, if $K$ is any
field of characteristic zero, then the $\sigma$-varieties $(\AA^1,\operatorname{id}_{\AA^1})$ and
$(\AA^1,x \mapsto x + 1)$ are almost orthogonal over $(K,\operatorname{id}_K)$, but after base extension to any
difference field containing a solution $a$ to the difference equation $\sigma(x) = x + 1$, these $\sigma$-varieties
are isomorphic as $\sigma$-varieties via the map $x \mapsto x + a$ so that the graph of this isomorphism
gives a skew-invariant variety not expressible as a product witnessing the non-orthogonality of these
two $\sigma$-varieties.    As a general rule, such instances of almost orthogonality but non-orthogonality are
mediated by the action of a definable group.   Thus, for the $\sigma$-varieties of principal concern to us,
the \emph{disintegrated} $\sigma$-varieties (see Definition~\ref{deftrivial}), at least when working over
an algebraically closed base, there is no difference between almost orthogonality and orthogonality.
\end{Rk}

The nonorthogonality relation defines an equivalence relation on the set of $\sigma$-varieties
whose underlying varieties are irreducible curves.

\begin{prop}
\label{nonorthoequiv}
If $(X,f)$, $(Y,g)$, and $(Z,h)$ are absolutely irreducible $\sigma$-varieties over some difference
field $(K,\sigma)$ for which each of $X$, $Y$ and $Z$ is a curve, $(X,f) \not \perp (Y,g)$
and $(Y,g) \not \perp (Z,h)$, then $(X,f) \not \perp (Z,h)$.
\end{prop}

\begin{Rk}
From the model theoretic perspective, Proposition \ref{nonorthoequiv} is almost a special case of the fact that the nonorthogonality
relation is an equivalence relation on minimal types where the provisos from Remark~\ref{orthocheat} explain the
sense in which this remark is only approximately true.
\end{Rk}

When we view a subvariety $\Gamma \subseteq X \times Y$ as a many valued function from $X$ to $Y$, we call it a \emph{correspondence} from $X$ to $Y$.
Before proving Proposition~\ref{nonorthoequiv}, we recall what it means to compose correspondences and
 record some basic properties of this operation.

\begin{definition} \label{curve-compose-def}
Let $X$, $Y$, and $Z$ be three varieties over some field $K$, $\Gamma \subseteq X \times Y$ and
$\Xi \subseteq Y \times Z$ subvarieties of $X \times Y$ and $Y \times Z$, respectively.   Let
$\pi:X \times Y \times Z \to X \times Z$ be the projection map onto the first and third coordinates.
We define $\Xi \circ \Gamma := \pi ((\Gamma \times Z) \cap (X \times \Xi))$, the projection
of the fibre product of $\Gamma$ and $\Xi$ over $Y$.   If $W \subseteq X$ is any
subvariety, then $\Gamma(W) := \Gamma \circ \Delta_W$ where $\Delta_W \subseteq W \times X$ is the
graph of the embedding of $W$ in $X$.
\end{definition}

\begin{Rk}
At the level of points, provided that $K = K^{\text{alg}}$,  $\Xi \circ \Gamma$ is the
Zariski closure of the set
$$
\{ (a,c) \in (X \times Z) (K) ~:~ (\exists b \in Y(K))  (a,b) \in \Gamma(K) ~\&~ (b,c) \in \Xi(K) \}
$$
\end{Rk}

\begin{Rk}
One treats a correspondence $\Gamma \subseteq X \times Y$ as a many valued function from $X$ to $Y$.
Provided that the projection map $\pi:\Gamma \to X$ is dominant, this ``function'' is defined almost
everywhere.  If $f:X \to Y$ is a rational function, then one regards $f$ as a correspondence by identifying $f$ with
its graph $\Gamma(f)$. Switching the roles of input and output in $\Gamma(f)$ gives what we call \emph{the converse relation to the graph of $f$}. If the projection map to the output coordinate is finite, then $\Gamma$ may be regarded as a finite valued function. In the cases of interest to us, $X$ and $Y$ are irreducible curves and $\Gamma$ is a curve each of whose components projects dominantly to $X$ and to $Y$.  Here, $\Gamma(K)$ really is a finite-to-finite
correspondence between $X(K)$ and $Y(K)$.
\end{Rk}

\begin{Rk}
Even if $X$, $Y$, $Z$, $\Gamma$ and $\Xi$ are all irreducible, then $\Xi \circ \Gamma$ may be reducible.
For example, if $X = Y = Z = \AA^1$, $f:\AA^1 \to \AA^1$ is any polynomial of degree at least two, $\Gamma$
is the graph of $f$ and $\Xi$ is its converse relation, then $\Xi \circ \Gamma$ is defined by
$f(x) = f(z)$ which always has the diagonal as one component and other components
corresponding to the factors of the polynomial $\frac{f(x) - f(z)}{x-z}$.
\end{Rk}

\begin{lem}
\label{compper}
Let $(X,f)$, $(Y,g)$ and $(Z,h)$ be $\sigma$-varieties over some difference field $(K,\sigma)$.
Suppose that $\Gamma \subseteq X \times Y$ is weakly $(f,g)$-skew-invariant and that
$\Xi \subseteq Y \times Z$ is weakly $(g,h)$-skew-invariant.  Then $\Xi \circ \Gamma$ is
weakly $(f,h)$-skew-invariant and $(\Xi \circ \Gamma)_{\per}  = (\Xi_\per \circ \Gamma_\per)_\per$.
\end{lem}
\begin{proof}
Clearly, the intersection of two (weakly) $(f,g,h)$-skew-invariant varieties is $(f,g,h)$-skew-invariant
so that $(\Gamma \times Z) \cap (X \times \Xi)$ is weakly $(f,g,h)$-skew-invariant and
$(\Gamma_\per \times Z) \cap (X \times \Xi_\per)$ is $(f,g,h)$-skew-invariant. Let
$\pi:X \times Y \times Z \to X \times Z$ be the projection map.  By Proposition~\ref{pushpullinv},
$(\Xi \circ \Gamma)_\per = (\pi((\Gamma \times Z) \cap (X \times \Xi)))_\per =
( \pi((\Gamma_\per \times Z) \cap (X \times \Xi_\per)))_\per = (\Xi_\per \circ \Gamma_\per)_\per$.
\end{proof}

With our observations on compositions in place, we prove Proposition~\ref{nonorthoequiv}.

\begin{proof}
Taking a base change if need be, we find
$\Gamma \subseteq X \times Y$ and $\Xi \subseteq Y \times Z$ which are $(f,g)$-skew-invariant (respectively, $(g,h)$-skew-invariant) curves
witnessing $(X,f) \not \perp (Y,g)$ and $(Y,g) \not \perp (Z,h)$.  By Lemma~\ref{compper}, $(\Xi \circ \Gamma)_\per$ is
an $(f,h)$-skew-invariant subvariety of $X \times Z$.  Since $\Xi$ and $\Gamma$ are curves for which the restriction of the
various projection maps are all finite, $\Xi \circ \Gamma$ is a curve all of whose components project dominantly onto $X$ and $Z$.
Because the maps $f$, $g$, and $h$ are finite,
$\dim (\Xi \circ \Gamma)_\per = \dim (\Xi \circ \Gamma) = 1$.  Hence, $(\Xi \circ \Gamma)_\per$ witnesses $(X,f) \not \perp (Z,h)$.
\end{proof}

We employ the theory of orthogonality to reduce the study of skew-invariant varieties to that of plane
curves.  To this end we use a simple, but powerful, observation that
orthogonality of products of $\sigma$-varieties follows from pairwise orthogonality.

\begin{prop}
\label{productortho}
Given a difference field $(K,\sigma)$ and two sequences of $\sigma$-varieties
$(X_1,f_1)$, \ldots, $(X_n,f_n)$ and $(Y_1,g_1)$, \ldots, $(Y_m, g_m)$ for which $(X_i,f_i) \perp (Y_j,g_j)$ for each
$i \leq n$ and $j \leq m$, we have
$$
\prod_{i=1}^n (X_i, f_i) \perp \prod_{j=1}^m (Y_j, g_i)
$$
\end{prop}

\begin{Rk}
In model theoretic stability theory, Proposition~\ref{productortho} is usually deduced as an immediate consequence of
transitivity for the independence relation coming from nonforking.
\end{Rk}

\begin{proof}
Working by induction, one sees that it suffices to show that if $(X,f)$, $(Y,g)$, and $(Z,h)$ are $\sigma$-varieties for which
$(X,f) \perp (Y,g)$ and $(X,f) \perp (Z,h)$, then $(X,f) \perp (Y \times Z, (g,h))$.   Let now $(L,\sigma)$ be some difference field extension of
$(K,\sigma)$ and $U \subseteq (X \times (Y \times Z))_L$ an $(f,g,h)$-skew-invariant variety over $L$.  For any
difference field extension $(M,\sigma)$ of $(L,\sigma)$ and point $a \in (Z,h)^\sharp(M,\sigma)$, the fibre
$U_a$ of $U$ is a $(f,g)$-skew-invariant subvariety of $(X \times Y)_M$.  Since $(X,f) \perp (Y,g)$ we know that
$U_a$ is a finite union of varieties of the form $V(a) \times W(a)$ where $V(a) \subseteq X_M$ is $f$-skew-invariant
and $W(a) \subseteq Y_M$ is $g$-skew-invariant.  Since this is true for every point in $(Z,h)^\sharp$, it
follows from compactness that there are finite sequences of locally closed sets $V_i \subseteq X \times Z$
and $W_i \subseteq Y \times Z$
(for $i \leq n$) so that for any $a \in (Z,h)^\sharp(M,\sigma)$ there is some $J \subseteq \{ 1, \ldots, n \}$ with
$$
U_a = \bigcup_{i \in J} ((V_i)_a \times (W_i)_a)
$$

Taking the sequences to be minimal, we see that each $V_i$ is a component of an $(f,h)$-skew-invariant and each $W_i$ is
a component of an $(g,h)$-skew-invariant.
Hence, by orthogonality, we may write each $V_i$ as a finite union of products of $f$-skew-invariant varieties with $h$-skew-invariant
varieties.  Hence, $U$ itself
is a finite union of products of components of $f$-skew-invariant varieties with components of $(g,h)$-skew-invariant varieties.
\end{proof}

It is difficult to determine whether two given $\sigma$-varieties are orthogonal, though one expects that
``most'' pairs of $\sigma$-varieties are orthogonal.  However, we
exhibit a procedure to determine orthogonality  in the special case of $\sigma$-varieties of the form $(\AA^1,f)$.
On the other hand, there are some easily verified sufficient conditions for orthogonality.
For example, if $f:\PP^1 \to \PP^1$ and $g:\PP^1 \to \PP^1$ are two rational functions and $\deg(f) \neq \deg(g)$, then
$(\PP^1,f) \perp (\PP^1,g)$. (This follows immediately from limit degree computations; see \cite{CH-ACFA1} for details.)  In a different direction, the dichotomy between $\sigma$-varieties
coming from group actions and \emph{disintegrated} $\sigma$-varieties gives a basic instance of
orthogonality.

Two contradictory notions of triviality for $\sigma$-varieties appear in the literature.  Sometimes (see, for example,~\cite{CH-AD1}),
one says that $(X,f)$ is \emph{trivial} if it is isomorphic (as a $\sigma$-variety) to a $\sigma$-variety of the form $(Y,\operatorname{id}_Y)$.
On the other hand, sometimes (see, for instance, the preprint version of this very paper~\cite{selfarXiv}), one says that
$(X,f)$ is trivial if every type in $(X,f)^\sharp$ is trivial in the sense of its forking geometry.  Since this latter property also
goes under the name of \emph{disintegratedness}, we use this term.
Just as orthogonality is usually defined using the theory of forking, so is disintegratedness, but we give a geometric
definition for $\sigma$-varieties.

\begin{definition}
\label{deftrivial}
Let $(X,f)$ be a  $\sigma$-variety over the difference field $(K,\sigma)$.  We say that
$(X,f)$ is \emph{disintegrated} if for each natural number $n \in \NN$ and each algebraically closed difference field
$(L,\sigma)$ extending $(K,\sigma)$, each component of an $f^{\times n}$-skew-invariant subvariety
$Z \subseteq X_L^{\times n}$ is a component of the intersection $\bigcap_{1 \leq i \leq j \leq n } \pi_{i,j}^{-1} \pi_{i,j}(Z)$
where $\pi_{i,j}:X^{\times n} \to X^{\times 2}$ is the projection map $(x_1,\ldots,x_n) \mapsto (x_i,x_j)$.
\end{definition}

With the next proposition we note that for any product of disintegrated $\sigma$-varieties the algebraic relations
are essentially binary.

\begin{prop}
\label{powertriv}
If $(X_1,f_1), \ldots, (X_n,f_n)$ is a finite sequence of disintegrated $\sigma$-varieties over the
difference field $(K,\sigma)$ where each $X_i$ is an absolutely irreducible curve, then
for every difference field $(L,\sigma)$ extending $(K,\sigma)$ every component $Z$ of a skew-invariant
subvariety of $\prod_{i=1}^n (X_i,f_i)$  is a component of $\bigcap_{1 \leq i \leq j \leq n } \pi_{i,j}^{-1} \pi_{i,j}(Z)$
where $\pi_{i,j}:\prod_{i=1}^n X_i \to X_i \times X_j$ is the projection map $(x_1,\ldots,x_n) \mapsto (x_i,x_j)$.
\end{prop}

\begin{proof}
By Proposition~\ref{nonorthoequiv} we may partition
the components of this product so that the factors are non-orthogonal within each block of the partition but are
orthogonal between blocks.
By Proposition~\ref{productortho} we may assume that for every pair
we have $(X_i,f_i) \not \perp (X_j,f_j)$.   In particular, for each $i \leq n$ there is an $(f_i,f_1)$-skew
invariant curve $Y_i \subseteq X_i \times X_1$ none of whose components is a vertical or horizontal line.  Let
$Y := \prod Y_i$ regarded as an $(f_1,\ldots,f_n;f_1^{\times n})$-skew-invariant subvariety of
$\prod_{i=1}^n X_i \times X_1^{\times n}$.   Let $\rho:Y \to \prod_{i=1}^n X_i$
 be the restriction of the projection map onto the
first $n$-coordinates and $\eta:Y \to X_1^{\times n}$ the projection onto the last $n$ coordinates.

Suppose now that $Z \subseteq \prod_{i=1}^n X_i$ is an $(f_1,\ldots,f_n)$-skew-invariant variety and
that $W \subseteq Z$ is an irreducible component. By Lemma~\ref{compper} $Y(Z)$
is a weakly $f_1^{\times n}$-skew-invariant variety and $Z = Z_\per = (Y^{-1} (Y(Z)_\per))_\per$.
Thus, there is a component $V \subseteq (Y(Z))_\per$ with $W \subseteq Y^{-1}(V)$.

Since $(X_1,f_1)$ is
disintegrated,  $V$
is a component of $\bigcap_{1 \leq i \leq j \leq n} \pi_{i,j}^{-1} \pi_{i,j} V$.  Since
$Y$ respects the product decomposition, it follows that $Y^{-1}(V)$ is contained
in $\bigcap_{1 \leq i \leq j \leq n} \pi_{i,j}^{-1} \pi_{i,j} Y^{-1}(V)$.  As $W$ is a component of $Y^{-1}(V)$,
the result follows.
\end{proof}

As we noted above, $\sigma$-varieties coming from algebraic groups are never disintegrated.

\begin{prop}
\label{groupnontrivial}
Let  $(K,\sigma)$ be a difference field, $G$  a connected positive dimensional algebraic group over $K$,
$\phi:G \to G^\sigma$ a dominant map of algebraic groups, and $g \in G(K)$ a $K$-rational point.
Let $\tau_g:G \to G$ be defined by $\tau_g(x) := gx$.  Let $f:G \to G^\sigma$ be given by $f := \phi \circ \tau_g$. Then $(G,f)$ is \emph{not} disintegrated.
\end{prop}
\begin{proof}
Let $(L,\sigma)$ be a difference field extending $(K,\sigma)$ and containing a solution $h$ to the difference
equation $\sigma(h) = \phi(h) \cdot \phi(g)^{-1}$.  One checks immediately that the subvariety $\Gamma$ of
$G_L^{\times 3}$ defined by the equation $z = x \cdot h \cdot y$ is an irreducible, proper closed
$f^{\times 3}$-skew-invariant variety which projects onto $G^{\times 2}$ for each pair of coordinate projections,
witnessing that $(G,f)$ is not disintegrated.
\end{proof}

More
generally, quotients of such $\sigma$-varieties and $\sigma$-varieties coming from actions of algebraic groups are never disintegrated.
In a precise sense, the main theorem of~\cite{CH-ACFA1,CHP} asserts that the presence of a group action is the only
obstruction to disintegratedness.  Specializing to the case of $\sigma$-varieties of the form $(\AA^1,f)$ over a difference field
of characteristic zero, the main theorem of the first author's doctoral dissertation~\cite{Med} characterizes the nondisintegrated
$\sigma$-varieties as exactly those coming from monomials and Chebyshev polynomials.

\begin{definition}
\label{defCheb}
For each positive integer $n \in \ZZ_+$ we write $P_n(x) := x^n \in \ZZ[x]$ for the standard $n^\text{th}$ power monomial.
We define $C_n(x) \in \ZZ[x]$ to be the unique polynomial satisfying the functional equation
$C_n \circ \pi = \pi \circ P_n$ where $\pi:\Gm \to \AA^1$ is given by $x \mapsto x + \frac{1}{x}$.  We call
$C_n$ the \emph{$n^\text{th}$ Chebyshev polynomial}.   By a \emph{negative Chebyshev polynomial} we mean a polynomial of
the form $-C_n$ for some $n \in \ZZ_+$.  In practice, when we speak of a power function, Chebyshev polynomial or
negative Chebyshev polynomial we mean one of degree at least two.
\end{definition}

\begin{Rk}
What we call the $n^\text{th}$ Chebyshev polynomial is sometimes called the $n^\text{th}$ Dickson polynomial.  Moreover,
our normalization differs from that of the Chebyshev polynomials of the first kind, $T_n(x)$, defined by the
relation $T_n(\cos(\theta)) = \cos(n \theta)$, in that $C_n(x) = 2 T_n (\frac{1}{2} x)$.
\end{Rk}

In the following theorem and throughout this paper
we abuse notation by saying that $f$ is a disintegrated polynomial (respectively, rational function) when we mean that $(\AA^1,f)$ (respectively, $({\mathbb P}^1,f)$) is a  disintegrated $\sigma$-variety.

\begin{fact}[Theorem 10 of~\cite{Med}]
\label{medpoly}
Over a difference field of characteristic zero, a polynomial of degree greater than one is disintegrated unless
it is (possibly after base change) skew-conjugate to a Chebyshev polynomial, negative Chebyshev polynomial or a monomial.
\end{fact}

Using Proposition~\ref{productortho}, the observation that polynomials of different degrees are orthogonal, and Fact~\ref{medpoly} we see that for a  $\sigma$-variety $(\AA^n,\Phi)$ where $\Phi:\AA^n \to \AA^n$ takes the form $(x_1,\ldots,x_n) \mapsto (f_1(x_1),\ldots,f_n(x_n))$, we may partition the coordinates so that
$(\AA^n,\Phi)$ is a product of pairwise orthogonal $\sigma$-varieties, each of which has the form
$(\AA^m,\Psi)$ where $\Psi(x_1,\ldots,x_m) = (g_1(x_1),\ldots,g_n(x_n))$ with the $g_i$'s univariate
polynomials for which exactly one of the following occurs:

\begin{itemize}
\item  each $g_i$ is linear,
\item  there is some $N \in \ZZ_+$ so that each $g_i$ is skew-conjugate to $\pm C_N$ or $P_N$, or
\item  $(\AA^m,\Psi)$ is disintegrated and the polynomials $g_i$ are pairwise
nonorthogonal.
\end{itemize}

The skew-invariant varieties for $\sigma$-varieties of the first two kinds are very easy to describe.

After base change, a $\sigma$-variety of the form $(\AA^1,g)$ with $g$ linear is skew-conjugate to
$(\AA^1,\operatorname{id})$.  Clearly, the skew-invariant subvarieties of $(\AA^m,\operatorname{id})$
are precisely those varieties which are defined over the fixed field.  Thus, if $\Phi:\AA^n \to \AA^n$
is any dominant affine map, then, after a base change required to find an isomorphism of
$\sigma$-varieties $\alpha:(\AA^n,\Phi) \to (\AA^n,\operatorname{id})$, the
$\Phi$-skew-invariant varieties are precisely the varieties of the form $\alpha^{-1} Y$ where
$Y \subseteq {\mathbb A}^n$ is a variety defined over the fixed field.   Which of these descend to
$\Phi$-skew-invariant varieties defined over our base field can be an interesting question best
addressed through the Picard-Vessiot theory for difference equations.  We do not pursue the matter here
other than to spell out what happens in the category of algebraic dynamics.

In general, if $K$ is an algebraically closed field, $G$ is an algebraic group over $K$,
$\mu:G \times X \to X$ is a morphism of varieties giving an action of $G$ on $X$, and $g \in G(K)$ is
any $K$-rational point, then we have an algebraic dynamical system $(X,\mu(g,\cdot))$ given by the action of $g$ on $X$.  Let
$H \subseteq G$ be the Zariski closure of the group generated by $g$, which is itself an algebraic group.
Note that a subvariety $Y \subseteq X$ is $\mu(g,\cdot)$-invariant just in case it is $H$-invariant.  Thus,
the $\mu(g,\cdot)$-invariant varieties correspond exactly to the $H$-orbits.   Specializing to the case that
$G$ is the affine group acting on $\AA^n$ and $\mu(g,\cdot)$ is given by a sequence of univariate
linear polynomials, it is easy to see that we may make a change of variables so that each such component has
the form $f_i(x) = \lambda_i \cdot x$ or $f_j(x) = x + 1$.  For the remainder of this calculation, we
 shall assume that the polynomials do have this form.  The Zariski closure $H$ of the group generated by $g$
is then isomorphic to either $\Gm^r$ or $\Gm^r \times \Ga$ where $r$ is the rational rank of the multiplicative
group generated by the scalars $\lambda_i$ and there is a $\Ga$ factor just in case at least one of the
$f_j$ is $x+1$.  For any point $a \in \Gm^n(K)$, the stabilizer  of $a$ in $H$ is trivial.  Hence, as long
as we arrange for $a_i \neq 0$ when $f_i(x) = \lambda_i x$, the dimension of the Zariski closure of the
$\mu(g,\cdot)$-orbit of $a = (a_1,\ldots,a_n)$ is $\dim(H) = r$ or $r+1$.  Let us isolate this observation as a
proposition.

\begin{Def}
\label{independentpoly}
Let $K$ be a field of characteristic zero, $f_1, \ldots, f_n \in K[x]$ a sequence of linear polynomials over
$K$ is \emph{independent} if  either the numbers $f_1'(0), \ldots, f_n'(0)$
are multiplicatively independent or the multiplicative
group generated by $f_1'(0),\ldots,f_n'(0)$ has rank $n-1$ and for some $j \leq n$ we have $f_j(x) = x + b$ with
$b \neq 0$.
\end{Def}

\begin{prop}
\label{lineardense}
Let $K$ be a field of characteristic zero, $f_1, \ldots, f_n \in K[x]$ an independent sequence of linear polynomials over
$K$, and define $\Phi:\AA^n_K \to \AA^n_K$ by $\Phi(x_1,\ldots,x_n) := (f_1(x_1),\ldots,f_n(x_n))$.
Then there is some $a = (a_1,\ldots,a_n) \in \AA^n(K)$ for which $\cO_\Phi(a)$ is Zariski dense.
\end{prop}

In the case of the power functions, for $N > 1$ by a fairly routine argument with degrees, one shows that
any irreducible skew-invariant subvariety of $(\Gm^g,x \mapsto x^N)$ is a translate of an
algebraic group~\cite{Hindry}.  From the point of view of the model theory of difference fields, this result is
a special case of the classification of definable groups~\cite{Chat}.  Since the
map $\pi:(\Gm,P_N) \to (\AA^1,C_N)$ is a dominant map of $\sigma$-varieties, we see that any skew-invariant
subvariety of $(\AA^n,(f_1,\ldots,f_n))$ where each $f_i$ is either $P_N$ or $C_N$ pulls back to a weakly
skew-invariant variety for $(\AA^n,(P_N,\ldots,P_N))$ and thus comes from images of multiplicative
translates of algebraic tori.  In general,
if each $f_i:\AA^1 \to \AA^1$ is merely (after base change) skew-conjugate to $P_N$ or $\pm C_N$, then as
with the linear polynomials, after base change, the skew-invariant varieties are precisely the images under the
isomorphism with the standard polynomials of certain images of  torsion translates of algebraic tori, but the question of which ones
descend to skew-invariant varieties over our base field reduces to problems in difference Galois theory. Since it is
easy to find points in $\Gm^n$ not contained in any proper algebraic subgroups, for example, take $a = (p_1,\ldots,p_n) \in \Gm^n(\QQ)$
where the $p_i$'s are distinct primes, one sees that for dynamical systems given by sequences of power maps and Chebyshev
polynomials, there are rational points with Zariski dense orbits.  Again, let us note this as a proposition.

\begin{prop}
\label{chebpowdense}
Let $K$ be a field of characteristic zero and $f_1, \ldots, f_n \in K[x]$ a sequence of polynomials of degree at least two
such that each $f_i$ is a power function,
a Chebyshev polynomial or a negative Chebyshev polynomial.
Let $\Phi:\AA^n_K \to \AA^n_K$ be
defined by $\Phi(x_1,\ldots,x_n) := (f_1(x_1),\ldots,f_n(x_n))$.  Then there is a point $a \in \AA^n(K)$ for which
$\cO_\Phi(a)$ is Zariski dense.
\end{prop}

Let us collect all of these observations into a single theorem in which we reduce the problem of
describing skew-invariant varieties for $\sigma$-varieties on $\AA^n$ given by sequences of univariate polynomials to
the study of disintegrated polynomials.

\begin{nota}
If we are given a finite sequence of polynomials $f_1, \ldots, f_n$ and a subset $S \subseteq \{1, \ldots, n\}$,
then we write $(\AA^S,f_S)$ for the $\sigma$-variety  $\prod_{i \in S} (\AA^i,f_i)$.
\end{nota}

\begin{thm}
\label{coarsestructurethm}
Suppose that $(K,\sigma)$ is an algebraically closed difference field of characteristic zero and $f_1,\ldots,f_n$ is a
sequence of nonconstant polynomials.  Then there is a partition ${\mathsf P}$ of $\{1 ,\ldots, n \}$ so that for distinct $S$ and $T$ from
${\mathsf P}$,  $(\AA^S,f_S) \perp (\AA^T,f_T)$ implying that if $X \subseteq \AA^n$ is a component of an
$(f_1,\ldots,f_n)$-skew-invariant variety, then it is a product of components of
$f_S$-skew-invariant varieties as $S$ ranges through ${\mathsf P}$ and for each
$S \in {\mathsf P}$ exactly one of the following is true.
\begin{itemize}
\item The polynomial $f_i$ has degree one for each $i \in S$ and the $f_S$-skew-invariant varieties are obtained (after base change)
by pullback from an isomorphism with $(\AA^S,\operatorname{id})$ from the varieties defined over the fixed field,
\item there is a number $N > 1$ so that each $f_i$ is skew-conjugate to $P_N$ or $\pm C_N$ for $i \in S$ and the $f_S$-skew-invariant
varieties are obtained from algebraic tori, or
\item all of the polynomials $f_i$ are pairwise nonorthogonal and disintegrated for $i \in S$ implying that the irreducible $f_S$-skew-invariant varieties are components of
intersections of pullbacks of $(f_i,f_j)$-skew-invariant curves in $\AA^2$ and $(i,j)$ ranges through $S^2$.
\end{itemize}
\end{thm}

\begin{Rk}
A version of Theorem~\ref{coarsestructurethm} holds for rational functions in arbitrary characteristic.
The first case must include purely inseparable maps and the second case must include Latt\`{e}s maps and their
additive analogues in positive characteristic (see~\cite{Med}).
\end{Rk}

\subsection{From curves to polynomials}
We now convert the problem of describing $(f,g)$-skew-invariant curves
to a question about polynomial compositional identities.

\begin{nota}
In what follows we work with an algebraically closed
difference field $(K,\sigma)$ of characteristic zero on which $\sigma$ is an automorphism.  When we speak of a polynomial $f$ we mean a
polynomial with coefficients from $K$.  For the associated $\sigma$-variety, we may write
$(\AA^1,f)$ or in some cases $(\PP^1,f)$.   We will say that a rational function $g:\PP^1 \to \PP^1$ is
a polynomial if $\infty$ is a totally ramified fixed point for $g$.
\end{nota}

\begin{lem}
\label{polycov}
Let $f$ be a disintegrated polynomial.  If $C$ is a smooth, projective, irreducible curve, $(C,h)$ is a
 $\sigma$-variety and $\gamma:(C,h) \to (\PP^1,f)$ is a nonconstant morphism of $\sigma$-varieties,
then $C = \PP^1$ and $h$ and $\gamma$ are polynomials.
\end{lem}
\begin{proof}
The preimage $S := \gamma^{-1} (\{\infty\} )$ of $\infty$ under $\gamma$ is finite,
 and totally invariant for $h$ (that is, $h^{-1}(S) = S$).
 Thus, by an easy Riemann-Hurwitz argument (see Theorem 1.6 of~\cite{Silverman}), we see that $C = \PP^1$ and either $h$ and $\gamma$ are polynomials, or $S$  
 has exactly two elements and $(\PP^1,h)$ is isomorphic to $(\PP^1,x \mapsto \frac{1}{x^{\deg(f)}})$.  However, such a
$\sigma$-variety cannot be disintegrated as the restriction of this map to $\Gm$ is an isogeny.
\end{proof}

It follows from Lemma~\ref{polycov} that all $(f,g)$-invariant curves for $f$ a disintegrated polynomial
come from solutions to polynomial compositional identities.

\begin{prop}
\label{curvetocomposition}
If $f$ and $g$
are disintegrated polynomials and $C \subseteq \AA^2$ is an irreducible
$(f,g)$-skew-invariant curve, then there are
a polynomial $h$ and polynomial morphisms of $\sigma$-varieties $\pi:(\AA^1,h) \to (\AA^1,f)$ and
$\rho:(\AA^1,h) \to (\AA^1,g)$ for which $C$ is parametrized by the map $t \mapsto (\pi(t),\rho(t))$.
That is, there are polynomials $h$, $\rho$ and $\pi$ satisfying the compositional equations
$f \circ \pi = \pi^\sigma \circ h$ and $g \circ \rho = \rho^\sigma \circ h$.

$$
\begin{CD}
\AA^1 @<{\pi}<< \AA^1 @>{\rho}>> \AA^1 \\
@V{f}VV   @VV{h}V   @VV{g}V \\
\AA^1 @<{\pi^\sigma}<< \AA^1 @>{\rho^\sigma}>> \AA^1
\end{CD}
$$
\end{prop}
\begin{proof}

Passing to the closures in $\PP^1 \times \PP^1$, the projective curve $\overline{C}$ is
$(f,g)$-skew-invariant.  Let $\overline{h}$ be the restriction of $(f,g)$ to $\overline{C}$.
Let $\beta:C' \to \overline{C}$ be the normalization map.  Since $C'$ is a smooth curve and $\beta$ is an
 isomorphism off a finite set, there is regular map $h:C' \to C'$ for which
 $\beta:(C',h) \to (\overline{C},\overline{h})$ is a map of $\sigma$-varieties.

Let $\alpha_i:C \to \AA^1$ be the projection map onto the $i^\text{th}$ coordinate for $i = 1$ or $2$.
Since the result is obvious if either projection map is constant, we shall assume that both $\alpha_1$ and
$\alpha_2$ are nonconstant.   By Lemma~\ref{polycov} applied to $\gamma = \alpha_1 \circ \beta$
(or $\gamma = \alpha_2 \circ \beta$), $C' = \PP^1$ and $h$ is a polynomial.
Take $\pi := \alpha_1 \circ \beta$ and $\rho := \alpha_2 \circ \beta$.
\end{proof}

Combining these observations we see that nonorthogonality between disintegrated polynomials is always witnessed
by a solution to a system of polynomial compositional identities.

\begin{corollary}
\label{nonorthotopoly}
Given two disintegrated polynomials $f$ and $g$, then $(\AA^1,f) \not \perp (\AA^1,g)$ if and only if there are a natural number $M$
and nonconstant polynomials
$\pi$, $\rho$ and $h$ for which $f^{\lozenge M} \circ \pi = \pi^{\sigma^M} \circ h$ and $g^{\lozenge M} \circ \rho = \rho^{\sigma^M} \circ h$.

$$
\begin{CD}
\AA^1 @<{\pi}<< \AA^1 @>{\rho}>> \AA^1 \\
@V{f^{\lozenge M}}VV   @VV{h}V   @VV{g^{\lozenge m}}V \\
\AA^1 @<{\pi^{\sigma^M}}<< \AA^1 @>{\rho^{\sigma^M}}>> \AA^1
\end{CD}
$$

\end{corollary}
\begin{proof}
If $(\AA^1,f) \not \perp (\AA^1,g)$, then, possibly after base change, we find an $(f,g)$-skew-invariant curve $C \subseteq \AA^2$ none of whose
components is horizontal or vertical.  Taking $M$ sufficiently divisible, we find a component $C'$ of $C$ which is
$(f^{\lozenge M},g^{\lozenge M})$-skew-invariant (with respect to $\sigma^m$).  The existence of $\pi$, $\rho$ and $h$ now follows from
Proposition~\ref{curvetocomposition}.

In the other direction, the curve $C := (\pi,\rho)(\AA^1) \subseteq \AA^2$ witnesses that $(\AA^1,f^{\lozenge M}) \not \perp (\AA^1,g^{\lozenge M})$
as $\sigma^M$-varieties.   The curve $C' := \bigcup_{j=0}^{M-1}  (f^{\lozenge j},g^{\lozenge j})(C)^{\sigma^{-j}}$  witnesses that
$(\AA^1,f) \not \perp (\AA^1,g)$.
\end{proof}

\subsection{Decompositions and actions}
\label{rittmonoid}

We analyze the identities of Corollary~\ref{nonorthotopoly} through the combinatorics of decompositions of polynomials.

\begin{Def}
A polynomial $f$ is \emph{indecomposable} if $\deg(f) \geq 2$ and it cannot be written as a composition $f = g \circ h$ of two non-linear polynomials $g$ and $h$.

A finite sequence $\vec{f} := (f_k,\ldots,f_1)$ of polynomials $f_i$ is a \emph{decomposition} of a polynomial $f$ if
 $f = f_k \circ \cdots \circ f_1$ and each $f_i$ is indecomposable.
\end{Def}

\begin{Rk}
What we call decompositions are called ``complete decompositions'' in the literature (see, for example,~\cite{MZ}).  Note that in
our convention on the indexing of the factors of a decomposition is decreasing from left to right since composition is performed
from right to left.
\end{Rk}

\begin{Rk}
Induction on degree shows that every non-linear polynomial has a decomposition.  Linear polynomials are
 compositional units.  As such, if $L$ if a linear polynomial, then we write $L^{-1}$ for its
 compositional inverse.  More concretely, if $L(x) = ax + b$, then $L^{-1}(x) = \frac{1}{a} x - \frac{b}{a}$.
\end{Rk}

\begin{Def}
The decompositions $(f_k,\ldots,f_1)$ and $(g_k,\ldots,g_1)$ are \emph{linearly equivalent} if there are linear polynomials $L_{k-1},\ldots,L_1$ for which
$g_k = f_k \circ L_{k-1}$, $g_i = L_i^{-1} \circ f_i \circ L_{i-1}$ for $k > i > 1$,
 and $g_1 = L_1^{-1} \circ f_1$.

Polynomials $a$ and $b$ are \emph{linearly related} if there are linear $L$ and $M$ such that $L \circ a \circ M = b$.
\end{Def}

If $\vec{f}$ and $\vec{g}$ are linearly equivalent, then they are decompositions of the same polynomial. Linear equivalence, as the name suggests, is an equivalence relation. Corresponding factors of linearly equivalent decompositions are linearly related.

\begin{Def}
 The linear-equivalence class of a decomposition $\vec{f}$ is denoted by $[\vec{f}]$.
 For a polynomial $f$, $\linequi_f$ is the set of linear-equivalence classes of decompositions of $f$.
\end{Def}

Not all decompositions of a polynomial are linearly equivalent; for example, $(x^2, x^3+x)$ and $(x^3+2x^2+x, x^2)$ are both decompositions of $( x \cdot (x^2 +1))^2$.  Ritt's theorem~\cite{Ritt} gives a precise sense in which all decompositions of a polynomial may be
 obtained from one given decomposition.

\begin{Def} \label{specialpoly}
A \emph{Ritt polynomial} is an indecomposable polynomial of one of the following kinds:
\begin{itemize}
\item Monomial: $P_p(x) := x^p$, $p$ a prime
\item Chebyshev: $C_p(x)$, $p$ an odd prime
\item $x^k \cdot u(x^\ell)^n$ where $k \neq 0$, $\gcd(k,\ell) = 1$, $\gcd(k,n) = 1$, $u(0) \neq 0$, $u$ is monic non-constant, and at least one of $\ell$ and $n$ is greater than one.
\end{itemize}

The following identities involving Ritt polynomials are the \emph{basic Ritt identities.}
\begin{itemize}
\item $P_p \circ P_q = P_q \circ P_p$ for prime $p \neq q$
\item $C_p \circ C_q = C_q \circ C_p$ for odd prime $p \neq q$
\item $P_p \circ ( x^k \cdot u(x^{\ell p})^n ) = ( x^k \cdot u(x^\ell)^{pn}) \circ P_p$ for prime $p$
\end{itemize}
\end{Def}

\begin{Rk} These notions are closely related but not identical to ``Ritt moves'' and ``Ritt neighbors'' in~\cite{MZ}.
\end{Rk}

\begin{Def}
If $\vec{g}$ and $\vec{f}$ are two decompositions of the same polynomial, we say that $\vec{g}$ is \emph{obtained from $\vec{f}$ by a Ritt swap at $i$} if there are linear polynomials $L$, $M$, and $N$ such that
$$g_i := S \circ N^{-1} \mbox{ and } g_{i+1} = L \circ R \mbox{ and } g_j := f_j \mbox{ for } j \neq i, i+1$$
and
$(L^{-1} \circ f_{i+1} \circ M) \circ (M^{-1} \circ f_i \circ N) = R \circ S$ is a basic Ritt identity.

An indecomposable polynomial $f$ is \emph{swappable} if it is linearly related to a Ritt polynomial.
\end{Def}

\begin{Rk}
The compositional identity $C_2 \circ C_p = C_p \circ C_2$ is not a basic Ritt identity, but $(C_p, C_2)$ can be obtained by a Ritt swap at $1$ from $(C_2, C_p)$ as follows. As $C_p$ is an odd polynomial, it is of the form $x \cdot u(x^2)$, and $C_2(x) = x^2 -2 = L \circ P_2$ where $L(x) = x-2$. Now taking $M = N = \id$ makes
$(L^{-1} \circ C_2 \circ M) \circ (M^{-1} \circ C_p \circ N)$ look like the left side of a basic Ritt identity. This is pursued in great detail in Section~\ref{section41}.
\end{Rk}

\begin{Rk}
\label{crucrk}
While it may be possible to obtain many different decompositions from the same $\vec{f}$ by a Ritt swap at the same $i$ by choosing different linear witnesses $L$, $M$, and $N$, we show (see page \pageref{proofofcrucial}) that all decompositions so obtained are linearly equivalent.  This
 invariance result is also proved in \cite{MZ}, and is already implicit in Ritt's work. \end{Rk}

\begin{Rk}
\label{costumes}
The term ``swap'' should suggest that when a decomposition is obtained from another via a Ritt swap, then the factors involved swap
places.  However, a Ritt swap arising from a basic Ritt identity of the third kind is not really a swap,
in that one of the factors, linearly related to $x^k \cdot u(x^p)$, not only switches places with the monomial,
but also ``becomes'' a different polynomial, linearly related to $x^k \cdot u(x)^p$.
\end{Rk}

\begin{Rk}
\label{monicu}
 We depart from~\cite{Ritt} in requiring Ritt polynomials to be monic.
 An easy computation verifies that this has no effect on the meaning of ``Ritt swap''
 and the truth of Ritt's Theorem below.\end{Rk}

\begin{fact}[Ritt, \cite{Ritt}]
\label{Rittthm}
Over $\CC$, any two decompositions of the same polynomial have the same number of factors.  Indeed, if $\vec{f}$ and $\vec{g}$ are decompositions of the same polynomial, then $\vec{g}$ is linearly equivalent to a decomposition obtained from $\vec{f}$ by a finite sequence of Ritt swaps.
\end{fact}

 Ritt's Theorem may be stated loosely as ``decompositions of polynomials are unique up to permutations'',
  and indeed it is tempting to look for an action by the symmetric group, identifying the Ritt swap at $i$ with
   the transposition $\tau_i := (i \,\,\, i+1) \in \operatorname{Sym}_k$ in the symmetric group on $k$ elements.

 Since often nothing can be obtained from $\vec{f}$ by a Ritt swap at $i$, for example,
 when one of the factors $f_i$ and $f_{i+1}$ is not swappable,  at best this is a partial action.
 In light of Remark~\ref{crucrk}, Ritt swaps can only act on decompositions up
  to linear equivalence, that is on $\linequi_f$.   The next two results show that this action is well-defined when defined.

 \begin{lemma} \label{EasyLinEquivRittSwapLemma}
 If $\vec{f}$, $\vec{g}$, and $\vec{h}$ are decompositions of the same polynomial,
 $\vec{g}$ is obtained from $\vec{f}$ by a Ritt swap at $i$,
 and $\vec{h}$ is linearly equivalent to $\vec{f}$,
then there is a decomposition obtained from $\vec{h}$ by a Ritt swap at $i$ and linearly related to $\vec{g}$.
\end{lemma}
\begin{proof}
Let $R_{k-1},\ldots,R_1$, $L$, $M$, and $N$ be linear polynomials
witnessing our hypotheses.
That is, the $R$s witness that $\vec{h}$ is linearly related to $\vec{f}$:\\
$h_k = f_k \circ R_{k-1}$, $h_j = R_j^{-1} \circ f_j \circ
R_{j-1}$ for $1 < j < k$, $h_1 = R_1^{-1} \circ f_1$ and the other
linear polynomials witness the Ritt swap: $(L^{-1} \circ f_{i+1}
\circ M) \circ (M^{-1} \circ f_i \circ N) = T \circ S$ is a basic
Ritt identity, $g_i := S \circ N^{-1}$, $g_{i+1} = L \circ T$, and
$g_j := f_j$ for the other $j \leq k$.  To simplify the notation,
we define $R_k(x) = R_0(x) = x$.

Define $\widetilde{L} := R_{i+1}^{-1} \circ L$,
$\widetilde{M} := R_i^{-1} \circ M$, and
$\widetilde{N} := R_{i-1} \circ  N$.
It is now routine to check that this
choice of $\widetilde{L}$, $\widetilde{M}$, and $\widetilde{N}$
witnesses that $\vec{h}$ admits a
Ritt swap at $i$ and that the resulting decomposition is linearly
equivalent to $\vec{g}$.
\end{proof}

With the following theorem, whose proof is delayed to page~\pageref{proofofcrucial}, we show that the action
of Ritt swaps on linear equivalence classes of decompositions is well-defined. Stronger versions of this result are obtained in \cite{MZ} and \cite{Ritt}.

\begin{forefthm}
\label{cruciallemma} If two decompositions $\vec{h}$ and $\vec{g}$ are both obtained from $\vec{f}$ by a Ritt swap at $i$, then $\vec{h}$ is linearly equivalent to $\vec{g}$.
\end{forefthm}

 In the symmetric group, the adjacent transpositions $\tau_i$ have order $2$ and satisfy the braid relations
 $\tau_i \tau_{i+1} \tau_i =  \tau_{i+1} \tau_i  \tau_{i+1}$ for all $i$ and $\tau_i \tau_j = \tau_j \tau_i$ for $j \neq i \pm 1$.
 While Ritt swaps do satisfy the braid relations, they do not quite have order two.  We formalize this
 symmetric group-like action via Ritt swaps as an action of a certain monoid.

\begin{Def}
Let $\rittmonoid_k$ be the free monoid on the $(k-1)$ generators $t_1, \ldots, t_{k-1}$. \label{defpermrep}
 The \emph{permutation represented by} a word $t_{a_r} \ldots t_{a_2} t_{a_1}$ in $\rittmonoid_k$ is the product
 $\tau_{a_r} \ldots \tau_{a_2} \tau_{a_1} \in \operatorname{Sym}_k$.

 The action $\star$ of $\rittmonoid_k$ on $\linequi_f^\ast := \linequi_f \cup \{ \infty \}$ is defined by \begin{itemize}
 \item $t_i \star [\vec{f}]$ is the linear equivalence class of a decomposition obtained from $\vec{f}$ by a Ritt swap at $i$, if one exists;
 \item  otherwise, $t_i \star [\vec{f}] := \infty$;
 \item $t_i \star \infty = \infty$ for all $i$. \end{itemize}

For $w \in \rittmonoid_k$ and $[\vect{f}] \in \linequi_f$ we say that $w \star
[\vect{f}]$ is \emph{defined} if $w \star [\vec{f}] \neq \infty$.

We often abuse notation writing $w \star \vect{f} = \vect{g}$ for
$w \star [\vect{f}] = [\vect{g}]$.
\end{Def}

With the following theorem, whose proof is completed on page~\pageref{proofoffunda},
 we show that Ritt swaps satisfy the braid relations, and that $t_i$ has order $2$ except when $t_i \star w = \infty$. The first two parts are immediate given Theorem~\ref{cruciallemma}, but the last is not so easy.

\begin{forefthm}
\label{fundamentallemma} For any $[\vect{f}] \in \linequi_f$ and $i < k$
\begin{itemize}
\item If $t_i \star [\vect{f}]$ is defined, then $t_i^2 \star [\vect{f}] = [\vect{f}]$.
\item For $j \neq i \pm 1$, $t_it_j \star [\vect{f}] = t_j t_i \star [\vect{f}]$.  In particular, one is defined if and only if the other is.
\item $t_i t_{i+1} t_i \star [\vect{f}] = t_{i+1} t_i t_{i+1} \star [\vect{f}]$.  In particular, one is defined if and only if the other is.
\end{itemize}
\end{forefthm}

With these identities, a purely combinatorial analysis yields (see Section~\ref{canformsec}) normal forms for words in the Ritt monoid, roughly corresponding to insert-sort and to merge-sort. That is, for each $w \in \rittmonoid_k$, we find another word $w'$ of a special form, representing the same permutation and such that $w' \star \vec{f}$ is defined and equal to $w \star \vec{f}$ whenever $w \star \vec{f}$ is defined. For example, $w'$ is the empty word when $w = t_i t_i$. This implies that if two words $w$ and $w'$ represent the same permutation and both $w \star [\vec{f}]$ and $w' \star [\vec{f}]$ are defined, then $w \star [\vec{f}] = w' \star [\vec{f}]$ (see Corollary~\ref{nearaction}), and also provides an invaluable explicit computational tool for the rest of the paper.

\begin{Rk}
A stronger version of Corollary~\ref{nearaction}, that the polynomial $f$ and the sequence of degrees of the factors $f_i$ already determine the linear equivalence class of the decomposition $\vec{f}$ is shown in~\cite{MZ}.  It seems that our stronger Theorem~\ref{fundamentallemma}
is not a simple consequence of the work in~\cite{MZ}. The canonical forms in the present paper are substantially different from
those in~\cite{MZ}, and are better suited to our purposes.
\end{Rk}

\subsection{Skew-twists} \label{sktwsec}
	Recall that the purpose of studying decompositions is to characterize polynomial identities in  Proposition~\ref{curvetocomposition}:
 \begin{equation}\begin{CD} \label{sketwdiag}
\AA^1 @>{f}>> \AA^1 \\
@A{\pi}AA @AA{\pi^\sigma}A\\
\AA^1 @>{g}>> \AA^1 \\
\end{CD}\end{equation}
for disintegrated polynomials $f$ and $g$. We eventually show that all such identities come from those where $\pi$ is
indecomposable or linear. For indecomposable $\pi$, these identities are either rare exceptions characterized in Proposition~\ref{notskewtwist}, or \emph{single skew-twists} where $\pi$ is an initial compositional factor of $f$, and also a terminal compositional factor of $g$ twisted by $\sigma$.
The rest of this section is devoted to the study of sequences of single skew-twists.

\begin{Def} \label{single-twist-def}
 The decomposition $(f_1^\sigma, f_k, \ldots, f_2)$ is called \emph{the single-skew-twist} of the decomposition $\vec{f} := (f_k, \ldots, f_2, f_1)$ and denoted $\phi \star \vec{f}$. (Here, $\phi$ stands for ``forward''.)

 If $\vec{f}$ is a decomposition of a polynomial $f$, then $\phi \star \vec{f}$ is a decomposition of a (probably different) polynomial $h$; we call $h$ \emph{a single-skew-twist} of $f$.

 For polynomials $f$ and $g$, the relation ``$f$ is a \emph{skew-twist} of $g$'' is the symmetric-transitive closure of the relation ``$f$ is a single-skew-twist of $g$''. That is, $f$ is a skew-twist of $g$ if there are $f = f_0, f_1, \ldots, f_n = g$ such that each $f_i$ is a single-skew-twists of $f_{i+1}$, or vice versa.

 To undo what $\phi$ does, we define $\beta \star \vec{f} := (f_{k-1}, \ldots, f_1, f_k^{(\sigma^{-1})})$. (Here, $\beta$ stands for ``back''.)

 When $n < k$ and $\vec{g} = \phi^n \star \vec{f}$, we call $g$ a \emph{plain skew-twist of $f$}.
\end{Def}

\begin{Rk}
A polynomial may have several single-skew-twists, coming from different decompositions.
In composing a correspondence from $f$ to $g$ with one from $g$ to $h$, both of which come from single skew-twists,
the decompositions of $g$ used to represent the skew-twists may differ.  Thus, to describe correspondences obtained from sequences of skew-twists we need to keep track of decompositions of intermediate polynomials.
\end{Rk}

\begin{Def} \label{correncdef}
For a given positive integer $k \in \ZZ_+$, the \emph{skew-twist monoid}, $\skewmonoid_k$, is the free monoid generated by
the symbols $\phi$, $\beta$, $t_1, \ldots, t_{k-1}$.  If $\vec{f} = (f_k,\ldots,f_1)$ is a decomposition of a polynomial $f$
and $w := w_n \ldots w_2 w_1 \in \skewmonoid_k$ where each $w_i$ is a generator, then
 a sequence of decompositions $\vec{f} = \vec{f}^0, \vec{f}^1, \ldots, \vec{f}^n$ is \emph{a witnessing sequence for $w \star \vec{f}$} if for each $j$,
 \begin{itemize}
 \item if $w_j = t_i$, then $[\vec{f}^{j+1}] = w_j \star [\vec{f}^{j}]$; and
 \item if $w_j$ is $\phi$ or $\beta$, then $\vec{f}^{j+1} = w_j \star \vec{f}^{j}$ in the sense of Definition~\ref{single-twist-def}.
 \end{itemize}

\emph{The correspondence $\kriva$ encoded by this witnessing sequence} is the composite of the curves
$\krivb_n \circ \cdots \circ \krivb_1$ where \begin{itemize}
\item if $w_j =t_i$ for some $i$, then $\krivb_j = \Delta_{\AA^1}$ is the graph of the identity map on $\AA^1$,
\item if $w_j = \phi$, then $\krivb_j$ is the graph of $f^j_1$, and
\item if $w_j = \beta$, then $\krivb_j$ is the converse relation of the graph of $f^{j+1}_1$.
\end{itemize}

We also say that $\kriva$ is \emph{a correspondence encoded by $w \star \vec{f}$}.
\end{Def}

While the witnessing sequence uniquely determines the correspondence, $w$ and $\vec{f}$ do not uniquely determine the witnessing sequence because Ritt swaps are only defined up to linear equivalence, and even linearly equivalent decompositions may produce different single-skew-twists. We define skew-linear-equivalence and then formalize an action of the skew-twists monoid.

\begin{Rk} \label{concat-compose}
 Suppose that $w = vu \in \skewmonoid_k$ and $w \star \vec{f} = \vec{h}$ is defined. Let $\{ \vec{f}^j \}$ be a witnessing sequence for this, let $\krivd$ be the correspondence from $f$ to $h$ encoded thereby, and let $\vec{g}$ be the element of this sequence coming from $u \star \vec{f}$.
 Then $\krivd = \krivb \circ \kriva$ where $\kriva$ and $\krivb$ are the curves encoded by the two witnessing sequences $(\vec{f}, \ldots, \vec{g})$ and $(\vec{g}, \ldots, \vec{h})$, respectively.
\end{Rk}

\begin{Def}
 Two decompositions $\vec{f}$ and $\vec{h}$ are \emph{skew-linearly-equivalent} if there is a linear $L$ such that $\vec{h}$ is linearly equivalent to $( L^\sigma \circ f_k, f_{k-1}, \ldots, f_2, f_1 \circ L^{-1})$.
\end{Def}

 \begin{Rk}
 Skew-linear-equivalence is an equivalence relation. Skew-linearly-equivalent decompositions
  may be decompositions of different, but always skew-conjugate, polynomials. Indeed,
$$
 \vec{f} \mapsto \vec{g} := ( L^\sigma \circ f_k, f_{k-1}, \ldots, f_2, f_1 \circ L^{-1})
$$
is a bijection between decompositions of $f$ and decompositions
 of $g:= L^\sigma \circ f \circ L^{-1}$, and this bijection respects linear equivalence.\end{Rk}

\begin{Def} \label{twist-monoid-def}
Let $\skequi_f$ be the set of skew-linear-equivalence classes of decompositions of skew-twists of $f$. We write $[[\vec{f}]]$ for the skew-linear equivalence class of $\vec{f}$.

The action $\star$ of $\skewmonoid_k$ on $\skequi_f^\ast := \skequi_f \cup \{ \infty \}$ is given by
 \begin{itemize}
 \item $t_i$  still acts by the Ritt swap at $i$ as in Definition~\ref{defpermrep};
 \item $\phi \star [[f_k, \ldots, f_1]] := [[f_1^\sigma, f_k, \ldots, f_2]]$ and $\phi \star \infty = \infty$;
 \item $\beta \star [[\vec{f}]] := [[f_{k-1}, \ldots, f_1, f_k^{(\sigma^{-1})}]]$ and $\beta \star \infty = \infty$.
 \end{itemize}

For $w \in \skewmonoid_k$ and $[[\vect{f}]] \in \skequi_f$ we say that $w \star [\vect{f}]$ is \emph{defined} if $w \star [\vec{f}] \neq \infty$.
\end{Def}

\begin{lemma} \label{skewcruciallemma}
\label{newhapppycorrs} \label{welldefcorr}
\begin{enumerate}
\item  Ritt swaps are well-defined up to skew-linear-equivalence.
\item  Single skew-twists are well-defined up to skew-linear-equivalence.
\item Suppose that $w \in \skewmonoid_k$ and $\vec{f}$ is a decomposition of a polynomial $f$,
$\vec{g}$ and $\vec{h}$ are witnessing sequences for $w \star \vec{f}$ with corresponding
encoded correspondences $\kriva$ and $\krivb$ between $(\AA^1,f)$ and $(\AA^1,g)$ (and $(\AA^1,f)$
and $(\AA^1,h)$, respectively).  Then there is a linear $L$
with $h = L^\sigma \circ g \circ L^{-1}$ and $\krivb = L \circ \kriva$. That is,
$h$ is skew-conjugate to $g$ and the correspondence is off by the same linear factor.
\end{enumerate}
\end{lemma}

\begin{proof}
 The proof of the first two parts serves as the induction (on the length of $w$) step for the proof of the last part.

 For the first part, note that if a decomposition $(h_k, h_{k-1}, \ldots, h_2, h_1)$ of $f$
 is obtained from $\vec{f}$ by a Ritt swap at $i$,
 then the decomposition $(L^\sigma \circ h_k, h_{k-1}, \ldots, h_2, h_1 \circ L^{-1})$
 of $g := L^\sigma \circ f \circ L^{-1}$  is obtained from
 $\vec{g}:= ( L^\sigma \circ f_k, f_{k-1}, \ldots, f_2, f_1 \circ L^{-1})$ by a Ritt swap at $i$.

 For the second, take $\vec{g} := ( L^{\sigma} \circ f_k \circ L_{k-1}^{-1}, L_{k-1} \circ f_{k-1} \circ L_{k-2}^{-1}, \ldots,  L_2^{-1} \circ f_2 \circ L_1, L_1^{-1} \circ f_1 L^{-1})$ skew-linearly equivalent to $\vec{f}$. The same linear factors, reindexed, witness that single skew-twists of $\vec{g}$ are skew-linearly equivalent to the corresponding single skew-twists of $\vec{f}$.

 For the third part, let $n$ be the length of $w$ and take witnessing sequences $\vec{g}^j$ and $\vec{h}^j$ for $w \star \vec{f} = \vec{g}$ and for $w \star \vec{f} = \vec{h}$; so $\vec{g}^0 = \vec{f} = \vec{h}^0$ and $\vec{g}^n = \vec{g}$ and $\vec{h}^n = \vec{h}$. We induct on $n$, strengthening the induction hypothesis from $h = L^\sigma \circ g \circ L^{-1}$ to $\vec{h} = L^\sigma \circ \vec{g} \circ L^{-1}$.

 Let $v := w_{n-1} \ldots w_1$ and let $\kriva_0$ and $\krivb_0$ be the curves encoded by $v \star \vec{f}$ via these witnessing sequences. By induction hypothesis and the first two parts, there is a linear factor $L$ such that $\vec{h}^{n-1}$ is linearly equivalent to $L^{\sigma} \circ \vec{g} \circ L^{-1}$ and $\krivb = L \circ \kriva$. If $w_n$ is a Ritt swap, the same $L$ works: look at the proof of the first part of this lemma to prove the first part of the induction hypothesis, and note that the curve encoded is the same for $v$ and $w$ to prove the second part of the induction hypothesis. If $w_n$ is a single skew-twist, composing the graph of the first or last factor of $\vec{h}$ with $L \circ \kriva$ cancels $L$, and introduces a new linear factor, one of the witnesses of the linear equivalence of $\vec{h}^{n-1}$ and $L^{\sigma} \circ \vec{g} \circ L^{-1}$.
\end{proof}

\begin{Rk} \label{reintroduce-linear}
The definition of witnessing sequences and encoded correspondences allows linear equivalence and skew-conjugacy in some cases but not in others. Because of this inconsistency, it is safest to artificially reintroduce the linear factor $L$ at the end, as we do in Theorem~\ref{realchar}, Theorem~\ref{hhprop}, and Theorem~\ref{centralthm}.\end{Rk}

\begin{corollary}
\label{coruniqueencode}
If two correspondences $\kriva$ and $\krivb$ between the polynomials $f$ and $g$ are both encoded by $w \in \skewmonoid_k$, then they are off by a (skew)-symmetry $L$ of $g$, that is, $\krivb = L \circ \kriva$ and $L^{\sigma} \circ g \circ L^{-1} = g$.
\end{corollary}

Our characterization of correspondences encoded by words in $\skewmonoid_k$ comes from the canonical form for such words, obtained in Proposition~\ref{border_guard_prop} and Lemmata~\ref{goodplaintwist} through~\ref{w1w2lemma}. Here we state an imprecise nontechnical version as motivation.

\begin{Rk}
\label{newsecondslogan} This is a motivational imprecise nontechnical version of Proposition~\ref{border_guard_prop} and Lemmata~\ref{goodplaintwist} through~\ref{w1w2lemma}.

Any word $ w \in \skewmonoid_k$ such that $w \star \vec{f}$ is defined is \emph{equivalent} to a word of the form $\phi^{Nk} w_0$ or $\beta^{Nk} w_0$ where the length of $w_0$ is bounded
by a constant depending only on the degree of $f$.

Any $(f,g)$-skew-invariant curve coming from skew-twists is a composition of the graph of $f^{\lozenge N}$ for some $N \in \mathbb{N}$ with a correspondence both of whose degrees are bounded by $2 \cdot \deg{f}$; or a composition of a correspondence both of whose degrees are bounded by $2 \cdot \deg{f}$ with the converse relation to the graph of $g^{\lozenge N}$ for some $N \in \mathbb{N}$.
\end{Rk}

We define equivalence for words in the skew-twist monoid so as to make the second part of Remark~\ref{newsecondslogan}
a consequence of the first. Thus, it must take into account the curves encoded by the words in the monoid, but need
not keep track of their strictly skew-pre-periodic components.

\begin{Def} \label{equicorrdef}
Given $v, w \in \skewmonoid_k$  and a decomposition $\vec{f} = (f_k,\ldots,f_1)$.
We say that $v$ and $w$ are \emph{equivalent with respect to $\vec{f}$} and write $v \approx_{\vec{f}} w$ if
$v \star [[\vec{f}]] = w \star [[\vec{f}]]$ and there are  witnessing sequences
$(\vec{g}^j)$ and $(\vec{h}^j)$ for $v \star \vec{f}$ and $w \star \vec{f}$, respectively so
that the final $\vec{g}^n$ and $\vec{h}^n$ are decompositions of the same polynomial $g$, and
$(\kriva_v)_\per =(\kriva_{w})_\per$ for the curves $\kriva_v$ and $\kriva_w$ encoded by $v$ (respectively, $w$) via  $(\vec{g}^j)$
 (respectively, $(\vec{h}^j)$).

 When $v \approx_{\vec{f}} w$ for all $\vec{f}$, we write $v \approx w$ and say that the two words are \emph{equivalent}.
\end{Def}

This notion is weaker than the purely syntactic one in Definition~\ref{defapporx} of $v \simeq w$ for $v, w \in \rittmonoid_k$.

\begin{lemma}\label{biggercorr}
\begin{enumerate}
\item $\phi \beta \approx \id \approx \beta \phi$
\item Suppose $u_1 \approx_{\vec{f}} v_1$,
 and so let $\vec{g} := u_1 \star \vec{f} = v_1 \star \vec{f}$,
 and suppose $u_2 \approx_{\vec{g}} v_2$;
 then $u_2 u_1 \approx_{\vec{f}} v_2 v_1$.
\item For any word $w$ in $\skewmonoid_k$,  $w \phi^k \approx \phi^k w$ and
$w \beta^k \approx \beta^k w$.

\item  $t_i \phi \approx \phi t_{i+1}$ for $i < k-1$ while  $t_i \beta \approx \beta t_{i-1}$ for $i > 1$
\end{enumerate}
\end{lemma}

\begin{proof}
\begin{enumerate}
\item   The $(f,f)$-skew-invariant correspondence $\kriva_{\beta\phi}$ encoded by $\beta \phi$ is defined by $f_1(x) = f_1(y)$.
The diagonal is one of its irreducible components, is $(f,f)$-skew-invariant and is equal to the image $(f, f) ( \kriva_{\beta \phi})$
of the whole curve. Thus,  $\beta\phi  \approx_{\vec{f}} \id$ for any decomposition $\vec{f}$.
\item  This is an immediate consequence of Lemma~\ref{compper}, Lemma~\ref{welldefcorr} and
Corollary~\ref{coruniqueencode}.
\item Since $\phi^k \star \vec{f} = \vec{f}^\sigma$ and $\beta^k \star \vec{f} = \vec{f}^{(\sigma^{-1})}$, it is clear that $\phi^k$ and $\beta^k$ commute with Ritt swaps. Part (1) ensures that they commute with $\phi$ and $\beta$.
\item  After a shift, the same two factors participate in the Ritt swap on the two sides of each equation.
\end{enumerate}

\end{proof}

\begin{lemma}
\label{stdformBk}
For all $w \in \skewmonoid_k$, there is some $u \in \skewmonoid_k$ that does not contain $\beta$ nor $\phi^k$ as a substring, and such that $w \approx  \phi^{mk} u$  or $w \approx  \beta^{nk} u$.
\end{lemma}

\begin{proof}
 We may introduce extra $\beta^i \phi^i$ pairs into the word $w$. We introduce enough of them to obtain $w' \approx w$ so that $\beta$ only occurs in multiples of $k$ in $w'$.
 Then we pull all $\beta^k$ to the left, and obtain $\beta^{Nk} w'' \approx w'$ where $w''$ contains no instances of $\beta$.
 Then we can also pull all $\phi^k$ to the left and obtain
 $\beta^{Nk} \phi^{Mk} u \approx \beta^{Nk} w''$ where $u$ contains no instances of $\beta$, and no instances of $\phi^r$ for $r \geq k$. Then we cancel $\beta \phi$ pairs in the beginning.
\end{proof}

\begin{Rk}
 Here is the geometry behind this bit of combinatorics. When $w \star \vec{f}$ is defined, the correspondence $\kriva$ encoded is (an irreducible component of) the fiber product of a diagram
 $$ (\AA^1, f)  \leftrightarrow \ldots \leftrightarrow (\AA^1, g)$$
 where each arrow corresponds to an occurrence of $\phi$ or $\beta$ in $w$. What we just proved is that, for correspondences coming from skew-twists, we may instead look at irreducible components of the fiber product of the diagram
\begin{equation*} \label{neweqcorr}
 (\AA^1, f) \xleftarrow{F} (\AA^1, g^{\sigma^N}) \xrightarrow{g^{\lozenge N}} (\AA^1, g)
\end{equation*}

or

\begin{equation*}
(\AA^1,f) \xleftarrow{F} (\AA^1,g^{\sigma^N}) \xleftarrow{g^{\lozenge N}} (\AA^1,g)
\end{equation*}

where we know one arrow, $g^{\lozenge N}$, exactly, and the other arrow is a sequence of plain skew twists.
\end{Rk}

In most cases, it is also possible to bring together all the $\phi$s in $u$ in Lemma~\ref{stdformBk}, and then $F$ must be skew-compositional power of $g$ composed with (a not necessarily indecomposable) factor of $g$.   However, it is not always possible to do this.  Consider the following
example.

$$
\phi t_1 \phi \star (x \cdot (x^5 +1), x^5) = \phi t_1 \star (x^5, x \cdot (x^5 +1)) = \phi \star (x \cdot (x +1)^5, x^5) = (x^5, x \cdot (x +1)^5)
$$

The encoded correspondence, defined by $y = x^{25}$, is not a compositional power of $x^5 \cdot (x^5)^4$ in any sense. The trouble is that
Lemma~\ref{biggercorr} does not give a way to simplify $t_{k-1} \phi$ and $\beta t_{k-1}$.  We deal with this issue by introducing what we call the \emph{border guard monoid} whose action on decompositions leaves the leftmost factor fixed, though possibly altering
it via Ritt swaps in the sense of Remark~\ref{costumes}.

\begin{Def}
For a fixed positive integer $k \geq 2$, $\borgarmonoid_k$ is the free monoid on the symbols $\psi$, $\gamma$ and
$t_1, \ldots, t_{k-2}$ (where there are no generators of the form $t_i$ if $k=2$)
and $\psi$ and $\gamma$.  Regard $\borgarmonoid_k$ as a submonoid of $\skewmonoid_k$
by mapping $t_i$ to $t_i$, $\psi$ to $(t_{k-1} \phi)$ and $\gamma$ to $(\beta t_{k-1})$.
\end{Def}

 The action of $\borgarmonoid_k$ on $\skequi_f$ is the restriction of the action of $\skewmonoid_k$.  More concretely,
$$
\psi \star (f_k, \ldots, f_1) = t_{k-1} \phi \star (f_k, \ldots, f_1) = t_{k-1} \star (f_1^\sigma, f_k, \ldots, f_2) = (\widehat{f_k}, \widehat{f_1^\sigma}, \ldots, f_2)
$$
$$
\gamma \star (f_k, f_{k-1}, \ldots, f_1) = \beta t_{k-1} \star (f_k, \ldots, f_1) = \beta \star (\widehat{f_{k-1}}, \widehat{f_k}, \ldots, f_1) = (\widehat{f_k}, \ldots, f_1, \widehat{f_{k-1}}^{\sigma^{-1}})
$$

We use $\borgarmonoid_k$ to establish the bounds in Remark~\ref{newsecondslogan}.  Indeed, finding the word $w'$ of the following proposition
goes a long way towards producing the short word $w_0$ of Remark~\ref{newsecondslogan}.

\begin{prop}
\label{border_guard_prop}
Any word $w$ in $\skewmonoid_k$ is equivalent to $\phi^N w'$ or to $\beta^N w'$ for some $N \in \mathbb{N}$ and some word $w' \in \borgarmonoid_k$.
\end{prop}

\begin{proof}
We take $w \in \skewmonoid_k$, start from the right, and move to the left. At every step, we have a word $w_{\text{bad}} \beta^{a} \phi^b w_{\text{good}}$ with $w_{\text{bad}} \in \skewmonoid_k$ and $w_{\text{good}} \in \borgarmonoid_k$. Working by induction on
  the length of $w_{\text{bad}}$, Thus, it is clearly sufficient to prove that if $s$ is a generator of $\skewmonoid_k$, then there are
  natural numbers $a'$ and $b'$ and some $u \in \borgarmonoid_k$ with $s \phi^a \beta^b = \phi^{a'} \beta^{b'} u$.   If $s = \beta$, then
  we may take $a' := a + 1$, $b' = b$ and $u$ the empty word.  If $s = \phi$ and $a = 0$, then we take $a' = 0$, $b' = b+1$ and $u$ the empty
  word while if $a > 0$, then we take $a' = a -1$, $b' = b$ and $u$ the empty word.

  We work by induction on $(a+b)$ for the case that
  $s = t_i$ for some $ i < k$.  In the base case of $a = b = 0$, if $i < k-1$, then we may take $a' = b' = 0$ and $u = t_i$. For $i = k-1$, we
  note $t_{k-1} \approx \phi \beta t_{k-1} = \phi \gamma$, so that we may take $a'=0$, $b' =1$, and $u = \gamma$.
 If $a=0$ and $i \neq k-1$, then $t_i \phi \approx \phi t_{i+1}$ and we can apply the inductive hypothesis to $t_{i+1} \phi^{b-1}$.
 If $a=0$ and $i = k-1$, then $b \neq 0$. If $b=1$, then we are looking at $(t_{k-1} \phi)$, so we  let $a'=b'=0$ and $u = \psi$. If $b \geq 2$, note that $t_{k-1} \phi^2 \approx \phi^{2} t_1$ so we can apply the inductive hypothesis to $t_1 \phi^{b-2}$.
 If $a \neq 0$ and $i \neq 1$, then $t_i \beta \approx \beta t_{i-1}$ and we can apply the inductive hypothesis to $t_{i-1} \beta^{a-1} \phi^b$.
 If $a \neq 0$ and $i=1$, note that $t_1 \beta \approx \beta^{2} t_{k-1} \phi$. If $a=1$, then we get $t_1 \beta \phi^b \approx \beta^2 t_{k-1} \phi^{b+1}$ and we can apply the second inductive step to $t_{k-1} \phi^{b+1}$. If $a \geq 2$, we get $t_1 \beta^a \phi^b \approx \beta^2 t_{k-1} \phi \beta^{a-1} \phi^{b} \approx \beta^2 t_{k-1} \beta^{a-2} \phi^b$, and we can apply the inductive hypothesis to $t_{k-1} \beta^{a-2} \phi^b$.
\end{proof}

\begin{Rk} \label{circledance}
It is sometimes helpful to think of $\skequi_f$ as a bunch of indecomposable factors arranged in a circle, rather than a line, with $f_k$ standing next to $f_1$. In that spirit, both $\gamma$ and $\psi$ act by a Ritt swap between these two. To be more precise, for any decomposition $\vec{f}$, the following three are equivalent:\begin{itemize}
\item $\gamma \star \vec{f}$ is defined
\item $\psi \star \vec{f}$ is defined
\item $t_1 \star (f_1^\sigma, f_k)$ is defined
\end{itemize}
\end{Rk}

\begin{corollary} \label{unswcrack}
If $w \in \skewmonoid_k$ and $w \star \vec{f}$ is defined and $f_k$ is not swappable, then the correspondence encoded by $w \star \vec{f}$ is already encoded by $\phi^N \star \vec{g}$ or $\beta^N \star \vec{g}$ for some $N \in \mathbb{N}$ and some decomposition $\vec{g}$ of $f$.
 \end{corollary}

\begin{proof}
 Get the $\phi^N w' \approx w$ or $\beta^N w' \approx w$ from Proposition~\ref{border_guard_prop}, with $w' \in \borgarmonoid_k$. Because $f_k$ is not swappable, $w' \star \vec{f}$ is only defined if $w' \in \rittmonoid_k$. Let $\vec{g} := w' \star \vec{f}$. Then the correspondence encoded by $w \star \vec{f}$ is the same as the one encoded by $\phi^N \star \vec{g}$, or $\beta^N \star \vec{g}$, as the case may be.\end{proof}

Of course, the hypothesis that $f_k$ is the special unswappable factor is purely artificial.

\begin{corollary} \label{unswcrack2}
If $w \in \skewmonoid_k$ and $w \star \vec{f}$ is defined and $f_i$ is not swappable for some $i$, then the correspondence encoded by $w \star \vec{f}$ is already encoded by $\phi^N u \phi^i \star \vec{f}$ or $\beta^N u \phi^i \star \vec{f}$ for some $N \in \mathbb{N}$ and some $u \in \rittmonoid_k$.
 \end{corollary}

\begin{proof}
Recall that $w \beta^i \phi^i \approx w$.
Since the $k$th factor $f_i^\sigma$ of $\vec{h} := \phi^i \star \vec{f}$ is unswappable, Corollary~\ref{unswcrack} applies to $(w \beta^i) \star \vec{h}$, with the sequence $u$ of Ritt swaps giving the potentially necessary new decomposition $\vec{g}$ in the statement of that Corollary. \end{proof}

\begin{Rk} \label{crackrk} The hypothesis that $f_k$ is not swappable is unnecessarily strong.  Requiring merely
that $t_1 \star (g_1^\sigma, g_k)$ is not defined where
$\vec{g} := u \star \vec{f}$ for some $u \in \rittmonoid$ would suffice.  Many explicit examples satisfying this requirement appear in a previous draft of this paper~\cite{selfarXiv} related to the concept of a  ``crack''. \end{Rk}

\section*{Outline of the technical Sections~\ref{section41} --~\ref{puttingtogether}}

The next four Sections~\ref{section41} --~\ref{puttingtogether} constitute technical proofs of the results described in Section~\ref{coarsestructure}. Three of the four sections are devoted to refinements of Ritt's Fact~\ref{Rittthm}, and the last one uses these refinements to obtain the desired characterization of skew-invariant curves.

The characterization of linear relatedness between Ritt polynomials in our Section~\ref{section41} is also carried out in \cite{MZ}, and is implicit in \cite{Ritt}. We include our analysis because we use many of the intermediate results in the two following Sections~\ref{clustersec} and~\ref{canformsec}.

In Section~\ref{clustersec}, we describe a nearly unique way to write a polynomial as a composition of \emph{clusters}. One of our two kinds of clusters is the same as one of the two kinds of blocks in \cite{MZ}, but our \textsf{C}-free clusters are nothing like their monomial blocks. Again, similar technical issues come up for us and for them, such as the fact that no more than one quadratic factor may cross a boundary between clusters in the same direction. Our first use of these clusters is to prove our fundamental Theorem~\ref{fundamentallemma} for the Ritt monoid action, that $t_i t_{i+1} t_i \star \vec{f}$ is defined if and only of $t_{i+1} t_i t_{i+1} \star \vec{f}$ is defined. While it follows immediately from \cite{MZ} that the two are equal when defined, it is not clear to us whether our stronger result follows from their work.

In Section~\ref{canformsec}, we use the fact that the action of the Ritt monoid on linear-equivalence classes of decompositions factors through the ``braid monoid'' to find canonical forms for sequences of Ritt swaps, roughly corresponding to insert-sort and to merge-sort.
 To the best of our understanding, our results on canonical forms do not follow easily from~\cite{MZ}, where different canonical forms are used to obtain tighter bounds on the number of Ritt swaps necessary to obtain one decomposition from another.  Applying the second canonical form to a clustering produces particularly strong results.
We end that section with a characterization (see Proposition~\ref{notskewtwist}) of those rare polynomial identities $\pi^\sigma \circ f = g \circ \pi$ which have nothing to do with skew-twists. A slight weakening of it follows immediately from~\cite{MZ}, and the full version can be deduced with a little more work.

Section~\ref{puttingtogether} combines all of our technical tools and finally characterizes skew-invariant curves.
In Section~\ref{more-gen-sec}, we introduce more generators into our monoids in order to encode correspondences coming from Proposition~\ref{notskewtwist} rather than from skew-twists. Within this formalism, we describe precisely how the correspondences arising from Theorem~\ref{notskewtwist} interact (commute) with those arising from skew-twists.
In Section~\ref{skew-clust-sec}, we then combine our work on clusterings with our understanding of skew-twists in order to obtain a characterization of correspondences encoded by $w \star \vec{f}$ for $w \in \skewmonoid_k$ for those rare $\vec{f}$ that are not subject to Corollary~\ref{unswcrack2}.
In Section~\ref{triumph-sec}, Theorem~\ref{realchar} is a complete, precise, and technical characterization of $(f,g)$-invariant curves for disintegrated polynomials $f$ and $g$. The technical conclusion of Theorem~\ref{realchar} becomes much more readable in the special case of $(h,h)$-invariant curves. It is stated in Theorem~\ref{hhprop}, and then used to obtain a more readable but less tight characterization for the general case in Theorem~\ref{centralthm}.

\begin{nota}
Throughout the next four technical Sections~\ref{section41} --~\ref{puttingtogether}, we work over a fixed difference-closed field of characteristic zero with automorphism $\sigma$.
We reserve the symbol ``$x$'' for the variable in the
polynomial ring.
When we speak of a polynomial, linear polynomial, scalar, \emph{et cetera}, we mean a polynomial over this field,
linear polynomial over this field, element of this field, \emph{et cetera}.
Occasionally, and especially towards the end, we explicitly note how our results specialize to the category of algebraic dynamical systems defined over the fixed field of $\sigma$.
\end{nota}

\section{Linear relations between Ritt polynomials}
\label{section41} \label{newsectionlinrel}

In this section we identify the possible linear relations between Ritt polynomials and identify certain
classes of Ritt polynomials admitting extra linear relations.  Using these results on linear relations
we complete the proof of Theorem~\ref{cruciallemma} showing that the action of a Ritt swap at $i$ is well-defined
on the linear equivalence classes of decompositions of a polynomial.  Much of the basic work on
linear relations appears also in~\cite{MZ} (see Lemmas 3.20 -- 3.22) and is implicit in~\cite{Ritt}.

\subsection{Definitions and examples}
\begin{Def}
A \emph{scaling} is a linear polynomial of the form $(\cdot \lambda) := \lambda x$ for some
nonzero scalar $\lambda$.
A \emph{translation} is a linear polynomial of the form $(+A) := x + A$ for some scalar $A$.
Two linearly related polynomials $f$ and $g$ are \emph{translation related} (respectively, \emph{scaling related}) if
 $g = L \circ f \circ m$ for some translations (respectively, scalings) $L$ and $M$.
\end{Def}

\begin{Rk}
The group of automorphisms of $\AA^1_K$ may be identified with the semidirect product of the group of
translations by the group of scalings.
\end{Rk}

\begin{Def}
Given a polynomial $f$ and a nonzero scalar $\lambda$
 we define $\lambda \ast f := (\cdot \lambda^{-\deg(f)}) \circ f \circ (\cdot \lambda)$.
\end{Def}

\begin{Rk}
\label{scaleast}
If $f$ is monic, then so is $\lambda \ast f$.  On the other hand, if $f$ and $g$ are monic polynomials and
$(\cdot \mu) \circ f \circ (\cdot \lambda) = g$, then $\mu = \lambda^{-\deg(f)}$.  That is, $g = \lambda \ast f$.
\end{Rk}

\begin{Rk}
\label{astRitt}
For any $n \in \NN$ and scalar $\lambda$, we have $\lambda \ast P_n = P_n$.  More generally, if $f = x^k \cdot U(x^\ell)$ for some polynomial $U$ and $\lambda$ a scalar, then
$\lambda \ast f = x^k \cdot (\lambda^\ell \ast U)(x^\ell)$. Thus if $f$ is a Ritt polynomial, then so is $\lambda \ast f$ for any nonzero $\lambda$.  In particular, if
$\ell$ is maximal for which $f$ takes this form, then $\lambda \ast f = f$ if and only
if $\lambda$ is an $\ell^\text{th}$ root of unity.
\end{Rk}

The above observations imply that to describe all instances of linear relatedness between Ritt polynomials, it suffices to separately describe those witnessed by translations and those witnessed by scalings.

\begin{lem}
\label{rittscaltran}
If $f$ and $g$ are linearly related Ritt polynomials, then there is a third Ritt polynomial $h$ which is
translation related to $f$ and scaling related to $g$.
\end{lem}
\begin{proof}
Let $L$ and $M$ be linear polynomials with $L \circ f \circ M = g$.   Write $L = (\cdot \lambda) \circ (+B)$ and $M = (+A) \circ (\cdot \mu)$
for appropriate scalars $A$, $B$, $\lambda$ and $\mu$.  Set $h := (+B) \circ f \circ (+A)$.  Since translations preserve the highest degree term, $h$
is still monic and translation related to $f$.   As $h  = (\cdot \lambda^{-1}) \circ g (\cdot \mu^{-1})$ and both $g$ and $h$ are monic, we conclude
by Remark~\ref{scaleast} that $h = \mu^{-1} \ast g$.  From Remark~\ref{astRitt} we see that $h$ is a Ritt polynomial.
\end{proof}

By similar reasoning, the class of basic Ritt identities other than $C_p \circ C_q = C_q \circ C_p$ is closed under scalings.

\begin{prop}
\label{scaledbasic}
If $b \circ a = d \circ c$ is a basic Ritt identity, at least one of $a$ or $b$ is not a Chebyshev polynomial, and $\lambda$ and $\mu$
are nonzero scalars, then there are scalars $\eta$ and $\nu$ for which $(\mu \ast b) \circ (\lambda \ast a) = (\eta \ast d) \circ (\nu \ast c)$
is a basic Ritt identity.
\end{prop}
\begin{proof}
At least one of $a$ or $b$ must be a monomial $P_p$ for some prime $p$.  If they are both monomials, then the result is immediate as
$\lambda \ast P_p = P_p$.    Suppose now that $a = P_p$ and $b$ takes the form $x^k \cdot u(x^\ell)^n$ for some monic $u$ with nonzero constant
term.  Then $d = P_p$ and $c = x^k \cdot u(x^{p \ell})^{\frac{n}{p}}$.  We saw in Remark~\ref{astRitt} that $\mu \ast b = x^k \cdot (\mu^\ell \ast u)(x^\ell)^n$ and $\lambda \ast a = a$.  Thus, $(\mu \ast b) \circ (\lambda \ast a) =  (1 \ast d) \circ (\sqrt[p]{\mu} \ast c)$.  Likewise, if
$b = P_p$, we may take $\eta = \lambda^p$ and $\nu = 1$.
\end{proof}

Since Chebyshev polynomials of odd degree are odd functions, every Ritt polynomial is of the form $x^k \cdot u(x^\ell)^n$ with $k \ell n > 1$, and therefore is involved in a nontrivial scaling relation to a Ritt polynomial.
  We focus on translation relations amongst Ritt polynomials which appear in only two special classes, what we call
types \textsf{A} (for ``adaptable'') and \textsf{C} (for ``Chebyshev-like'').

\begin{Def}
\label{Jdef}
A type \textsf{A} Ritt polynomial is a Ritt polynomial of the form $f(x) = x^\ell \cdot (x - A)^m u(x)^n$ where $u$ is a monic polynomial with nonzero constant term, $A$ is some nonzero scalar and both $\gcd(\ell,n) > 1$ and $\gcd(m,n) > 1$.
A type \textsf{A} swappable polynomial is a polynomial which is linearly related to a type \textsf{A} Ritt polynomial.
\end{Def}

\begin{Rk}
Since a Ritt polynomial must be indecomposable, in Definition~\ref{Jdef} we must have $\gcd(\ell,m,n) = 1$.
\end{Rk}

\begin{Rk}
Lemma~\ref{rittscaltran}, the observation that for any $\lambda$ and $f$, either both $f$ and $\lambda \ast f$ are type \textsf{A} Ritt polynomials, or neither one is, and Theorem~\ref{transrelate} together imply that a Ritt polynomial which happens to be a type \textsf{A} swappable polynomial is, in fact, a type \textsf{A} Ritt polynomial.
\end{Rk}

\begin{Def}
A type \textsf{C} swappable polynomial is a polynomial of odd prime degree which is linearly related to a Chebyshev polynomial.
\end{Def}

\begin{Def} \label{see-hat-def}
For a natural number $n$ and scalar $\lambda$ we define $C_{n,\lambda} := \lambda \ast C_n$ and
$\widehat{C}_{n,\lambda} := \lambda \ast ( (+2) \circ C_n \circ (-2) )$. For odd prime $n$, these are the \emph{type \textsf{C} Ritt polynomials}.
\end{Def}

It follows from Remark~\ref{astRitt} that $C_{n,\lambda}$ is a Ritt polynomial for odd prime $n$ and non-zero $\lambda$. For odd $n$, we show (Proposition~\ref{Cprop}) that both $\widehat{C}_{n,1} = (+2) \circ C_n \circ (-2)$ and
$\widehat{C}_{n,-1} =  (-2) \circ C_n \circ (+2)$ are of the form $x \cdot u(x)^2$ as a consequence
of the fact $C_n$ commutes with $C_2(x) = x^2 -2$. It then follows from Remark~\ref{astRitt} that $\widehat{C}_{n,\lambda}$ are Ritt polynomials for all odd prime $n$ and nonzero $\lambda$. It follows from Theorem~\ref{rittscaltran} that these are the only Ritt polynomials amongst type \textsf{C} swappable polynomials.

\begin{prop}
\label{Cprop}
For every odd prime $p$ and scalar $\lambda$, the
polynomial $\widehat{C}_{p,\lambda}$ is a  Ritt polynomials of the form $x \cdot u(x)^2$. Moreover,
for any number $n$, we have $\widehat{C}_{n,-1} =  (-4) \circ \widehat{C}_{n,1} \circ (+4)$.
\end{prop}

\begin{proof}
For odd $n$, we show that both $\widehat{C}_{n,1} = (+2) \circ C_n \circ (-2)$ and
$\widehat{C}_{n,-1} =  (-2) \circ C_n \circ (+2)$ are of the form $x \cdot u(x)^2$ as a consequence
of the fact $C_n$ commutes with $C_2(x) = x^2 -2$.
 For the first observation, we compute:
$$C_n \circ C_2 = C_2 \circ C_n$$
$$C_n \circ (-2) \circ P_2 = (-2) \circ P_2 \circ C_n$$
$$(+2) \circ C_n \circ (-2) \circ P_2 =  P_2 \circ C_n$$

Thus, since $\widehat{C}_n = (+2) \circ C_n \circ (-2)$ appears
in a basic Ritt identity with $P_2$,
it must be of the form $x \cdot u(x)^2$ for some polynomial $u$.

By Remark~\ref{astRitt}, it follows that the same holds of
$\widehat{C}_{n, \lambda}$ for  all nonzero $\lambda$.

For the second, first observe that
$$i \ast C_2 = \frac{1}{i^2} ((ix)^2 - 2) = -(-x^2-2) = x^2+2$$
Now $C_n \circ C_2 = C_n \circ ( \cdot -1)\circ ( \cdot -1) \circ
C_2 \circ (\cdot i) \circ (\cdot -i) = $
$$ = (\cdot -1) \circ C_n \circ (x^2+2) \circ (\cdot -i) = C_2
\circ C_n$$
 Bringing all outside linear factors to the right
and introducing $(-2)$ on the left,
$$ (-2) \circ C_n \circ (+2) \circ P_2 = (-2) \circ (\cdot -1) \circ C_2
\circ C_n \circ (\cdot i)$$
 Now, $[(-2) \circ (\cdot -1) \circ C_2] (x)
= -(x^2-2) -2 = -x^2 = [P_2 \circ (\cdot \pm i)] (x)$, so
$$ (-2) \circ C_n \circ (+2) \circ P_2 = P_2 \circ (\cdot \pm i) \circ C_n \circ (\cdot i)
 = P_2 \circ (i \ast C_n)$$
\end{proof}

Although $C_2$ is not a Ritt polynomial, how it might be linearly related to itself or to the monomial $P_2$ is important in Section~\ref{clustersec}
and is summarized with the following remark.

\begin{Rk} \label{linrelquadrk}
Since the only way $P_2$ is linearly related to itself is by scalings $\lambda \ast P_2 = P_2$, the only way $C_2(x) = x^2 -2$ is linearly related to itself is
by $A_\lambda \circ C_2 \circ (\cdot \lambda) = C_2$ for
$A_\lambda(x) := \frac{1}{\lambda^2} x + \frac{2}{\lambda^2} - 2$. Note the immediate consequence
that if $L \circ P_2 \circ M = C_2$, then $M = \cdot \lambda$ is a scaling,
and $L = B_\lambda$ where
$B_\lambda(x) := \frac{1}{\lambda^2} x -2 = (-2) \circ (\cdot \frac{1}{\lambda^2}) (x)$.
Note that $A_\lambda(x)$ is never a scaling unless $\lambda = \pm 1$
and $A_\lambda = \id$, and $B_\lambda$ is never a scaling.
\end{Rk}

\subsection{Characterization of translation related Ritt polynomials} \label{sectio32}
In the next theorem, whose proof occupies the rest of this section \ref{sectio32}, we collect all instances of
linear relatedness amongst Ritt polynomials via translations.  Using Lemma~\ref{rittscaltran},
a general description follows.

\begin{theorem}
\label{transrelate}
If $f$ and $g$ are Ritt polynomials and $A$ and $B$ are scalars, not both zero, for which
$(+B) \circ f \circ (+A) = g$, then either
\begin{itemize}
\item $B = 0$, $f$ and $g$ are type \textsf{A} Ritt polynomials, or
\item $B \neq 0$, $f$ and $g$ are type \textsf{C} Ritt polynomials.
\end{itemize}
In fact, if $B \neq 0$, then either $f = C_{p,\lambda}$ and $g = \widehat{C}_{p,\lambda}$ where
$\lambda = \frac{-2}{A} = \sqrt[p]{\frac{2}{B}}$ and $p$ is an odd prime  or
$f = \widehat{C}_{p,\mu}$ and $g = \widehat{C}_{p,-\mu}$ where $\mu = \frac{4}{A} = \sqrt[p]{\frac{-4}{B}}$ and
$p$ is an odd prime.
\end{theorem}

We turn to the task of proving Theorem~\ref{transrelate} reformulating its statement as the solution of the following problem.

\begin{problem}
\label{problemtrans}
For which Ritt polynomials $f$ and $g$ and scalars $A$ and $B$ can we have
$$
(+B) \circ f \circ (+A) = g \text{ ?}
$$
\end{problem}

In the solution of Problem~\ref{problemtrans} and in the course of the analysis of the monoid actions introduced in
Section~\ref{rittmonoid}, we make use of some refined degrees of Ritt polynomials.

\begin{Def} \label{inoutdegdef}
If $f$ is any polynomial which is not a monomial, then $f$ may be expressed as $x^k \cdot u(x^\ell)^n$ where
$u$ is a polynomial with a nonzero constant term and $n$ and $\ell$ are maximal.  The number $k$ is the order of
vanishing of $f$ at $0$.   The number $n$, which we call the \emph{out-degree} of $f$, is the greatest common divisor of the
orders of vanishing of $f$ at points other than $0$.  The number $\ell$, which we call the \emph{in-degree} of $f$, is the
size of the multiplicative stabilizer of the set of roots of $f$.
\end{Def}

\begin{Rk}
Of course, it is true that a monomial may be expressed in the above form, taking $u = 1$, but then no maximal $n$ nor $\ell$ would exist.
If $f$ is a non-monomial Ritt polynomial, then either its in-degree or its out-degree must be at least two.
\end{Rk}

\begin{Rk} \label{scalrel-inout}
By considering type \textsf{A} Ritt polynomials, one sees that even for Ritt polynomials, the out-degree and in-degree are \emph{not} invariants of the linear relatedness class of a polynomial. However, two scaling related Ritt polynomials $f$ and $\lambda \ast f$ clearly have the same in-degrees and out-degrees.

\end{Rk}

\begin{lemma} \label{seehatinout}
All $C_{p, \lambda}$ have in-degree $2$ and out-degree $1$. All $\widehat{C}_{p, \lambda}$ have in-degree $1$ and out-degree $2$. \end{lemma}

\begin{proof}
Since $C_p$ is an odd function, its in-degree is divisible by $2$. From the computations in the proof of Proposition~\ref{Cprop}, it follows that the out-degree of $\widehat{C}_p$ is divisible by $2$. The rest of the result for $C_p$ and $\widehat{C}_p$ follows by Proposition~\ref{ellis1} and Lemma~\ref{finalreduction}, and Remark~\ref{scalrel-inout} finishes the proof.\end{proof}

Returning to Problem~\ref{problemtrans} we observe that $A = B = 0$ and $f = g$ always gives a trivial solution.  On the other hand, evaluating both sides at $0$ we see that there are no solutions with $A = 0 \neq B$.  Thus, we may and do assume that $A \neq 0$ examining the cases where $B = 0$ and where $B \neq 0$ separately. We have already found some solutions of these problems: type \textsf{A} Ritt polynomials for the case when $B=0$, and type \textsf{C} Ritt polynomials for the case $B \neq 0$. Our task is to prove that there are no others. Using an appropriate scaling, we reduce to the case that $A = 1$.

\begin{lem}
If $A$, $B$, $f$, and $g$ give a solution to Problem~\ref{problemtrans}, then $1$, $\frac{B}{A^{\deg(f)}}$,
$A \ast f$, $A \ast g$ is also a solution
to Problem~\ref{problemtrans}.
\end{lem}
\begin{proof}
$(+\frac{B}{A^{\deg(f)}}) \circ (A \ast f) \circ (+1) =  A^{-\deg(f)} f(A (x+1)) + A^{-\deg(f)} B = A^{-\deg(f)} ( f(Ax + A) + B) =
A \ast ( (+B) \circ f \circ (+A) ) = A \ast g$
\end{proof}

\begin{reduction} \label{reductoa1}
For the remainder of this section, we assume that $A = 1$.  Thus, we seek solutions to
$$
(+B) \circ f \circ (+1) = g
$$
where $f$ and $g$ are Ritt polynomials.  By way of notation, we write $f = f_1 = x^{k_1} u_1(x^{\ell_1})^{n_1}$ and
$g = f_2 = x^{k_2} u_2(x^{\ell_2})^{n_2}$ where $\ell_i$ is the in-degree of $f_i$ and $n_i$ is the out-degree of $f_i$.
We write $s_i := \deg(u_i)$ and $t_i$ for the number of zeros of $u_i$, not counted with multiplicity.
\end{reduction}

Let us record a simple ramification calculation.

\begin{lem}
\label{ramificationcalc}
Let $k$, $\ell$, $n$ be natural numbers with $\gcd(k,\ell) = \gcd(k,n) = 1$ and $u$ a polynomial with $u(0) \neq 0$.
Set $f := x^k \cdot u(x^\ell)^n$.   Let $t$ be the number of zeros of $u$ \emph{not} counted with multiplicity and let $s := \deg(u)$.
Then the following holds.
\begin{itemize}
\item The number of points (counted with multiplicity) at which both $f$ and $f'$ vanish, that is, the number of
ramification points above zero, is $(k-1) + \ell(ns - t)$.
\item The number of points at which $f'$ vanishes but $f$ does not, that is, the number of ramification points lying above
points other than zero, is $\ell t$.  Moreover, this set of points is closed under multiplication by the group of $\ell^\text{th}$ roots
of unity.
\end{itemize}
\end{lem}
\begin{proof}
This is a straightforward computation which we include for completeness.

$$
f'(x) = k x^{k-1} u(x^\ell)^n + x^k n u(x^\ell)^{n-1} u'(x^\ell) \ell x^{\ell - 1} =
 x^{k-1} u(x^\ell)^{n-1} (k u(x^\ell) + \ell n u'(x^\ell) x^{\ell - 1})
$$

Since $u(0) \neq 0$, we see that $\ord_0 f' =  k-1$.    On the other hand, if $u(a^\ell) = 0$ and $f'(a) = 0$, then we must have
$u'(a^\ell) = 0$, and $\ord_a f' = (n-1) \ord_{a^\ell} u + \ord_{a^\ell} u' = n \ord_{a^\ell}(u) - 1$. Summing over the distinct
roots of $u$, we finish the calculation of the total ramification over zero.    If we let $\tilde{u} := \gcd(u,u')$, by which we mean
the monic polynomial which generates the ideal
generated by $u$ and $u'$, then
the other zeros of $f'$ come from the zeros of $k \frac{u}{\tilde{u}}(x^\ell) + \ell n \frac{u'}{\tilde{u}}(x^\ell) x^{\ell - 1}$ which has degree exactly $\ell t$.
\end{proof}

Differentiating the equation $(+B) \circ f_1 \circ (+1) = f_2$, we see that $f_1' \circ (+1) = f_2'$.   Hence, for any point $a$ we have
$\ord_a f_2' = \ord_{a+1} f_1'$.  That is, $(+1)$ translates the zeros of $f_2'$ to the zeros of $f_1'$ respecting multiplicities.  If $B = 0$,
then the ramification above zero is matched.  If $B \neq 0$, then there is one nonzero point for which the ramification of $f_2$ above zero is
matched with the ramification of $f_1$ above that point and vice versa.  It is this consequence which makes these seemingly trivial observations
useful.

\begin{prop} \label{ellis1}
In the notation from Reduction \ref{reductoa1}, either $\ell_1 = 1$ or $\ell_2 = 1$.
\end{prop}
\begin{proof}
If $\ell_1 > 1$, then the sum of the roots of $f_1$ is zero as is the sum of the roots of $f'_1$. Indeed,
zero contributes nothing to the sum.  The other roots both of $f_1$ and of $f_1'$ are partitioned into cosets of the $\ell_1^\text{th}$
roots of unity over which the sum is zero.  Because $f_2' = f_1' \circ (+1)$, we see that the sum of the roots of
$f_2'$ is $(1 - \deg(f_1)) \neq 0$ (as $\deg(f_1) \geq 3$).
\end{proof}

\begin{reduction}
For the remainder of this section, we take $\ell_1 = 1$.
\end{reduction}

\begin{lem}
If $B = 0$, then $\ell_2 = 1$.
\end{lem}
\begin{proof}
As $f_2(x+1) = f_1(x)$, we see that $k_1 = \ord_0 f_1 = \ord_{-1} f_2$.  That is, $-1$ is a $k_1$-fold zero of
$u_2(x^{\ell_2})^{n_2}$.  We thus have $\ord_{-\zeta} f_2 = k_1$ for any other $\ell_2^\text{th}$ root of unity $\zeta$.  Unless,
$\ell_2 = 1$, we can choose $\zeta$ so that $-\zeta + 1 \neq 0$, but then $k_1 = \ord_{-\zeta + 1} f_1 = \ord_{-\zeta + 1} u_1(x)^{n_1}$, so $n_1$ divides $k_1$. If $n_1 > 1$, this contradicts the indecomposability of $f_1$. Otherwise, $n_1 = \ell_1 = 1$, so $f_1$ is not a Ritt polynomials, again a contradiction. 
\end{proof}

We first complete the solution for the case where $B = 0$.

\begin{prop}
In Problem~\ref{problemtrans}, if $B = 0$ and $\ell_1 = \ell_2 = 1$, then there are positive integers $m_1$, $m_2$ and a monic
polynomial $U$ for which $u_1(x) = (x-1)^{m_1} U(x)^{n_2}$ and $u_2(x) = (x+1)^{m_2} U(x+1)^{n_1}$.  In particular,
$f_1$ and $f_2$ are type \textsf{A} Ritt polynomials.
\end{prop}
\begin{proof}
As $k_2 = \ord_0 f_2 = \ord_1 f_1$, we see that $k_2 \mid n_1$.  Set $m_1 := \frac{k_2}{n_1}$.  Observe that $\ord_1 u_1 = m_1$.  Likewise,
since $k_1 = \ord_0 f_1 = \ord_{-1} f_2$, $n_2$ divides $k_1$.  Write $m_2 := \frac{k_1}{n_2}$ and observe that $\ord_{-1} u_2 = m_2$.
Express $u_1(x) = (x-1)^{m_1} V_1(x)$ and $u_2(x) = (x+1)^{m_2} V_2(x)$.    Specializing Problem~\ref{problemtrans}, we have the
following equation.
$$
(x+1)^{m_2 n_2} \cdot [x^{m_1} V_1(x+1)]^{n_1} = x^{m_1 n_1} \cdot [ (x+1)^{m_2} V_2(x) ]^{n_2}
$$
Canceling $(x+1)^{m_2 n_2} x^{m_1 n_1}$ we obtain $V_1(x+1)^{n_1} = V_2(x)^{n_2}$.  Recalling that $n_2 m_2 = k_1$ and $n_1$ are
relatively prime, so that $\gcd(n_1,n_2) = 1$, it must be that $V_1$ is an $n_2^\text{th}$ power and $V_2$ an $n_1^\text{th}$ power.
Write $V_1 = U_1^{n_2}$ and $V_2 = U_2^{n_1}$.  As $f_1$ is monic, we may take each of $U_1$ and $U_2$ to be monic.
As $U_1(x+1)^{n_1 n_2} = U_2(x)^{n_1 n_2}$, we have $U_1(x+1) = U_2(x)$, as required.
\end{proof}

\begin{reduction}
In what follows, we assume that $B \neq 0$.
\end{reduction}

\begin{lem}
\label{finalreduction}
Given our reductions, $k_1 = k_2 = 1$, all roots of $u_1$ and $u_2$ are simple, and $n_1 = n_1 \ell_1 = n_2 \ell_2 = 2$.
\end{lem}

\begin{proof}
Concretely, we are considering the equation $f_1(x+1) + B = f_2(x)$.
Since $\gcd(k_2,\ell_2) = 1$, if $f_2(a) \neq 0$ and $\zeta \neq 1$ is an $\ell_2^\text{th}$ root of unity, then
$f_2(\zeta a) \neq f_2(a)$.   Thus, in each of the cosets of the $\ell_2^\text{th}$ roots of unity contained in the
critical points of $f_2$ there can be at most one point which maps to $B$ under $f_2$.  As translation by $1$ takes the
critical points of $f_2$ over $B$ to the critical points of $f_1$ over $0$, we conclude from Lemma~\ref{ramificationcalc}
(taking into account that $\ell_1 = 1$) that
$$(k_1 - 1) + (n_1 s_1 - t_1) \leq t_2 \text{ .}$$
On the other hand, since translation by $1$ induces a (multiplicity preserving) bijection between the critical points of
$f_2$ with those of $f_1$, we see that the other critical points of $f_2$ must be mapped to critical points of $f_1$ not above $0$.
From Lemma~\ref{ramificationcalc} again we see that
$$
(k_2 - 1) + \ell_2 (n_2 s_2 - t_2) + (\ell_2 - 1)t_2 \leq t_1 \text{ .}
$$

Combining these two inequalities we obtain

$$
(k_1 - 1) + (k_2 - 1) + n_1 s_1 + \ell_2 n_2 s_2 \leq 2 t_1 + 2 t_2 \text{ .}
$$

Bearing in mind that $t_i \leq s_i$, $2 \leq n_1 \ell_1 = n_1$, $2 \leq n_2 \ell_2$, and $1 \leq k_i$ we see that
all of these inequalities must be equalities.
\end{proof}

Thus, we are left with describing those solutions where $\ell_2 = 2$ and $n_2 = 1$ and where $\ell_2 = 1$ and $n_2 = 2$. We already have examples of these in Definition~\ref{see-hat-def} and Proposition~\ref{Cprop}; the next two propositions say that there are no others.

\begin{prop}
\label{uniqueCprop}
For each positive integer $s$, there is a unique monic polynomial $u$ for which there is some nonzero scalar $B$ and
polynomial $v$ satisfying
\begin{equation}
\label{Weq}
(+B) \circ (x \cdot u(x)^2) \circ (+1) = (x \cdot v(x)^2) \text{ .}
\end{equation}
\end{prop}

\begin{proof}
The polynomials $u$ and $v$ have only simple roots by Lemma~\ref{finalreduction}.  Since
$u$ is monic, we may assume that $v$ is monic as well.

Differentiating we obtain

\begin{equation}
u(x+1) ( u(x+1) + 2 (x+1) u'(x+1) ) = v(x) (v(x) + 2 x v'(x))
\end{equation}

Since $B \neq 0$, it follows that $u(x+1)$ and $v(x)$ are coprime.
Hence, $ u(x+1) + 2 (x+1) u'(x+1)$ is a scalar multiple of $v(x)$ and
$v(x) + 2x v'(x)$ is a scalar multiple of $u(x+1)$.  Taking
into account the leading coefficients, we deduce the following equations.

\begin{equation}
\label{odecons1}
(2s + 1) v(x) = u(x+1) + 2(x+1) u'(x+1)
\end{equation}

\begin{equation}
\label{odecons2}
 (2s + 1) u(x+1) = v(x) + 2x v'(x)
\end{equation}

Differentiating Equation~\ref{odecons1} we obtain

\begin{equation}
\label{odecons3}
(2s + 1) v'(x) = 3 u'(x+1)  + 2(x+1) u''(x+1) \text{ .}
\end{equation}

Multiplying Equation~\ref{odecons2} by $(2s+1)$, and then using
Equations~\ref{odecons1} and~\ref{odecons3} to eliminate $v$ and
$v'$, we obtain

\begin{equation}
\label{odedef}
(2s + 1)^2 u(x+1) = u(x+1) + 2(x+1) u'(x+1) +
2x(3u'(x+1) + 2(x+1) u''(x+1))
\end{equation}

Collecting terms, we see that $u(x+1)$ must satisfy the following
differential equation.

\begin{equation}
\label{odedef2}
(2s^2 + 2s) Y + (3 - 4x) Y' + 2(x - x^2) Y'' = 0
\end{equation}

A routine calculation shows that if $u(x+1)$ is a solution to
Equation~\ref{odedef2} and we define $v(x)$ via
Equation~\ref{odecons1} and set $B := -u(1)^2$, then these data
satisfy Equation~\ref{Weq}.

The linear differential operator $L = 2(x - x^2) \frac{d^2}{dx^2} + (3 -
4x) \frac{d}{dx} + (2s^2 + 2s)$ defines a linear operator on the
$(s+1)$-dimensional space of polynomials of degree $s$.  With
respect to the standard monomial basis of this space, the matrix $M =
(M_{i,j})$ of $L$ is upper triangular.  On the main diagonal, we
have $M_{j,j} = 2(1-j)j -4j + (2s^2 + 2s) = (2s^2 + 2s) - (2j^2 +
2j)$ and just above the diagonal we have $M_{j,j+1} = (j+1)(3 +
2j)$.  In particular, $M_{s,s} = 0$ so that
$\operatorname{rank}(L) \leq s$ while the $(s,s)$-minor is
invertible.  Thus, the rank of $L$ is $s$ and the dimension of the
space of solutions to Equation~\ref{odedef2} is exactly one.
As we require $u$ to be monic, there is exactly one solution of
degree $s$.
\end{proof}

\begin{prop}
\label{typeCprop}
For each positive integer $s$, there is a unique
monic polynomial $u$ of degree $s$ and nonzero parameter $B$ for
which there is another monic polynomial $v$ satisfying

\begin{equation}
\label{Ceq} (+B) \circ (x \cdot u(x)^2) \circ (+1) = (x \cdot
v(x^2))
\end{equation}
\end{prop}

\begin{proof}

As before, since $B \neq 0$, $u(x+1)$ and
$v(x^2)$
 are coprime.
Differentiating, we obtain

\begin{equation}
\label{Cdeq}
u(x+1) \cdot (u(x+1) + 2(x+1) u'(x+1)) = v(x^2) + 2
x^2 v'(x^2) = (v + 2 x \cdot v') \circ P_2
\end{equation}

The zeros of the righthand side of Equation~\ref{Cdeq} come in
$\pm$-pairs.  We claim that for each such pair one is a root of
$u(x+1)$ and the other is a root of $(u(x+1) + 2(x+1) u(x+1))$.
Indeed, it cannot happen that $u(c+1) = 0$ and $u(-c+1) =0$ for
Equation~\ref{Ceq} would yield $c u_2(c^2) = B = -c u_2((-c)^2) =
- c u_2(c^2)$ contrary to the fact that $B \neq 0$.  Thus, at most
one of each pair of roots of the righthand side is also a root of
$u(x+1)$. As the degree of the righthand side of
Equation~\ref{Cdeq} is twice that of $u$, it follows that at least
one root from each pair must be a root of $u(x+1)$.   Matching
leading coefficients, we conclude:

\begin{equation}
\label{Cdeq2} (-1)^s (2s + 1) u(x+1) = (u(-x+1) + 2(-x+1)
u'(-x+1))
\end{equation}

Substituting $z := -x + 1$, we see that $u$ satisfies the
following difference-differential equation:

\begin{equation}
\label{Cddeq} 0 = u(z) + 2z u'(z) - (2s+1)(-1)^su(2-z)
\end{equation}

The difference-differential operator in Equation~\ref{Cddeq} is a
linear operator on the space of degree $s$ polynomials and it is given by an upper triangular matrix relative to the standard monomial basis.
The entries along the main diagonal are $$1 + 2j -
(-1)^{j+s}(2s +1)$$  Hence, the rank of this operator is exactly
$s$ implying that there is a unique monic solution.
\end{proof}

This concludes the proof of Theorem~\ref{rittscaltran}.

\subsection{Proof of Theorem~\ref{cruciallemma} and related results}
 We collect some observations about Ritt swaps towards and around the proof of Theorem~\ref{cruciallemma}.

\begin{Rk}
\label{inoutdegForSwap}
 It is clear from the definitions that if some decomposition may be obtained from $\vec{f}$ by a Ritt
swap at $i$, then one of the following must happen: \begin{itemize}
\item both $f_i$ and $f_{i+1}$ are linearly related to monomials;
\item both $f_i$ and $f_{i+1}$ are linearly related to odd-degree Chebyshev polynomials;
\item $f_i$ is linearly related to a monomial $P_p$ and $f_{i+1}$ is linearly related to a Ritt polynomial whose out-degree is a multiple of $p$; or
\item $f_{i+1}$ is linearly related to a monomial $P_p$ and $f_{i}$ is linearly related to a Ritt polynomial whose in-degree is a multiple of $p$.
\end{itemize}
\end{Rk}

\begin{Rk}
\label{Jnoswapleft}
The in-degree of a type \textsf{A} swappable $f$ is $1$, in the sense that any Ritt polynomial linearly related to $f$ has
in-degree $1$. Remark~\ref{inoutdegForSwap} then implies that if $f_i$ is type \textsf{A}, no decomposition may be obtained from $\vec{f}$ by a Ritt swap
at $i$.
\end{Rk}

We now prove some useful consequences of Theorem~\ref{transrelate}, including Theorem~\ref{cruciallemma}.
We begin with a few slightly more comprehensive results about Chebyshev polynomials.

\begin{corollary}\label{oneCRk}
If $L$ and $M$ are linear, $p \geq 3$ is prime, and $L \circ C_p \circ M = C_p$, then both $M$ and $L$ are $(\cdot \pm 1)$.\end{corollary}

\begin{proof}
As in the proof of Lemma~\ref{rittscaltran}, there are scalars $A$, $B$, $\lambda$ and $\mu$ such that
$L = (\cdot \lambda) \circ (+B)$ and $M = (+A) \circ (\cdot \mu)$.
Let $h := (+B) \circ C_p \circ (+A) = (\cdot \frac{1}{\lambda}) \circ C_p \circ (\cdot \frac{1}{\mu})$.
By the first equality, $h$ is monic, so $h = \frac{1}{\mu} \ast C_p = C_{p,\frac{1}{\mu}} $ is a Ritt polynomial. Since $C_p$ has in-degree at least $2$, so does $h$. By Theorem~\ref{transrelate}, $h$ cannot be non-trivially translation related to another Ritt polynomial $C_p$ with in-degree $2$, so $A=B=0$ and $h = C_p$. Since all complex roots of $C_p$ are real, $C_p \neq C_{p,\frac{1}{\mu}} $ unless $\mu = \pm 1$.
\end{proof}

\begin{lemma} \label{seeswaplem}
For any Ritt swap involving a type \textsf{C} swappable, the underlying basic Ritt identity is either of the form $C_p \circ C_q = C_q \circ C_p$ for odd prime $p$ and $q$, or of the form $P_2 \circ C_p = \widehat{C}_p \circ P_2$ for some odd prime $p$. In particular, if a type \textsf{C} swappable $f_i$ ``becomes'' $g_j$ through Ritt swaps, in the sense of Remark~\ref{costumes}, then $g_j$ is also a type \textsf{C} swappable. \end{lemma}

\begin{proof}
By Theorem~\ref{transrelate}, $C_p$ (for odd prime $p$) is not linearly related to any Ritt polynomials except $C_{p, \lambda}$ and $\widehat{C}_{p, \lambda}$. By Lemma~\ref{seehatinout}, all these two have in- and out-degrees $1$ and $2$, so they can only participate in basic Ritt identities of the third kind with the quadratic $P_2$. It is easy to obtain the identity $P_2 \circ C_p = \widehat{C}_p \circ P_2$ from $C_2 \circ C_p = C_p \circ C_2$ and the definition of $\widehat{C}_{p}$ (see the proof of Proposition~\ref{Cprop}).\end{proof}

\begin{lemma} \label{ccluster-unique}
If $A$ and $B$ are linear, $n \neq 2$, and $B \circ C_n \circ A = C_n$, then each of $A$ and $B$ are scalings by $\pm 1$. \end{lemma}

\begin{proof}
Let $p_k, \ldots, p_1$ be the prime factors of $n$, with repetitions, with $p_k, \ldots, p_m$ equal to $2$ and the rest odd. Now $( B \circ C_{p_k}, C_{p_{k-1}}, \ldots, C_{p_{2}}, C_{p_1} \circ A)$ must be linearly equivalent to $( C_{p_k}, C_{p_{k-1}}, \ldots, C_{p_{2}}, C_{p_1})$. Let $L_{k-1}, \ldots, , L_1$ witness this. Induct right-to-left.

If any $p_i$ are odd, then $p_1$ is odd, so in $L_1^{-1} \circ C_{p_1} \circ A = C_{p_1}$ we must have $L_1 = A = ( \cdot (\pm 1))$ by Corollary~\ref{oneCRk}. Then at each step, $L_i^{-1} \circ C_p \circ ( \cdot (\pm 1))$ forces $L_i = ( \cdot (\pm 1))$ (even for $p=2$), and finally at the last step, $B = ( \cdot (\pm 1))$.

If all $p_i = 2$, then $k \geq 2$. From $L_1^{-1} \circ C_2 \circ A = C_{2}$ we get (using Remark~\ref{linrelquadrk}) that $A = (\cdot \lambda)$ is a scaling and $L_1^{-1}(x) := \frac{1}{\lambda^2} x + \frac{2}{\lambda^2} - 2$.
From the next step (since $k \geq 2$, there is a next step), we see that $L_1$ must also be a scaling, so $\lambda = \pm 1$ and $A$ is as desired, and $L_1 = \id$. Now inducting, at each step $L_i^{-1} \circ C_2 = C_2$ makes all $L_i = \id$, and at the last step $B= \id$.
\end{proof}

The next lemma is something of a converse to Proposition~\ref{scaledbasic}.

\begin{lemma}
\label{1mono} \label{scalingbasic2}
 If $a$ and $b$ are Ritt polynomial and not both type \textsf{C};
$L$, $M$, and $N$ are linear; and $( L \circ b \circ M^{-1})
\circ (M \circ a \circ N^{-1}) = \tilde{d} \circ \tilde{c}$ is a
basic Ritt identity, then $L$, $M$, and $N$ are scalings.

 Furthermore, there are Ritt polynomial $c$ and $d$ such that $b \circ a = d \circ c$ is
another basic Ritt identity, which is
 linearly equivalent to the first one,
 and in particular $(\tilde{d}, \tilde{c})$ is linearly equivalent to $(d,c)$.
\end{lemma}

\begin{proof}
Since $a$ and $b$ are not both type \textsf{C}, one of them must be
(linearly related to, and therefore equal to) a monomial.

If $a$ is a monomial, then $M$ and $N$ must be scalings, since
monomials are not translation related to any other Ritt polynomial.
Since both $b$ and $L \circ b \circ M^{-1}$ are Ritt polynomial and $M$
is a scaling, $L$ must also be a scaling because the equation in
Problem~\ref{problemtrans} has no solutions with $B \neq 0 = A$.

If $b$ is a monomial, then $L$ and $M$ must be scalings. Since
both $a$ and $(M \circ a \circ N^{-1})$ must be Ritt polynomial, either $N$
is a scaling or both $a$ and $(M \circ a \circ N^{-1})$ must be
type \textsf{A}. However $( L \circ b \circ M^{-1} ,  M \circ a \circ N^{-1})$
is swappable, contradicting Remark~\ref{Jnoswapleft}.

The ``furthermore'' clause follows immediately from Proposition~\ref{scaledbasic}.
\end{proof}

We complete the proof of Theorem~\ref{cruciallemma}: If two decompositions $\vec{h}$ and $\vec{g}$
are both obtained from $\vec{f}$ by a Ritt swap at $i$, then $\vec{h}$ is linearly equivalent to $\vec{g}$.
\label{proofofcrucial}

\begin{proof} \textbf{This is the proof of Theorem~\ref{cruciallemma}.}
Let us collect and name the witnesses for the two Ritt swaps at $i$.

That is, for $j = 1$ or $2$ we have linear polynomials
$L_j$, $M_j$, and $N_j$ and Ritt polynomial polynomials $G_j$, $H_j$,
$\widehat{G}_j$ and $\widehat{H}_j$ such that
\begin{itemize}
\item $G_j = L_j^{-1} \circ f_{i+1} \circ M_j$
\item $H_j = M_j^{-1} \circ f_i \circ N_j$
\item $G_j \circ H_j = \widehat{H}_j \circ \widehat{G}_j$ is a basic Ritt identity
\item $g_{i+1} = L_1 \circ \widehat{H}_1$
\item $g_i = \widehat{G}_1 \circ N_1^{-1}$
\item $h_{i+1} = L_2 \circ \widehat{H}_2$, and
\item $h_{i} = \widehat{G}_2 \circ N_2^{-1}$.
\end{itemize}

We seek a linear $R$ for which
$$(L_1 \circ \widehat{H}_1) \circ R = (L_2 \circ \widehat{H}_2) \mbox{ and }
R^{-1} \circ (\widehat{G}_1 \circ N_1^{-1}) = \widehat{G}_2 \circ N_2^{-1}$$
Let
 $$L :=  L_2^{-1} \circ L_1\ \mbox{ and } M := M_2^{-1} \circ M_1\mbox{ and }N := N_2^{-1} \circ N_2$$
Then
 $$L \circ G_1 \circ M^{-1} = G_2 \mbox{ and } M \circ H_1 \circ N^{-1} = H_2$$
Applying $L_2$ to the left of the first equation below and $N_2^{-1}$ to the right of the second one shows that it is sufficient to find $R$ such that
$$(L \circ \widehat{H}_1) \circ R = \widehat{H}_2 \mbox{ and }
 R^{-1} \circ (\widehat{G}_1 \circ N^{-1}) = \widehat{G}_2$$
Recall that $G_j \circ H_j = \widehat{H}_j \circ \widehat{G}_j$ are basic Ritt identities, so the above equations are linear relations between Ritt polynomial polynomials.

We claim that $R = M^{-1}$ always works and  is always a scaling.
We consider separately the three cases that none, one, or both of $G_1$ and $H_1$
are monomials. Since $G_2$ is linearly related to $G_1$, $G_2$ is a monomial if and
only if $G_1$ is, and if both are monomials, then $G_1 = G_2$, and similarly for $H_i$.

\begin{itemize}
 \item[(none)] In this case, $\widehat{G}_i = G_i$ and $\widehat{H}_i = H_i$ are
Chebyshev polynomials of odd degree, since commuting Chebyshevs are the only
basic Ritt identity not involving any monomials. Then $R = M^{-1}$
works. (In fact, $L=M=N=(\cdot \pm 1)$ in this case, as Chebyshev polynomials are
not non-trivially linearly related to themselves except via $(-1) \ast C_p = C_p$.)

\item[(one)] This is done in Lemma~\ref{1mono}, with $b := G_1$, $a :=
H_1$ with one less assumption.

\item[(two)] In this case, $\widehat{G}_i = G_i$ and $\widehat{H}_i = H_i$
are monomials, since this is the only basic Ritt identity with two
monomials on one side. Then $R = M^{-1}$ works. (In fact, $L$,
$M$, and $N$ are scalings in this case, as monomials
are not non-trivially translation related to themselves.)
\end{itemize}
\end{proof}

The above proof does not use the hardest part of our analysis: it
suffices to know that Problem~\ref{problemtrans} has no solutions with
$A =0 \neq B$, and to have a characterization of solutions with
$B= 0 \neq A$, the type \textsf{A} Ritt polynomial polynomials. The full strength
of Theorem~\ref{transrelate} is used in the proof of
the fundamental Theorem~\ref{fundamentallemma}.

We end this section with a lemma closely resembling Lemma~\ref{scalingbasic2}.

\begin{lemma}
\label{neitherC-notranslation} \label{scalingbasic3}
 Suppose that $a$ and $b$ are Ritt polynomial and neither is type \textsf{C};
  $L$, $M$, $R$, and $S$ are linear; and
 $\tilde{b} := S \circ b \circ R$ and
 $\tilde{a} := M \circ a \circ L$ are Ritt polynomials; and
 $\tilde{b} \circ \tilde{a} = \tilde{d} \circ \tilde{c}$
 is a basic Ritt identity.

Then $L$, $M$, and $S$ are scalings; and there are scalars $A$ and $\lambda$,
 and a Ritt polynomial $\hat{b} := b \circ (+A)$
such that $b \circ R = \lambda \ast \hat{b}$; and for some $d$ and $c$,
$\hat{b} \circ a = d \circ c$ is a basic Ritt identity. Unless $b$ is
type \textsf{A}, $A = 0$ and $\hat{b} = b$.
\end{lemma}

\begin{proof}
Since neither $a$ nor $b$ is type \textsf{C}, $M$ and $S$ must be scalings.
If $L$ is not a scaling, then $a$ and $\tilde{a}$ must be type \textsf{A},
but this contradicts Remark~\ref{Jnoswapleft}. Thus, $L$ is a scaling, say, by $\lambda$,
and $\tilde{a} = \lambda \ast a$.

Write $C = (+A) \circ (\cdot \mu)$
for scalars $A$ and $\mu$.
Then $\hat{b} := b \circ (+A)$ is a monic polynomial
 scaling related to the Ritt polynomial $\tilde{b}$,
so it is itself Ritt polynomial. Thus if $A \neq 0$, then $b$ is type \textsf{A}, and in any case
$\tilde{b} = \mu \ast \hat{b}$.

Thus $(\mu \ast \hat{b}) \circ (\lambda \ast a) = \tilde{d} \circ
\tilde{c}$ is a basic Ritt identity. By Proposition~\ref{scaledbasic},
there are $\eta$ and $\nu$ such that
$ \hat{b} \circ a = (\eta \ast \tilde{d}) \circ (\nu \ast \tilde{c})$
is a basic Ritt identity.
\end{proof}

\section{Clusters}
\label{clustersec} \label{sect42}

In this section, we describe a natural and nearly canonical way to break a decomposition of a disintegrated polynomial into \emph{clusters} in a way that controls the linear factors floating amongst Ritt polynomials and makes it easy to see what other decompositions can be obtained via sequences of Ritt swaps.

One of our two kinds of clusters is the same as one of the two kinds of blocks in \cite{MZ}, but our \textsf{C}-free clusters are nothing like their monomial blocks.

Our first use of these clusters is to prove the fundamental Theorem~\ref{fundamentallemma} for the Ritt monoid action, that $t_i t_{i+1} t_i \star \vec{f}$ is defined if and only of $t_{i+1} t_i t_{i+1} \star \vec{f}$ is defined. While it follows immediately from \cite{MZ} that the two are equal when defined, it is not clear to us whether our stronger result follows from their work.

 We end up showing that any polynomial all of whose indecomposable factors are swappable can be written almost uniquely (not up to permutations!) as a composition of lower-degree polynomials (compositions of ``clusters'') each of which is linearly related to a (possibly decomposable) Chebyshev polynomial, or to a composition of several Ritt polynomials, none of them type \textsf{C}.

 We do not work out the straightforward generalization that includes decompositions with unswappable factors in this analysis by allowing a third kind of cluster, a single unswappable indecomposable: our case-outs are unwieldy enough as it is, and Corollary~\ref{unswcrack2} already takes complete care characterizing $(f,g)$-skew-invariant curves when one (and, therefore, both) polynomials has at least one unswappable factor.

 The polynomial then admits a decomposition where linear factors are collected outside these clusters, Ritt swaps can only take place inside these clusters, and can always be witnessed (with respect to this decomposition) by identity linear factors - almost. Factors of degree two are the only source of ambiguity in choosing these clusters, and one of them is always involved in any Ritt swap between two clusters.

 Further refining our analysis of the linear factors between clusters, we show that, when the clusters are maximal enough, only one quadratic may cross the boundary between two clusters, and then only in one direction.

 Throughout this section we constantly use the results of the Section~\ref{newsectionlinrel} without explicit reference.

 \begin{Def} \label{intervalindex}
 Given a sequence $\vec{f} := (f_k, \ldots, f_1)$ of polynomials and integers $k \geq b \gneq a \geq 0$, we use the following notations
 $$\vec{f}_{[b,a)} := (f_b, f_{b-1}, \ldots, f_{a+1}) \hspace{1cm}
 \vec{f}_{[b,a)}^\circ := f_b \circ f_{b-1} \circ \ldots \circ f_{a+1} \hspace{1cm}
 \vec{f}^\circ := f_k \circ f_{k-1} \circ \ldots \circ f_1$$
 \end{Def}

\begin{Def} \label{one-cluster-def}
 Let $\vec{f}$ be a decomposition of a polynomial $f$. \begin{itemize}
 \item If $ B \circ \vec{f}_{[b,a)}^\circ \circ A =  C_n $ is a Chebyshev polynomial for some integer $n$ that is not a power of $2$, and some linear $A$ and $B$, then $\vec{f}_{[b,a)}$ is a \emph{\textsf{C} cluster}.
 \item If $\vec{f}_{[b,a)}$ is linearly equivalent to $(B \circ h_b, h_{b-1}, \ldots,  h_{a+2}, h_{a+1} \circ A)$ for some linear $A$ and $B$ and some Ritt polynomials $h_i$ none of which are type \textsf{C}, then $\vec{f}_{[b,a)}$ is an \emph{\textsf{C}-free cluster}.
 \end{itemize}
In either case, $\vec{f}_{[b,a)}$ is a \emph{cluster} of $\vec{f}$.

  A \emph{preclustering} of a decomposition $\vec{f}$ is a sequence $k = a_r > a_{r-1} > \ldots > a_1 > a_0 =0$
   such that $\vec{f}_{[a_j, a_{j-1})}$ is a cluster for each $j$. We say that $i$ is \emph{a cluster boundary of $\vec{a}$} if $i = a_j$ for some $j$.
\end{Def}

\begin{Def}
If $k = a_r > a_{r-1} > \ldots > a_1 > a_0 =0$ is a preclustering of a decomposition $\vec{f}$, the data  $(h_k, \ldots, h_1; L_{k}, L_{k-1}, \ldots, L_0)$ is a \emph{cleanup} of this preclustering if \begin{enumerate}
 \item $(L_{k} \circ h_k \circ L_{k-1}, h_{k-1}\circ L_{k-2}, \ldots, h_1 \circ L_0)$ is linearly equivalent to $\vec{f}$;
 \item all $L_i$ are linear, and $L_i = \id$ except when $i$ is a cluster boundary of $\vec{a}$;
 \item inside \textsf{C} clusters (that is, whenever $\vec{f}_{[a_j, a_{j-1})}$ is a \textsf{C} cluster and $a_j \geq i > a_{j-1}$) $h_i = C_{\deg(f_i)}$ are Chebyshev polynomials;
 \item inside \textsf{C}-free clusters (that is, whenever $\vec{f}_{[a_j, a_{j-1})}$ is a \textsf{C}-free cluster and $a_j \geq i > a_{j-1}$) $h_i$ are Ritt polynomials;
 \item the linear factor $L_{a_{j}}$ to the right of any \textsf{C}-free cluster $\vec{f}_{[a_{j+1}, a_j)}$ is a translation, and if $h_{a_j+1} \circ L_{a_j}$ is a Ritt polynomial, then $L_{a_j} = \id$.
\end{enumerate}
\end{Def}

\begin{Rk} Swappable factors of a decomposition $\vec{f}$ are linearly related to Ritt polynomials, and the linear factors witnessing this can be gathered outside clusters in the following somewhat canonical fashion.
 Applying the definition of ``cluster'' to all clusters of a preclustering puts a linear factor on each side of each cluster. Composing pairs of linear factors that sit between clusters, we may assume that only the leftmost cluster has a linear factor on the left of it. To obtain a cleanup, push all scalings through \textsf{C}-free clusters as far left as possible (Lemma~\ref{babycleanuplem}). Generalizing the results of Section~\ref{section41} from single indecomposable factors to clusters (Lemma~\ref{nccluster-unique} and Lemma~\ref{ccluster-unique}) shows that these cleanups are essentially unique up to scalings by $\pm 1$ (Proposition~\ref{unique-cleanup}). \end{Rk}

The decomposition with $k=0$ factors, whose clustering has $r=0$ clusters and whose cleanup has no $h_i$ and $L_0 = \id$, provides a conveniently trivial base case for proofs by induction on the number of clusters.

\begin{Rk} \label{cluster-rk}
\begin{enumerate}
\item A decomposition with an unswappable factor does not admit a clustering. Dealing with such decompositions is much easier, and does not require the machinery of this section.
\item If every factor $f_i$ of $\vec{f}$ is swappable, then $r:=k$ and $a_j := j$ is a preclustering of $\vec{f}$.
\item All these notions (cluster, preclustering, cleanup) only depend on the linear equivalence class of the decomposition $\vec{f}$.
\item A non-empty subsequence of a cluster is a cluster, unless a \textsf{C} cluster loses all of its odd-degree factors. That is, if $\vec{f}_{[b,a)}$ is a cluster of $\vec{f}$ and $b \geq b' > a' \geq a$, and $\vec{f}_{[b',a')}$ is not a cluster, then $\vec{f}_{[b',a')}^\circ$ is linearly related to $C_{2^M}$ for some $M \geq 1$ and the original cluster $\vec{f}_{[b,a)}$ was a \textsf{C} cluster.
\item If $t_i \star \vec{f}$ is defined, then $(f_{i+1}, f_{i})$ is a cluster.
\end{enumerate} \end{Rk}

The next lemma is used to achieve the last part of the definition of ``cleanup''.

\begin{lemma}
\label{babycleanuplem}
 Any \textsf{C}-free cluster $\vec{f}_{[b, a)}$ is linearly equivalent to $(B \circ \tilde{h}_b, \tilde{h}_{b-1},$ $\ldots,$ $\tilde{h}_{a+2}, \tilde{h}_{a+1} \circ A)$ for some Ritt polynomials $\tilde{h}_i$, some linear $B$, and some translation $A$, such that $\tilde{h}_{a+1} \circ A$ is a Ritt polynomial only if $A = \id$.
\end{lemma}

\begin{proof}

Take linear $A_0$ and $B_0$ and Ritt polynomials $h_i$ from the definition of \textsf{C}-free cluster.
Write $A_0 := S_0 \circ A_1$ for a scaling $S_0 =: (\cdot \lambda)$ and a translation $A_1$.

Let $S_1 := (\cdot \lambda^{\deg(h)} )$; by Remark~\ref{astRitt}, $\tilde{h}_{a+1} := \lambda \ast h_{a+1} =S_1^{-1} \circ h_{a+1} \circ S_0$ is a Ritt polynomial. Similarly, there is a scaling $S_2$ such that $\tilde{h}_{a+2} := S_2^{-1} \circ h_{a+2} \circ S_1$ is Ritt polynomial; and so on until we get $\vec{f}_{[b, a)}$ linearly equivalent to $(B_0 \circ S_{b-a} \circ \tilde{h}_{b}, \tilde{h}_{b-1}, \ldots, \tilde{h}_{a+2}, \tilde{h}_{a+1} \circ A_1)$. Set $B := B_0 \circ S_{b-a}$.

 Finally, if $\tilde{h}_{a+1} \circ A_1$ is  Ritt polynomial, replace $\tilde{h}_{a+1}$ by $\tilde{h}_{a+1} \circ A_1$, and let $A := \id$. Otherwise, leave $\tilde{h}_{a+1}$ as is, and let $A := A_1$.
\end{proof}

The next lemma in some sense generalizes the results of Section~\ref{section41} from single indecomposable factors to whole \textsf{C}-free clusters; for \textsf{C} clusters, this is already done in Lemma~\ref{ccluster-unique}. The two are induction steps of the proof of uniqueness of cleanups in Lemma~\ref{unique-cleanup}.

\begin{lemma}
\label{nccluster-unique}
Suppose that all $h_i$ and $\tilde{h}_i$ are Ritt polynomials, not type \textsf{C}; that $B$ and $\tilde{B}$ are linear, and $A$ and $\tilde{A}$ are translations; and that $(B \circ h_b, h_{b-1}, \ldots,  h_{a+2}, h_{a+1} \circ A)$ is linearly equivalent to $(\tilde{B} \circ \tilde{h}_b, \tilde{h}_{b-1}, \ldots,  \tilde{h}_{a+2}, \tilde{h}_{a+1} \circ \tilde{A})$. Then $B = \tilde{B}$, and $h_i = \tilde{h}_i$ for all $i > a+1$, and  $\tilde{h}_{a+1} \circ \tilde{A} =  h_{a+1} \circ A$; unless $h_{a+1}$ is type \textsf{A}, $\tilde{h}_{a+1} = h_{a+1}$ and $A = \tilde{A}$. \end{lemma}

\begin{proof}
Let $L_b, \ldots, L_{a+2}$ witness linear equivalence:
$$ (B \circ h_b \circ L_b, L_b^{-1} \circ h_{b-1} \circ L_{b-1}, \ldots, L_{a+2} \circ h_{a+1} \circ A) =
(\tilde{B} \circ \tilde{h}_b, \tilde{h}_{b-1}, \ldots, , \tilde{h}_{a+1} \circ \tilde{A})$$

We induct right-to-left from $i = a+2$ to show that all $L_i = \id$, and for $i > a+2$, all $\tilde{h}_{i-1} = h_{i-1}$.

For the base case $i = a+2$, we have   $$L_{a+2}^{-1} \circ h_{a+1} \circ A = \tilde{h}_{a+1} \circ \tilde{A}$$
 Since the right-hand side of the equation is monic, $L_{a+2}$ is a translation. Since $h_{a+1}$ is not type \textsf{C}, $L_{a+2} = \id$, so $\tilde{h}_{a+1} \circ \tilde{A} =  h_{a+1} \circ A$.

For the induction step $a+1 < i < b$, we have $L_{i-1} = \id$ and $ L_{i}^{-1} \circ h_{i-1} \circ L_{i-1} = \tilde{h}_{i-1}$ with both $h_{i-1}$ and $\tilde{h}_{i-1}$ Ritt polynomials not type \textsf{C}, so $L_i = \id$ and $h_i = \tilde{h}_i$.

Finally, we have $L_b = \id$ and $ B \circ h_b \circ L_b^{-1} = \tilde{B} \circ \tilde{h}_b$, which forces $B = \tilde{B}$ and $h_b = \tilde{h}_b$.
\end{proof}

\begin{lemma} Every preclustering admits a cleanup. \label{clustercleanlem}
 \end{lemma}
\begin{proof}
Let $k=a_r > \ldots > a_1 > a_0=0$ be a preclustering of a decomposition $\vec{f}$, and induct on $r$, the number of clusters. For the trivial base case, $L_0 := \id$ is a cleanup of the preclustering with zero clusters of the decomposition with no factors.

The preclustering $k' := a_{r-1} > \ldots > a_1 > a_0=0$ of $\vec{f}_{[k',0)}$ has one less cluster, so by induction it admits a cleanup $(\vec{h}_{[k', 0)}; \vec{L}_{[k',0]}))$. Let $\vec{h}_{[k, k')}$ and linear $A_0$ and $B_0$ witness that $\vec{f}_{[k, k')}$ is a cluster.

Now $\vec{f}$ is linearly equivalent to $(B \circ h_k, h_{k-1}, \ldots, h_{k'+1} \circ A \circ L_{k'}, h_{k'} \circ L_{k'-1}, \ldots, h_1 \circ L_0)$. (Here, $A$ and $B$ are $A_0$ and $B_0$ or their inverses, depending on kind of cluster.) Set $L_k := B$; replace $L_{k'}$ by $A \circ L_{k'}$; and, if $\vec{f}_{[k, k')}$ is a \textsf{C}-free cluster, apply Lemma~\ref{babycleanuplem} to $(B \circ h_k, h_{k-1}, \ldots, h_{k'+1} \circ L_{k'})$ to obtain the desired cleanup of $\vec{f}$.
\end{proof}

The next lemma says that cleanups are unique, up to two minor variations arising from Lemmas~\ref{ccluster-unique} and~\ref{nccluster-unique}: a translation to the right of a type \textsf{A} factor is not well-defined, and scalings by $\pm 1$ can appear and disappear as in the following remark.

\begin{Rk} \label{neg1wild}
Here is what scalings by $-1$ can do. Suppose that $(\vec{h}, \vec{L})$ is a cleanup of a preclustering $\vec{a}$ of a decomposition $\vec{f}$, and that $\vec{f}_{[a_{j+1}, a_j)}$ is a \textsf{C} cluster. A new cleanup $(\vec{g}, \vec{M})$ of $\vec{a}$ can be obtained by introducing a scaling by $-1$ into $L_{a_j}$ and then pushing it left through the cleanup until it is swallowed by a factor of even degree, or absorbed into the last linear factor of the cleanup.
To state this precisely, let $b := a_j$ to lighten notation.\\
 For $i < b$, let $g_i := h_i$ and $M_i := L_i$.\\
 Let $g_b := h_b$ and $M_b := (\cdot -1) \circ L_b$.\\
 For $i > b$, let $s_i := \deg (\vec{f}_{[i,b)}^\circ))$ (we only care about its parity).\\
 For $k> i > b$, let $M_i := (\cdot -1)^{\circ s_i} \circ L_i \circ (\cdot -1)^{\circ s_i}$ and let
 $g_i := (-1)^{s_i} \ast h_i$. Finally, let $M_k := L_k \circ (\cdot -1)^{\circ s_k}$.\\
Of course, this can happen several times with different starting points $b$.
\end{Rk}

\begin{lemma} \label{unique-cleanup}
Suppose that $(\vec{h}, \vec{L})$ and $(\vec{g}, \vec{M})$ are two different cleanups of the same preclustering $\vec{a}$ of the same decomposition $\vec{f}$. Then \begin{enumerate}
    \item $g_i = (\pm 1) \ast h_i$ for all $i$ except as in (3) below;
    \item $M_k = L_k \circ ( \cdot (\pm 1))$ and $M_i = ( \cdot (\pm 1)) \circ L_i \circ ( \cdot (\pm 1))$ for all $i$ except as in (3) below;
    \item if $f_{a_j +1}$ is type \textsf{A}, then there is a translation $T$ such that
    $$\tilde{g}_{a_j +1} := g_{a_j +1} \circ T = (\pm 1) \ast h_{a_j +1} \mbox{ and }
    \tilde{M}_{a_j } := T^{-1} \circ M_{a_j} = ( \cdot (\pm 1)) \circ L_{a_j } \circ ( \cdot (\pm 1)).$$
    \end{enumerate}
Remark~\ref{neg1wild} gives more detail about the scalings $( \cdot (\pm 1))$.
\end{lemma}

\begin{proof}
By definition of cleanup, both $(L_{k} \circ h_k \circ L_{k-1}, h_{k-1} \circ L_{k-2}, \ldots, h_1 \circ L_0)$ and $(M_{k} \circ g_k \circ M_{k-1}, g_{k-1} \circ M_{k-2}, \ldots, g_1 \circ M_0)$ are linearly equivalent to $\vec{f}$, so they are linearly equivalent to each other. Name the linear factors witnessing this, and then start from the right and induct leftward, exactly as in the proof of Lemma~\ref{clustercleanlem}.

More formally, induct again on the number of clusters, with the same trivial base case of the unique cleanup $( ; \id )$ of the clusterless preclustering of a decomposition with zero factors. Once again, from the preclustering
$k=a_r > \ldots a_1 > a_0=0$ of the decomposition $\vec{f}$, we obtain a preclustering
$k' := a_{r-1} > \ldots a_1 > a_0=0$ of $\vec{f}_{[k',0)}$, with one less cluster.

Now $(\vec{g}; (E \circ M_{k'}, M_{k' -1}, \ldots, M_0))$ and  $(\vec{h}; (D \circ L_{k'}, L_{k' -1}, \ldots, L_0))$ are both cleanups of this preclustering of $\vec{f}_{[k',0)}$, for some linear $D$ and $E$. Applying the inductive hypothesis, we get the desired conclusion for $h_i$ with $i \leq k'$ and $L_i$ with $i < k'$, and we also get that
\begin{equation} \label{eqqqq}
E \circ M_{k'} = D \circ L_{k'} \circ ( \cdot (\pm 1)).
\end{equation}

Finally, the last cluster $f_{[k, k')}$ is now linearly equivalent to both
 $(L_k \circ h_k, h_{k-1}, \ldots, h_{k'+1} \circ D^{-1})$ and $(M_k \circ g_k, g_{k-1}, \ldots, g_{k'+1} \circ {E}^{-1})$.

 If this last cluster is a \textsf{C} cluster $g_i = C_{\deg(f_i)} = h_i$ and Lemma~\ref{ccluster-unique} gives $M_k = L_k \circ ( \cdot (\pm 1))$ and ${E}^{-1} = D^{-1} \circ ( \cdot (\pm 1))$. This together with Equation~\ref{eqqqq} gets us the desired conclusion for $M_{k'}$ and $L_{k'}$.

 If this last cluster is a \textsf{C}-free cluster, then $M_{k'}$ and $L_{k'}$ are both translations, so in Equation~\ref{eqqqq} the scaling parts of $D$ and $E$ are off by $\pm 1$. That is, for some scaling $S$ and some translations $T$ and $U$, we have ${E}^{-1} = (\cdot \pm 1) \circ U \circ S$ and $D^{-1} = T \circ S$. So now we have
 $(L_k \circ h_k, h_{k-1}, \ldots, h_{k'+1} \circ T)$ linearly equivalent to $(M_k \circ g_k, g_{k-1}, \ldots, g_{k'+1} \circ (\cdot \pm 1) \circ U)$. Once we push the scaling by $(\pm 1)$ left through the factors $g_i$, Lemma~\ref{nccluster-unique} applies, giving us the desired conclusion for $L_k$ and $M_k$, and for $h_i$ and $g_i$ for $i \neq k'+1$. If $h_{k'+1}$ is not type \textsf{A}, that lemma also gives the desired conclusion for $h_{k'+1}$ as well as $T = U$, which gives $E = (\cdot \pm 1) D$, and then the desired conclusion for $L_{k'}$ and $M_{k'}$. If $h_{k'+1}$ is type \textsf{A}, we fall into case (3) of the conclusion.
\end{proof}

The next proposition says that, for a given preclustering and cleanup, any Ritt swap inside a cluster can be witnessed by identity linear factors relative to the factors of that cleanup, and so can be accomplished without changing the linear factors of the cleanup.

\begin{prop} \label{swap-inside}
Suppose that $(\vec{h}, \vec{L})$ is a cleanup of a preclustering $\vec{a}$ of a decomposition $\vec{f}$; that $t_i \star \vec{f}$ is defined; and that $i$ is not a cluster boundary of $\vec{a}$.

Then $\vec{a}$ is also a preclustering of $t_i \star \vec{f}$, and it admits a cleanup $(\vec{g}, \vec{L})$ where $g_{i'} = h_{i'}$ for all $i' \neq i, i+1$, and where $g_{i+1} \circ g_i = h_{i+1} \circ h_i$ is a basic Ritt identity.
\end{prop}

\begin{proof}
 Inside a \textsf{C} cluster, all factors are type \textsf{C} or degree two, and we have shown (Lemma~\ref{seeswaplem})
 that the only Ritt swaps amongst these come from $C_m \circ C_n = C_n \circ C_m$, which clearly satisfy the conclusion of this proposition. Inside a \textsf{C}-free cluster, this follows immediately from Lemma~\ref{1mono}.\end{proof}

If each decomposition admitted a unique preclustering with maximal clusters, all possible Ritt swaps would be completely described by Proposition~\ref{swap-inside}, because of the observation  (Remark~\ref{cluster-rk}) that $(f_{i+1}, f_{i})$ is a cluster whenever $t_i \star \vec{f}$ is defined. The following example demonstrates how this can fail.

\begin{Rk}
Consider $\vec{f} := (C_3, C_2, x^{17}\cdot (x^2 +1) )$. Clearly, $3 > 1 >0$ is a preclustering of $\vec{f}$, with a cleanup given by $h_i = f_i$ and $L_i = \id$ for all $i$. But $3 > 2 >0$ is also a preclustering, with a
cleanup $( C_3, x^2, x^{17}\cdot (x^2 +1); ( \id, + (-2), \id) )$. Both Ritt swaps are defined, but for each preclustering only one of them is inside a cluster. It is also clear that $\vec{f}$ is not a single cluster.\end{Rk}

How clusters might fuse and overlap can be read off easily from the linear factors in a cleanup.
It is not hard to see that two adjacent clusters of the same kind (both \textsf{C} or both \textsf{C}-free) can be fused into one cluster if and only if the linear factor between them is identity in some cleanup. It is a good deal harder to show that a \textsf{C} cluster and a \textsf{C}-free cluster can only fuse when the \textsf{C}-free cluster is made up of a single factor of degree two. Along the way we show that two overlapping clusters fuse, unless at least one is a \textsf{C} cluster and the overlap is a single factor of degree two.

\begin{lemma} \label{land-cluster-fuse}
Suppose that $(\vec{h}, \vec{L})$ is a cleanup of a preclustering $k > \ldots > c > b > a > \ldots >0$ of a decomposition $\vec{f}$; that $\vec{f}_{[c,b)}$ and $\vec{f}_{[b,a)}$ are both \textsf{C}-free clusters; and their concatenation $\vec{f}_{[c,a)}$ is also a cluster. Then $L_b = \id$ and $(\vec{h}, \vec{L})$ is also a cleanup of the preclustering $k > \ldots > c > a > \ldots >0$ with the two clusters fused.
\end{lemma}

\begin{proof}
 Since $\vec{f}_{[a,c)}$ is a cluster, $k > \ldots > c > a > \ldots >0$ is indeed a preclustering of $\vec{f}$, which admits a cleanup $(\vec{g}, \vec{M})$ with $M_b = \id$. Clearly, $(\vec{g}, \vec{M})$ is also a cleanup of the original preclustering $k > \ldots > c > b > a > \ldots >0$. Apply Lemma~\ref{unique-cleanup} to compare the translations $M_b$ and $L_b$ sitting to the right of the \textsf{C}-free cluster $\vec{f}_{[a,b)}$ in these two cleanups of the original preclustering. If $f_{b+1}$ is not type \textsf{A}, part (2) of that Lemma immediately gives $L_b = M_b$.
 If $f_{b+1}$ is type \textsf{A}, part (3) of that Lemma gives a translation $T$ such that
  $$ g_{b+1} \circ T = (\pm 1) \ast h_{b +1} \mbox{ and }
 T^{-1} \circ M_{b} = ( \cdot (\pm 1)) \circ L_{b} \circ ( \cdot (\pm 1))$$
 So $g_{b+1} = ( (\pm 1) \ast h_{b +1}) \circ T^{-1} = ( (\pm 1) \ast h_{b +1}) \circ ( ( \cdot (\pm 1)) \circ L_{b} \circ ( \cdot (\pm 1)) )$ is a Ritt polynomial. With the more detailed analysis of $( \cdot (\pm 1)) )$ in Remark~\ref{neg1wild}, it follows that $h_{b +1} \circ L_b$ is a Ritt polynomial, and $L_b = id$.
 \end{proof}

 \begin{lemma} \label{C-cluster-fuse}
Suppose that $(\vec{h}, \vec{L})$ is a cleanup of a preclustering $k > \ldots > c > b > a > \ldots >0$ of a decomposition $\vec{f}$; that $\vec{f}_{[c,b)}$ and $\vec{f}_{[b,a)}$ are both \textsf{C} clusters; and their concatenation $\vec{f}_{[c,a)}$ is also a cluster. Then $L_{b} = (\cdot \pm 1)$.

 If $L_b = \id$, then $(\vec{h}, \vec{L})$ is also a cleanup of the preclustering $k > \ldots > c > a > \ldots >0$ with the two clusters fused. Otherwise, $L_b = (\cdot -1)$ and a cleanup of this preclustering may be obtained by pushing the scaling $(\cdot -1)$ left as in Remark~\ref{neg1wild}.
\end{lemma}

\begin{proof}
As in Lemma \ref{land-cluster-fuse}, any cleanup of the new preclustering $k > \ldots > c > a > \ldots >0$ is also a cleanup of the old preclustering, and applying Lemma~\ref{unique-cleanup} to compare the two cleanups of the old preclustering immediately gives $L_{b} = (\cdot \pm 1)$.
\end{proof}

 It is clear that a \textsf{C} cluster cannot merge with a \textsf{C}-free cluster unless all factors inside the \textsf{C}-free cluster are quadratic. The issue of quadratics is somewhat delicate. The intent of the next definition is that a quadratic needs a gate in the correct direction to get from one cluster to another.

\begin{Def} \label{gate-def}
Suppose that $(\vec{h}, \vec{L})$ is a cleanup of a preclustering $\vec{a}$ of a decomposition $\vec{f}$. Whether or not there are gates between two of the clusters depends on the kinds (\textsf{C} or \textsf{C}-free) of the two clusters and the linear factor between them. Fix $a_j \neq 0, k$ and call $\vec{f}_{[a_{j+1}, a_j)}$ the \emph{left cluster}, and $\vec{f}_{[a_{j}, a_{j-1})}$ the \emph{right cluster}, and $L_{a_j} =: L$

\begin{itemize}
\item If the left cluster and the right cluster are both \textsf{C} clusters, this preclustering has \begin{itemize}
 \item a \emph{left-to-right gate} at $j$ if $L$ is a scaling; and
 \item a \emph{right-to-left gate} at $j$ if $L =  A_\lambda$ or $L = (\cdot -1) \circ A_\lambda$ for some $\lambda$ (see Remark~\ref{linrelquadrk} for the definition of $A_\lambda$). \end{itemize}

\item If the left cluster is a \textsf{C}-free cluster and the right cluster is a \textsf{C} cluster, this preclustering has\begin{itemize}
 \item a \emph{left-to-right gate} at $j$ if $L = \id $ and $h_{a_j +1}$ is a monomial or a Ritt polynomial with in-degree greater than $1$; and
 \item a \emph{right-to-left gate} at $j$ if $h_{a_j +1} \circ L \circ (-2)$ is a Ritt polynomial. \end{itemize}
\item If the left cluster is a \textsf{C} cluster and the right cluster is a \textsf{C}-free cluster, this preclustering has \begin{itemize}
 \item a \emph{left-to-right gate} at $j$ if $L$ is a scaling; and
 \item a \emph{right-to-left gate} at $j$ if $L =  B_\lambda$ or $L = (\cdot -1) \circ B_\lambda$ for some $\lambda$ (see Remark~\ref{linrelquadrk} for the definition of $B_\lambda$).
   \end{itemize}
\item If both are \textsf{C}-free clusters, then this preclustering has \emph{ a two-way gate} at $j$ if $L = \id$.
\end{itemize}

In general, there is a \emph{two-way gate} whenever there are both a right-to-left gate and a left-to-right gate. Otherwise, there is a \emph{one-way gate}.
\end{Def}

Gates are a property of a preclustering and decomposition together, but when $\vec{f}$ is understood, we often say ``$\vec{a}$ has a such-and-such gate at $j$'', and vice versa.

\begin{prop} \label{twodoork}
If a cleanup of a preclustering has a two-way gate at $j$, then the $j$th and the $(j+1)$st cluster are of the same kind (both \textsf{C} or both \textsf{C}-free), and the concatenation $\vec{f}_{[a_{j+1}, a_{j-1})}$ of the two is itself a cluster.\end{prop}

\begin{proof}
 Between clusters of different kinds, two-way gates are not possible. If the \textsf{C} cluster is on the left, note that $B_\lambda$ is never a scaling. If the \textsf{C} cluster is on the right, note that in order for both $h_{a_j +1} \circ L$ and $h_{a_j +1} \circ L \circ (-2)$ to be Ritt polynomials, they must be type \textsf{A}, but that is incompatible with the ``monomial or non-trivial in-degree'' part of the definition.

 Between two \textsf{C}-free clusters, a two-way gate means that the linear factor $L_{a_j}$ is identity by definition. Between two \textsf{C} clusters, a two-way gate means that $L_{a_j}$ is both a scaling and $ (\cdot (\pm 1)) \circ A_\lambda$ for some $\lambda$, which is only possible for $L_{a_j} = (\cdot (\pm 1))$, which can then be absorbed into or passed through the left \textsf{C} cluster.
\end{proof}

Lemmas~\ref{land-cluster-fuse} and~\ref{C-cluster-fuse} give the converse of Proposition~\ref{twodoork}.

\begin{lemma}\label{ezdoork}
\begin{enumerate}
\item Gates are properties of the preclustering, independent of the cleanup.

\item In Proposition~\ref{swap-inside}, the preclustering of the pre-swap decomposition and the preclustering of the post-swap decomposition obtained there have gates in the same places, in the same directions.
\end{enumerate}\end{lemma}

\begin{proof}
 For the first part, it is easy to see that the definition of gates is invariant under the few ways listed in Lemma~\ref{unique-cleanup} for two cleanups of the same decomposition to differ from each other.

 For the second part, recall that the kinds (\textsf{C} or \textsf{C}-free) of clusters, and the linear factors of the cleanup do not change in Proposition~\ref{swap-inside}. Thus, continuing to use the notation from the definition of gates, the only case that needs any work is when the left cluster is a \textsf{C}-free cluster and the right cluster is a \textsf{C} cluster, and $(f_{a_j +2}, f_{a_j +1})$ are the factors involved in the Ritt-swap.

 By Remark~\ref{Jnoswapleft}, the $(a_j+1)$st factors $h_{a_j +1}$ of the pre-swap decomposition and $g_{a_j +1}$ of the post-swap decomposition are not type \textsf{A}. So for fixed linear $M$, $h_{a_j +1} \circ M$ is a Ritt polynomial if and only if $M = \id$, if and only if $g_{a_j +1} \circ M$ is a Ritt polynomial.
\end{proof}

\begin{lemma} \label{precl-linrel}
If $\vec{a}$ is preclustering of a decomposition $\vec{f}$, then for any linear $M$ and $N$, $\vec{a}$ is also a preclustering of the decomposition $\vec{g} := (N \circ f_k, f_{k-1}, \ldots, , f_2, f_1 \circ M)$, with the same kinds of gates in the same places. \end{lemma}

\begin{proof}
For any $b$ and $c$, it follows immediately from the definition of ``cluster'' that $\vec{g}_{[c,b)}$ is a cluster if and only if $\vec{f}_{[c,b)}$ is a cluster, so $\vec{a}$ is also a preclustering of $\vec{g}$. It clearly suffices to prove that the gates remain the same in two special cases, when both $M$ and $N$ are scalings, and when both $M$ and $N$ are translations. Let $(\vec{h}, \vec{L})$ be a cleanup of $\vec{f}$.

If both $M$ and $N$ are translations, then $(\vec{h}; (N \circ L_k, L_{k-1}, \ldots, , L_1, L_0 \circ M) )$ is a cleanup of $\vec{g}$. Since the two outside linear factors $L_k$ and $L_0$ do not contribute to gates in any way, $\vec{g}$ obviously has the same gates as $\vec{f}$.

If both $M := (\cdot \lambda)$ and $N$ are scalings, a cleanup of $\vec{g}$ is obtained by pushing $M$ left through $(\vec{h}, \vec{L})$ until it hits a \textsf{C} cluster or the leftmost linear factor $L_k$, as in the proof of the existence of cleanups (Lemma~\ref{clustercleanlem}). That is, the linear factors $L'_i$ of the new cleanup of $\vec{g}$ will be given by
$L'_i := \mu_i \ast L_i$ for $i < j_0$, $L'_{j_0} := L_{j_0} \circ (\cdot \mu_{j_0} )$, and $L'_i := L_i$ for $i > j_0$, where $j_0$ is the index of the rightmost \textsf{C} cluster of $\vec{f}$, or $j_0 = k$ if $\vec{f}$ has no \textsf{C} clusters; and $\mu_i$ are integer powers of $\lambda$.

The cluster boundaries with $j > j_0$ are clearly unaffected. The cluster boundaries $j < j_0$ lie between two \textsf{C}-free clusters, so $\vec{f}$ has a gate at $j$ if and only if $L_{a_j} = \id$, if and only if $\mu_{a_j} \ast L_{a_j} = \id$, if and only if $\vec{g}$ also has a gate at $j$.
If $j_0$ is a cluster boundary, then it has a \textsf{C} cluster on the left and a \textsf{C}-free cluster on the right. Thus $\vec{f}$ has a left-to-right gate at $j_0$ if and only if $L_{a_{j_0}}$ is a scaling, if and only if $L_{a_{j_0}} \circ (\cdot \mu_{a_{j_0}})$ is a scaling, if and only if $\vec{g}$ also has a left-to-right gate at $j$.
Similarly, $\vec{f}$ has a right-to-left gate at $j_0$ if and only if $L_{a_{j_0}} = B_\nu = (-2) \circ (\cdot \frac{1}{\nu^2})$ for some $\nu$, if and only if $L_{a_{j_0}} \circ (\cdot \mu_{a_{j_0}}) = (-2) \circ (\cdot \frac{1}{\nu^2})  \circ (\cdot \mu_{a_{j_0}}) = B_{\nu'}$ for $\nu' := \frac{\nu}{\sqrt{\mu_{a_{j_0}}}}$, if and only if $\vec{g}$ also has a right-to-left gate at $j$.
\end{proof}

\begin{Def}
If there is a left-to-right gate at $j$ and $f_{a_j +1}$ is quadratic, or if there is a right-to-left gate at $j$ and $f_{a_j}$ is quadratic, we say that this quadratic is a \emph{wandering quadratic}.

If the wandering quadratic $f_i$ is a whole cluster (that is, $a_j = i$ and $a_{j-1} = i-1$ for some $j$ ), then $f_i$ is a \emph{fake wandering quadratic} of this preclustering. Otherwise, it is a \emph{semi-persistent wandering quadratic}.

A one-way gate with no fake wandering quadratics next to it is a \emph{semi-persistent one-way gate}.
\end{Def}

\begin{Rk} \label{waquark}
Like gates, fake and semi-persistent wandering quadratics and semi-persistent one-way gates are properties of the preclustering, independent of the cleanup.

A straightforward exercise in matching Remark~\ref{linrelquadrk} with the definition of gates shows that
 the concatenation of a cluster $f_{[b,a)}$ and an adjacent quadratic factor $f_a$ (respectively, $f_{b+1}$) is a cluster
if and only if
 any preclustering of $\vec{f}$ with $a_{j+1} = b$ and $a_j = a$ has a right-to-left gate at $j$ (respectively, a left-to-right gate at $j+1$).
This implies that the adjacent factor is a wandering quadratic. Almost conversely, if the adjacent factor is a semi-persistent wandering quadratic, then the concatenation is a cluster.
\end{Rk}

\begin{lemma} \label{move-one-per-wa-qua}
If $f_{a_j}$ (respectively, $f_{a_j +1}$) is a semi-persistent wandering quadratic of a preclustering $\vec{a}$ of a decomposition $\vec{f}$, then $\vec{b}$ given by $b_j := a_j-1$ (respectively, $b_j := a_j+1$) and $b_i = a_i$ for all $i \neq j$ is another preclustering of $\vec{f}$. For any cleanup $(\vec{h}, \vec{L})$ of $\vec{a}$, there is a cleanup $(\vec{g}, \vec{M})$ of $\vec{b}$ with $g_i = h_i$ for $i \neq j$ (respectively, $i \neq j+1$) and with $M_i = L_i$ for $i \neq j-1, j$ (respectively, for $i \neq j, j+1$). In particular, $\vec{a}$ and $\vec{b}$ have the same gates at all $j' \neq j$. \end{lemma}

\begin{proof}
 First, we show that all $\vec{f}_{[b_{i}, b_{i-1})}$ are clusters.
 One of these is the concatenation of a cluster of $\vec{a}$ and an adjacent semi-persistent wandering quadratic of $\vec{a}$, so it is a cluster by Remark~\ref{waquark}.
 Another is a cluster of $\vec{a}$ that lost a semi-persistent wandering quadratic, so it is non-empty by semi-persistence and a cluster by Remark~\ref{cluster-rk}.
 The rest are clusters of $\vec{a}$.

 The rest of the proof is the same straightforward exercise in matching Remark~\ref{linrelquadrk} with the definition of gates as in the second part of Remark~\ref{waquark}.\end{proof}

The next lemma somewhat justifies the terminology ``semi-persistent''.

\begin{lemma} \label{one-way-gates-persist}
Suppose that two preclusterings $\vec{a}$ and $\vec{b}$ of the same decomposition $\vec{f}$ only differ at one place $j$ and only by $1$, that is,
 $$k = a_r = b_r >  \ldots > a_{j+1} = b_{j+1} > b_j = a_j + 1 > a_j > a_{j-1} = b_{j-1} > \ldots > a_0=b_0=0.$$
Then $\vec{a}$ and $\vec{b}$ agree on whether the contested factor $f_{b_j}$ is a wandering quadratic; and if it is, they also agree on whether the gate at $j$ is one-way or two-way. Of course, if the gate is one-way, it goes in different directions for the two preclusterings. \end{lemma}

\begin{proof}
Since $a_{j+1} = b_{j+1} > b_j = a_j + 1 > a_j$, the contested factor $f_{b_j}$ is semi-persistent if it is a wandering quadratic. This keeps other non-trivial linear factors of the cleanup from interfering.
This proof is another straightforward exercise in matching Remark~\ref{linrelquadrk} with the definition of gates. \end{proof}

We now return to the question of fusing and overlapping clusters. The next Lemma~\ref{eat-quadratic} serves two purposes. First, it describes a way for the concatenation of a \textsf{C} cluster and a \textsf{C}-free cluster to be itself a cluster. The following Lemma~\ref{no-eat-two-quadratic} asserts that this is the only way. Second, this Lemma~\ref{eat-quadratic} states that if the concatenation of two clusters is not a cluster, but a quadratic factor can enter one cluster from the other, a gate in the correct direction must be present at this cluster boundary in the original preclustering, and a gate in the other direction is present in the new preclustering; thus, no other quadratic cannot follow this one. A similar result is mentioned
on page~4 of~\cite{MZ} but is not explicitly stated as a theorem in the text.

\begin{lemma} \label{eat-quadratic}
 Fix a preclustering $\vec{a}$ of a decomposition $\vec{f}$, and suppose that the concatenation of a \textsf{C} cluster and an adjacent factor $f_i$ is a cluster. Then either $f_i$ is a wandering quadratic, or it comes from another \textsf{C} cluster, and the concatenation of these two clusters is itself a cluster. \end{lemma}

\begin{proof}
To state the lemma more precisely and less readably, let $c := a_{j+1}$ and $b:= a_j$ and $a := a_{j-1}$ to lighten notation. The lemma then says:

 If $\vec{[c,b)}$ is a \textsf{C} cluster and the concatenation $\vec{f}_{[c, b-1)}$ of it and the next factor $f_b$ is a cluster, then one of the following happens.\begin{enumerate}
 \item The whole concatenation $\vec{f}_{[c, a)}$ of the two clusters if itself a cluster; the other cluster $\vec{[b,a)}$ either is a \textsf{C} cluster, or has only one (quadratic) factor so $b = a+1$.
 \item The factor $f_b$ is a semi-persistent wandering quadratic, and Lemma~\ref{one-way-gates-persist} applies to
   the original preclustering and the preclustering
 $k > \ldots > c > b-1 > a > \ldots >0$.
 \end{enumerate}

 Similarly, if $\vec{f}_{[b,a)}$ is a \textsf{C} cluster and the concatenation $\vec{f}_{[b+1, a)}$ of it and the next factor $f_{b+1}$ is a cluster, then either the whole $\vec{f}_{[c, a)}$ is a cluster and the other cluster $\vec{f}_{[c, b)}$ was a \textsf{C} cluster or a single quadratic factor; or $f_{b+1}$ is a wandering quadratic next to a one-way gate in both preclusterings.

 The new factor $f_b$ or $f_{b+1}$ joining the \textsf{C} cluster must be type \textsf{C} or quadratic. A type \textsf{C} factor must come from a \textsf{C} cluster, and then the whole $\vec{f}_{[c, a)}$ is a cluster by Lemma~\ref{C-cluster-fuse} and Proposition~\ref{twodoork}. For the rest of the proof, we assume that the new factor is quadratic, and a wandering quadratic by the first part of Remark~\ref{waquark}.

 If the gate next to this wandering quadratic is two-way, by Proposition~\ref{twodoork} the other cluster must be a \textsf{C} cluster and the whole $\vec{f}_{[c, a)}$ is a cluster.

 If the new quadratic factor is a fake wandering quadratic (that is, a whole cluster), the whole $\vec{f}_{[c, a)}$ is precisely the thing assumed to be a cluster in the first place.
 \end{proof}

\begin{lemma} \label{no-eat-two-quadratic}
Fix a preclustering $\vec{a}$ of a decomposition $\vec{f}$, and suppose that the concatenation of a \textsf{C} cluster and more than one factor of a neighboring cluster is itself a cluster. Then the concatenation of these two clusters is itself a cluster, and the other cluster is a \textsf{C} cluster.\end{lemma}

\begin{proof}
Again, let $c := a_{j+1}$, $b:= a_j$, and $a := a_{j-1}$ to lighten notation.

Suppose the new factors $f_{b}, \ldots, f_{b'+1}$ are to the right of the \textsf{C} cluster $\vec{f}_{[c,b)}$, all inside the next cluster $\vec{f}_{[b,a)}$. The other case, when the new factors are to the left of the \textsf{C} cluster, is essentially identical.

Note that $\vec{f}_{[c,d)}$ is a \textsf{C} cluster for any $d$ with $b \geq d \geq b'$, because it sits inside a \textsf{C} cluster, and has a \textsf{C} cluster sitting inside it.

In particular, the concatenation $\vec{f}_{[c,b-1)}$ of the \textsf{C} cluster $\vec{f}_{[c,b)}$ and one factor $f_b$ is a cluster, so Lemma~\ref{eat-quadratic} applies. If the other cluster $\vec{f}_{[b,a)}$ is a \textsf{C} cluster and the concatenation $\vec{f}_{[c,a)}$ is a cluster, we are done.

 Otherwise, $f_b$ is a wandering quadratic and we work toward a contradiction.
Since more than one factor from $\vec{f}_{[b,a)}$ does something in the hypothesis of this lemma, $f_b$ cannot be fake and must be a semi-persistent quadratic.
 Thus the new preclustering $k = a_r > \ldots > a_{j+1} > a_j -1 > a_{j-1} > \ldots >a_0=0$ from Lemma~\ref{one-way-gates-persist} has a one-way left-to-right gate at $j$.
More than one factor from $\vec{f}_{[b,a)}$ joins $\vec{f}_{[c,b)}$, so $\vec{f}_{[c, b-2)}$ must also be a cluster.
If $f_{b-1}$ is quadratic, then by Remark~\ref{waquark} there must be a right-to-left gate at $j$, a contradiction.
 If $f_{b-1}$ is not quadratic, then Lemma~\ref{C-cluster-fuse} implies that there is a two-way gate at $j$, also a contradiction.
\end{proof}

\begin{lemma} \label{cluster-overlap}
Fix a decomposition $\vec{f}$.
If two clusters overlap, that is, $\vec{f}_{[d,b)}$ and $\vec{f}_{[c,a)}$ are both clusters with $d > c > b > a$, then either the whole $\vec{f}_{[d,a)}$ is a cluster, or one of the clusters is a \textsf{C} cluster and
$c=b+1$ and $f_{b+1} = f_c$ is a semi-persistent wandering quadratic of every preclustering $\ldots d > c > a > \ldots$ of $\vec{f}$. \end{lemma}

\begin{proof}
Consider the three pieces $\vec{f}_{[d,c)}$, $\vec{f}_{[c,b)}$, and $\vec{f}_{[b,a)}$. If both original clusters $\vec{f}_{[d,b)}$ and $\vec{f}_{[c,a)}$ are \textsf{C}-free clusters, then all three of those pieces are \textsf{C}-free clusters, and Lemma~\ref{land-cluster-fuse} forces the whole $\vec{f}_{[d,a)}$ to be a \textsf{C}-free cluster.

If both $\vec{f}_{[d,b)}$ and $\vec{f}_{[c,a)}$ are \textsf{C} clusters and the middle piece $\vec{f}_{[c,b)}$ is a \textsf{C} cluster, then the whole $\vec{f}_{[d,a)}$ is a cluster by Lemma~\ref{C-cluster-fuse}.

The only possibility left is that at least one of $\vec{f}_{[d,b)}$ and $\vec{f}_{[c,a)}$ is a \textsf{C} cluster, but the middle piece $\vec{f}_{[c,b)}$ is not a \textsf{C} cluster. We handle the case when $\vec{f}_{[d,b)}$ is a \textsf{C} cluster; the other case is identical.

Since $\vec{f}_{[c,b)}$ sits inside a \textsf{C} cluster but is not a \textsf{C}-cluster itself, it must be linearly related to $C_{2^M}$ for some $M$. Thus the leftover $f_{[d,c)}$ of the \textsf{C} cluster $f_{[d,b)}$ is still a \textsf{C} cluster. If $M \geq 2$, this \textsf{C} cluster $f_{[d,c)}$ absorbs more than one factor from the adjacent cluster $f_{[c,a)}$, so by Lemma~\ref{no-eat-two-quadratic} $f_{[c,a)}$ is again a \textsf{C} cluster and the whole $\vec{f}_{[d,a)}$ is a cluster. The case $M=1$ is precisely the last option in the conclusion of this lemma.
\end{proof}

The next definition characterizes preclusterings that have as few clusters as possible. The next few results build up to show that this minimal number of clusters, as well as the presence or absence of a gate at each cluster boundary, are properties of a polynomial, independent of decomposition; and that cluster boundaries can only change by $1$ between two clusterings, and then only because of quadratic factors.

\begin{Def}\label{cluster-def}
A preclustering $\vec{a}$ of $\vec{f}$ is a \emph{clustering} if \begin{itemize}
\item the concatenation of any two adjacent clusters is not a cluster; and
\item no cluster consists of only wandering quadratics.
\end{itemize}
(The second part is only relevant for clusters with exactly two factors.)
\end{Def}

\begin{Rk} \label{devour}
The second part of the definition of ``clustering'' can be replaced by \begin{itemize}
\item No cluster with exactly two factors can be devoured by adjacent clusters: if $a_i = a_{i-1}+2$, then $\vec{f}_{[a_{i+1}, a_i -1)}$ and $\vec{f}_{[a_i -1, a_{i-2})}$ are not both clusters.\end{itemize}
It follows from the first part of the definition of ``clustering'' that if this condition is violated, both factors of the violating cluster must be wandering quadratics. The advantage of this formulation is that it is entirely in terms of which $\vec{f}_{[c,b)}$ are clusters.
\end{Rk}

\begin{lemma} \label{cl-linrel}
If $\vec{a}$ is clustering of a decomposition $\vec{f}$, then for any linear $M$ and $N$, $\vec{a}$ is also a clustering of the decomposition $\vec{g} := (N \circ f_k, f_{k-1}, \ldots, , f_2, f_1 \circ M)$. \end{lemma}

\begin{proof}
By Lemma~\ref{precl-linrel}, $\vec{a}$ is a preclustering of $\vec{g}$. As noted in the proof of that lemma, for any $b$ and $c$, $\vec{g}_{[c,b)}$ is a cluster if and only if $\vec{f}_{[c,b)}$ is a cluster. As noted in Remark~\ref{devour}, this suffices to show that $\vec{a}$ is a clustering of $\vec{g}$.\end{proof}

\begin{Rk} \label{wandering-in-clustering}
By Proposition~\ref{twodoork}, there are no two-way gates in a clustering.
Thus, between two \textsf{C}-free clusters of a clustering, there are neither gates nor wandering quadratics.
There are no fake wandering quadratics in a clustering.
Thus, all gates in a clustering are semi-persistent one-way gates.
\end{Rk}

\begin{lemma} \label{clusteringsexist} If all factors of $\vec{f}$ are swappable, then $\vec{f}$ admits a clustering. \end{lemma}
\begin{proof} We already know that $\vec{f}$ admits a preclustering, so take one, and induct on the number of clusters in it: with Remark~\ref{devour}, it is clear how to rectify a violation of either of the two extra requirements of a clustering, and both decrease the number of clusters.\end{proof}

\begin{lemma}\label{imp-swap-inside}
Suppose that $\vec{a}$ is a clustering of $\vec{f}$, that $i$ is not a cluster boundary of $\vec{a}$, and that $\vec{g} := t_i \star \vec{f}$ is defined. Then $\vec{a}$ is also a clustering of $\vec{g}$. \end{lemma}
\begin{proof}
By Proposition~\ref{swap-inside}, $\vec{a}$ is a preclustering of $\vec{g}$. The rest is trivial. \end{proof}

\begin{lemma} \label{move-quad}
If $f_{a_j}$ (respectively, $f_{a_j +1}$) is a wandering quadratic of a clustering $\vec{a}$ of a decomposition $\vec{f}$, then $\vec{b}$ given by $b_j := a_j-1$ (respectively, $b_j := a_j+1$) and $b_{j'} = a_{j'}$ for all $j' \neq j$ is another clustering of $\vec{f}$. For each $j'$, either both clusterings have a gate at $j'$, or neither has a gate at $j'$.
\end{lemma}

\begin{proof}
We first show that $\vec{b}$ is a preclustering of $\vec{f}$, then that $\vec{b}$  has gates in the same places as $\vec{a}$, and then that $\vec{b}$ is a clustering.

By Remark~\ref{wandering-in-clustering}, the clustering $\vec{a}$ has no two-way gates and no fake wandering quadratics, so $f_{a_j}$ (respectively, $f_{a_j +1}$) is a semi-persistent wandering quadratic of $\vec{a}$. Then $\vec{b}$  is a preclustering of $\vec{f}$ by Lemma~\ref{move-one-per-wa-qua}, and has the same gates as $\vec{a}$ by Lemmas~\ref{move-one-per-wa-qua} and~\ref{one-way-gates-persist}. It remains to show that $\vec{b}$ is a clustering.

We first show that for every $i$, the factor $f_i$ is a wandering quadratic of $\vec{a}$ is and only if it is a wandering quadratic of $\vec{b}$. For $i = a_j$ (respectively, $i = a_j +1$), this is Lemma~\ref{one-way-gates-persist}. For all other cluster boundaries, this is Lemma~\ref{move-one-per-wa-qua}. It is clear that the two factors adjacent to $f_{a_j}$ (respectively, $f_{a_j +1}$) are not wandering quadratics of $\vec{b}$, so it remains to show that they are not wandering quadratics of $\vec{a}$. By Remark~\ref{wandering-in-clustering}, the clustering $\vec{a}$ has no two-way gates and no fake wandering quadratics, so $f_{a_j +1}$ (respectively, $f_{a_j }$), the factor on the other side of the cluster boundary, is not a wandering quadratic of $\vec{a}$. Since the $j$th (respectively, $(j+1)$st) cluster of $\vec{a}$ does not consist of two wandering quadratics, $f_{a_j -1}$ (respectively, $f_{a_j +2}$) is also not a wandering quadratic of $\vec{a}$. From now on we say ``wandering quadratic'' without specifying $\vec{a}$ or $\vec{b}$.

Let us now verify the two parts of the definition of clustering for the two clusters of $\vec{b}$ that differ from those of $\vec{a}$. We handle the case when $f_{a_j+1}$ is the wandering quadratic; the other case is analogous.

Two instances of the second part of the definition of clustering need to be verified.
\begin{itemize}
\item Since the $j$th cluster $\vec{f}_{[a_j, a_{j-1})}$ of $\vec{a}$ is not a single wandering quadratic, the $j$th cluster $\vec{f}_{[a_{j+1}, a_{j-1})}$  of $\vec{b}$ does not consist of two wandering quadratics.
\item We have already shown that the rightmost factor $f_{b_j+1} = f_{a_j +2}$ of the $(j+1)$st cluster of $\vec{b}$ is not a wandering quadratic.
\end{itemize}

Three instances of the first part of the definition of clustering need to be verified.
\begin{itemize}
\item The concatenation $\vec{f}_{[a_j+1, a_{j-2})}$ of the $j$th and the $(j-1)$st clusters of $\vec{b}$ is not a cluster because it is the concatenation of the quadratic $f_{b_j} = f_{a_j +1}$ and $\vec{f}_{[a_j, a_{j-2})}$, and this $\vec{f}_{[a_j, a_{j-2})}$ is not a cluster because it is the concatenation of the corresponding clusters of $\vec{a}$.

\item The concatenation of the $(j+1)^\text{st}$
 and the $j^\text{th}$ clusters of $\vec{b}$ is exactly the same as the concatenation of the $(j+1)^\text{st}$ and the $j^\text{th}$ clusters of $\vec{a}$, so it is not a cluster.
\item If the concatenation $\vec{f}_{[a_{j+2}, a_j+1)}$ of the $(j+2)^\text{nd}$ and the $(j+1)^\text{st}$
    clusters of $\vec{b}$ is a cluster, then it overlaps the $(j+1)^\text{st}$ cluster $\vec{f}_{[a_{j+1}, a_j)}$ of $\vec{a}$ in $\vec{f}_{[a_{j+1}, a_j +1)}$. We have already shown that the rightmost factor of this overlap $f_{a_j +2}$ is not a wandering quadratic, so by Lemma~\ref{cluster-overlap} the concatenation $\vec{f}_{[a_{j+2}, a_j)}$ of these overlapping clusters is itself a cluster. But that is also the concatenation of two clusters of $\vec{a}$, which cannot be a cluster.
\end{itemize}
\end{proof}

\begin{prop} \label{bridge-cluster}
Suppose that $\vec{a}$ is a clustering of $\vec{f}$.
If $(f_{a_j+1}, f_{a_j})$ is a cluster and one of $f_{a_j+1}$ and $f_{a_j}$ is not quadratic, then the other one of $f_{a_j+1}$ and $f_{a_j}$ is a semi-persistent wandering quadratic, and one of the clusters $\vec{f}_{[a_{j+1}, a_j)}$ and $\vec{f}_{[a_{j}, a_{j-1})}$ is a \textsf{C} cluster. \end{prop}

\begin{proof}
Lemma~\ref{cluster-overlap} applies to the clusters $\vec{f}_{[a_{j+1}, a_j)}$ and $(f_{a_j+1}, f_{a_j})$ which overlap in $f_{a_j+1}$. If $f_{a_j+1}$ is a wandering quadratic and $\vec{f}_{[a_{j+1}, a_j)}$ is a \textsf{C} cluster, we are done. If $f_{a_j+1}$ is a wandering quadratic and $(f_{a_j+1}, f_{a_j})$ is a \textsf{C} cluster, then $f_{a_j}$ is type \textsf{C}, so $\vec{f}_{[a_{j}, a_{j-1})}$ is a \textsf{C} cluster, and we are done. Otherwise, the whole $\vec{f}_{[a_{j+1}, a_{j-1})}$ is a cluster.

Now Lemma~\ref{cluster-overlap} applies to this $\vec{f}_{[a_{j+1}, a_{j-1})}$ and $\vec{f}_{[a_{j}, a_{j-1})}$, which overlap in $f_{a_j}$. The whole $\vec{f}_{[a_{j+1}, a_{j-1})}$ is the concatenation of two clusters of $\vec{a}$, so it cannot be a cluster. Thus $f_{a_j}$ is a wandering quadratic and one of $\vec{f}_{[a_{j+1}, a_{j-1})}$ and $\vec{f}_{[a_{j}, a_{j-1})}$ is a \textsf{C} cluster. It is easy to see that $(f_{a_j+1}, f_{a_j})$ is a \textsf{C} cluster whenever $(f_{a_j+1}, f_{a_j-1})$ is, so we are done in either case.
\end{proof}

\begin{lemma} \label{extra-cluster-swap} \label{interclusterswap}
If $\vec{a}$ is a clustering of $\vec{f}$ and $t_{a_j} \star \vec{f} = \vec{g}$ is defined, then there is a clustering
 $\vec{b}$ of $\vec{g}$ with $b_{j'} = a_{j'}$ for all $j' \neq j$. At each $j' \neq j$, the clustering $\vec{a}$ of $\vec{f}$ and the clustering $\vec{b}$ of $\vec{g}$ have the same gate(s). As for $b_j$, either \begin{itemize}
\item  $f_{a_j+1}$ is a wandering quadratic of the clustering $\vec{a}$ of $\vec{f}$, and $b_j := a_j +1$, and $\vec{b}$ for $\vec{g}$ has a one-way right-to-left gate at $b_j$; or
\item $f_{a_j}$ is a wandering quadratic of the clustering $\vec{a}$ of $\vec{f}$, and $b_j := a_j -1$, and $\vec{b}$ for $\vec{g}$ has a one-way left-to-right gate at $b_j$.
\end{itemize}
\end{lemma}

\begin{proof}
As noted in Remark~\ref{cluster-rk}, since $t_{a_j} \star \vec{f}$ is defined, $(f_{a_i+1}, f_{a_i})$ is a cluster. As tautological Ritt swaps are not permitted, $f_{a_j+1}$ and $f_{a_j}$ are not both quadratic.
By Proposition~\ref{bridge-cluster} one of $f_{a_i+1}$ and $f_{a_i})$ is a semi-persistent wandering quadratic.
By Lemma~\ref{move-one-per-wa-qua},  $\vec{b}$ is a clustering of $\vec{f}$.
Since $a_j$ is not a cluster boundary of $\vec{b}$, by Lemma~\ref{imp-swap-inside}, $\vec{b}$ is also a clustering of $\vec{g}$.
By the second part of Lemma~\ref{ezdoork}, $\vec{f}$ and $\vec{g}$ have the same gates in the same places with respect to $\vec{b}$.
\end{proof}

With the next lemma and proposition we show that two clusterings of the same decomposition can only differ by putting wandering quadratics into different clusters.

\begin{lemma}  \label{unique-clustering-tech}
 Any two clusterings $\vec{a}$ and $\vec{b}$ of the same decomposition $\vec{f}$ have the same number of clusters. For each $j$ either $b_j = a_j - 1$ (respectively, $b_j = a_j + 1$) and $f_{a_j}$ (respectively, $f_{a_j +1}$) is a wandering quadratic of both, or $b_j = a_j$. At each $j$, either both clusterings have a gate at $j$, or neither has a gate at $j$.
\end{lemma}

\begin{proof}
 Let $\vec{a}$ of length $r$ and $\vec{b}$ of length $s$ be two clusterings of the same decomposition $\vec{f}$.

Intuitively, we start from the right and match clusters of $\vec{b}$ with those of $\vec{a}$ one at a time. More formally, we induct on the number $r$ of clusters in $\vec{a}$. For the base case of the induction, take $r=1$, i.e. the whole decomposition is a single cluster; clearly, no other clustering is possible.

For the induction step, we first match the rightmost clusters of the two clusterings; that is, we show that $k = a_r > \ldots > a_2 > b_1 > b_0 = a_0 = 0$ is another clustering of $\vec{f}$, with gates at the same places as $\vec{a}$.
 If $a_1 = b_1$, we are done.
 Suppose that $a_1 < b_1$. Since the concatenation $\vec{f}_{[a_2, a_0)}$ of the first two clusters of $\vec{a}$ is not a cluster, and $\vec{f}_{[a_2, a_1)}$ is not a single wandering quadratic, this $\vec{f}_{[a_2, a_0)}$ cannot be contained in the first cluster of $\vec{b}$. That is, $b_1 < a_2$. Applying Lemma~\ref{cluster-overlap} to $\vec{f}_{[a_2, a_1)}$ and $\vec{f}_{[b_1, b_0)}$, we see that $b_1 = a_1 +1$ and $f_{b_1}$ is a wandering quadratic of $\vec{a}$.
 Identical reasoning shows that if $a_1 > b_1$, then $a_1 = b_1 +1$ and $f_{a_1}$ is a wandering quadratic of $\vec{a}$. In any case, Lemma~\ref{move-quad} makes $k = a_r > \ldots > a_2 > b_1 > b_0 = a_0 = 0$ another clustering of $\vec{f}$, with gates at the same places as $\vec{a}$.

Now $\vec{b}' := (b_s-b_1, \ldots, b_2-b_1, 0 = b_1-b_1)$ and $\vec{a}' := (a_r-b_1, \ldots, a_2-b_1, 0 = b_1-b_1)$ are both clusterings of $\vec{f}_{[k, b_1)}$, so by induction hypothesis, $s-1 = r-1$, and $\vec{b}'$ and $\vec{a}'$ have gates at the same boundaries, which finishes the proof.
\end{proof}

\begin{prop} \label{unique-clustering-ntech}
Given a clustering $\vec{a}$ of $\vec{f}$, another tuple $\vec{b}$ of the same length as $\vec{a}$ is a clustering of $\vec{f}$ if and only if for each $j$ where $a_j \neq b_j$ there is a wandering quadratic $f_i$ such that $\{ a_j, b_j \} = \{ i, i+1 \}$.
\end{prop}

\begin{proof}
 The ``if'' follows from applying Lemma~\ref{move-quad} at each $j$ where $a_j \neq b_j$.
The ``only if'' follows immediately from Lemma~\ref{unique-clustering-tech}.
\end{proof}

\begin{theorem}\label{swapsferries}
 The number of clusters in a clustering, the kind (\textsf{C} or \textsf{C}-free) of each cluster, and the presence of a gate between the $j$th and the $(j+1)$st clusters are properties of the polynomial, independent of decomposition, clustering, and cleanup. \end{theorem}

\begin{proof}
 It suffices to show that these are invariant under Ritt swaps, and we have in fact already done so.
 Suppose that $\vec{f}$ and $\vec{g} := t_i \star \vec{f}$ are two decompositions of the same polynomial, and $\vec{a}$ is a clustering of $\vec{f}$.

If $i \neq a_j$ for any $j$, Lemma~\ref{imp-swap-inside} shows that $\vec{a}$ is also a clustering of $\vec{g}$, and Lemma~\ref{ezdoork} shows it has all the same gates in the same places. Wandering quadratics are unchanged from $\vec{f}$ to $\vec{g}$ unless the swap brings a quadratic to a gate, or moves a quadratic inside a cluster away from the gate.

If $i = a_j$ for some $j$, then Lemma~\ref{interclusterswap} shows that Lemma~\ref{eat-quadratic} applies and gives a clustering of $\vec{g}$ with $a_j$ moved left or right by one; in any case, gates in this new clustering are the same as in the old clustering, except that the one at $j$ switches direction.
\end{proof}

 The next few results begin to apply the technical machinery around clusterings to the action of the Ritt monoid.

\begin{lemma} \label{double-jump}
If $t_{i+1} \star \vec{f}$ and $t_i \star \vec{f}$ are both defined, then either
$(f_{i+2}, f_{i+1}, f_i)$ is a cluster, or $f_{i+1}$ is a wandering quadratic (in any clustering of $\vec{f}$). \end{lemma}

\begin{proof}
As remarked in Remark~\ref{cluster-rk}, both $(f_{i+2}, f_{i+1})$ and $( f_{i+1}, f_i)$ must be clusters, and Lemma~\ref{cluster-overlap} finishes the proof. \end{proof}

\begin{lemma} \label{forfunda}
If $t_i t_{i+1} t_i \star \vec{f}$ (respectively, $t_{i+1} t_i t_{i+1} \star \vec{f}$) is defined, then $(f_{i+2}, f_{i+1}, f_i)$ is a cluster.
\end{lemma}

\begin{proof}
Note that in either sequence, each factor swaps with every other factor. So as long as at least one of the two is defined, all of $f_i$, $f_{i+1}$, and $f_{i+2}$ are swappable, and no more than one is quadratic.

Lemma~\ref{double-jump} applies to the first intermediate decomposition $\vec{g} := t_i \star \vec{f}$ (respectively, $\vec{g} := t_{i+1} \star \vec{f}$). If $\vec{g}$ is a cluster, we are done. Otherwise, the second intermediate decomposition $\vec{h} := t_{i+1} t_i \star \vec{f}$ (respectively, $\vec{h} := t_i t_{i+1} \star \vec{f}$) has a cluster boundary between $h_{i+1}$ and $h_i$ (respectively, $h_{i+2}$ and $h_{i+1}$), and neither one of these is quadratic, so the last swap $t_i \star \vec{h}$ (respectively, $t_{i+1} \star \vec{h}$) is not defined.
\end{proof}

 Our fundamental Theorem~\ref{fundamentallemma} is now an easy corollary. \label{proofoffunda}

\begin{proof} \textbf{This is the proof of Theorem~\ref{fundamentallemma}}.
We need to show that $t_{i+1} t_i t_{i+1} \star \vec{f}$ is defined if and only if $t_i t_{i+1} t_i \star \vec{f}$ is defined, and they are equal when defined.

By Lemma~\ref{forfunda}, the whole $(f_{i+2}, f_{i+1}, f_i)$ must be a cluster, so Ritt swaps can be witnessed by identity linear factors, and the result is immediate. \end{proof}

Let us prove two more statements of this flavour, with a view towards normal forms.

\begin{prop} \label{one-traverse}
 Suppose that $\vec{h} := (t_{k-1} t_{k-2} \ldots t_1) \star \vec{f}$ is defined, and let $\vec{a}$ be a clustering of $\vec{f}$ with $r >1$ clusters.
 Then $f_1$ is quadratic and there are (one-way) right-to-left gate at every $j \neq 0, r$.
 Furthermore, $\vec{b}$ given by $b_j = a_j -1$ for all $j \neq 0, r$ is a clustering of $\vec{h}$, which has (one-way) left-to-right gates at every $j \neq 0, r$.
\end{prop}

\begin{proof}
For $0 \leq i <k$, let $\vec{f}^i := (t_i t_{i-1} \ldots t_1) \star \vec{f}$, and define clustering $\vec{a}^i$ of $\vec{f}^i$ by setting $\vec{a}^0 := \vec{a}$ and continuing inductively as follows.
If $i$ is not a cluster boundary of $\vec{a}^{i-1}$, then Lemma~\ref{imp-swap-inside} applies to $t_i \star \vec{f}^{i-1} = \vec{f}^i$, so $\vec{a}^i := \vec{a}^{i-1}$ works.
If $i = a^{i-1}_j$ is a cluster boundary of $\vec{a}^{i-1}$, then Lemma~\ref{extra-cluster-swap} applies to $t_i \star \vec{f}^{i-1} = \vec{f}^i$, making $f^{i-1}_i$ quadratic and making $\vec{a}^i$ with $a^i_j = a^{i-1}_j -1$ and $a^i_{j'} = a^{i-1}_{j'}$ for all $j' \neq j$ a clustering of $\vec{f}^i$.
Since $\vec{f}$ has more than one cluster, the second possibility must occur at least once: $f^{i-1}_i$ quadratic is quadratic for some $i$. Since $f_1$ becomes (in the sense of Remark~\ref{costumes}) this $f^{i-1}_i$ via the product $t_{i-1} t_{i-2} \ldots t_1$ of Ritt swaps, this makes $f_1$ quadratic.
 It is now clear that $\vec{b} := \vec{a}^{k-1}$ works.\end{proof}

\begin{lemma} \label{two-traverses}
 Suppose that $ k \geq 3$, let $v_1 := t_{k-1} t_{k-2} \ldots t_2$ and $v_2 :=t_{k-2} t_{k-3} \ldots t_1$,
and suppose that $v_2 v_1 \star \vec{f}$ is defined.
Then the whole $\vec{f}$ is a cluster,
unless $k=3$ and
$f_3$ is quadratic.
\end{lemma}

\begin{proof}
Let $\vec{g} := v_1 \star \vec{f}$, and note that Proposition~\ref{one-traverse} applies to both
(1) $v_i \star \vec{f}_{[k, 1)} = \vec{g}_{[k, 1)}$ and (2) $v_2 \star \vec{g}_{[k-1, 0)}$.
If $f_1 = g_1$ is not quadratic, (2) makes $\vec{g}_{[k-1, 0)}$ a cluster.
If $f_2$ is not quadratic, (1) makes both $\vec{g}_{[k, 1)}$ and $\vec{h}_{[k, 1)}$ a cluster.
If either $f_1$ or $f_2$ is quadratic, then none of the other factors $f_i$ with $i \geq 3$ are quadratic, since we do not allow tautological Ritt identities.

We now treat the four cases separately.

\textit{Case 1:} If neither $f_1$ nor $f_2$ is quadratic, then Lemma~\ref{cluster-overlap} applies to the clusters
$\vec{g}_{[k-1, 0)}$ and $\vec{g}_{[k, 1)}$. If $\vec{g}$ is a cluster, then so is $\vec{f}$. Otherwise, the overlap $\vec{g}_{[k-1, 1)}$ is a single quadratic $g_2$, so $k = 3$, $v_1 = t_2$, and $f_3$ is quadratic.

\textit{Case 2:} Suppose $f_1=g_1$ is not quadratic but $f_2$, and therefore $g_k$, is. Now $(g_{k-1}, \ldots, g_2, g_1=f_1)$ is a cluster. On the other hand, $t_{k-1} \vec{g} =  t_{k-2} \ldots t_2 \star \vec{f}$ is defined, so $(g_k, g_{k-1})$ is also cluster, and $g_{k-1}$ is not quadratic. Then by Lemma~\ref{cluster-overlap} the whole $\vec{g}$ (and, therefore, $\vec{f}$) is a cluster as desired.

\textit{Case 3:} Suppose $f_1=g_1$ is quadratic but $f_2$, and therefore $g_k$, is not. Now $(g_k, \ldots, g_2)$ is a cluster. But $v_2 \star \vec{g}$ is defined, so $t_1 \star \vec{g}$ is defined, so $(g_2, g_1)$ is a cluster. Since $g_2$ is not quadratic, by Lemma~\ref{cluster-overlap} the whole $\vec{g}$ (and, therefore, $\vec{f}$) is a cluster as desired.

\textit{Case 4:} Finally, suppose that both $f_1$ and $f_2$ are quadratic. If $(f_k, \ldots, f_3)$ is not a cluster, according to Proposition \ref{one-traverse} there must be a right-to-left gate at every boundary between clusters inside there, which becomes a one-way left-to-right in the corresponding place in $(g_{k-1}, \ldots, g_2)$. That, according to the same proposition, makes it impossible for $v_2 \star \vec{g}$ to be defined. So $(f_k, \ldots, f_3)$ is a cluster. Since $t_2 \star \vec{f}$ is defined, $(f_3, f_2)$ is also a cluster. Since $f_3$ is not quadratic, this makes $(f_k, \ldots, f_2)$ a cluster. Then $(g_k, \ldots, g_2)$ is also a cluster. Since $t_1 \star \vec{g}$ is defined, $(g_2, g_1))$ is a cluster. Since $g_2$ is not quadratic, this means the whole $\vec{g}$ is a cluster. So $\vec{f}$ is a cluster.\end{proof}

\section{Canonical forms}
\label{canformsec} \label{sect43}

Much of this section is devoted to using syntactic operations on words in the Ritt monoid $\rittmonoid_k$ that appear in Theorem~\ref{fundamentallemma} to show that any decomposition of a polynomial may be obtained from any other by a sequence of Ritt swaps \emph{in a particular canonical order}. The words in the Ritt monoid $\rittmonoid_k$ corresponding to such sequences are said to be in a canonical form.

If one thinks of permuting factors as putting them  in a particular order, then our first canonical form roughly corresponds to an insert-sort, and the second one to a merge-sort.
While it is well known that every permutation is represented by a sequence of transpositions of each of these forms, we could not find a reference in literature for the more refined results taking into account the irreversibility of operation (1) in Remark~\ref{synoprk} below.

For each of our two canonical forms, we show (see Proposition~\ref{firstcanprop} and Proposition~\ref{2ndcanprop}) that for any word $w \in \rittmonoid_k$ there is a word $\hat{w}$ of this canonical form such that whenever $w \star \vec{f}$ is defined, $\hat{w} \star \vec{f} = w \star \vec{f}$. For example, for $w = t_i t_i$, we set $\hat{w}$ to be the empty word.

While it is convenient to speak of the factors of $\vec{f}$ in the statements and proofs of intermediate results, the canonical word $\hat{w}$ ultimately only depends on $w$ and works for all $\vec{f}$.

\begin{Rk} \label{synoprk}
Recall the three syntactic operations on words in the Ritt monoid $\rittmonoid_k$ from Theorem~\ref{fundamentallemma}.
\begin{enumerate}
\item Delete subword $t_i t_i$.
\item Replace subword $t_i t_j$ by $t_j t_i$ for non-consecutive $i$ and $j$.
\item Replace subword $t_{i+1} t_i t_{i+1}$ by $t_i t_{i+1} t_i$, or vice versa.
\end{enumerate}
 Operations (2) and (3) are reversible, while (1) is not.
 Operation (1) decreases the length of the word, while (2) and (3) leave it the same.
\end{Rk}

\begin{Rk}
 If a word $v$ is obtained from a word $w$ by operations (1), (2) and (3) above, then they represent the same permutation (see Definition~\ref{defpermrep}), the length of $v$ is less than or equal to the length of $w$, and for any decomposition $\vec{f}$, if $w \star \vec{f}$ is defined, the $v \star \vec{f} = w \star \vec{f}$. It may be that $v \star \vec{f}$ is defined while $w \star \vec{f}$ is not.

 If $v$ and $w$ also have the same length, or, equivalently, if one was obtained from the other by operations (2) and (3) alone, then $v \star \vec{f} = w \star \vec{f}$ for all $\vec{f}$.
\end{Rk}

These observations motivate the following definitions.

\begin{Def} \label{defapporx}
 If two words $v, w \in \rittmonoid_k$ can be obtained from each other by operations (2) and (3) above, we write $w \simeq v$.

 A word $w \in \rittmonoid_k$ \emph{is length-minimal} if no strictly shorter word $v$ may be obtained from $w$ by operations (1), (2) and (3) above.
\end{Def}

\begin{Rk}
This notion of equivalent words, only used in this section, is stronger than $\approx$ in Definition~\ref{equicorrdef}. For example, they disagree on the pair $t_i t_j t_j t_j$ and $t_i t_i t_i t_j$ for $j \neq i-1, i, i+1$. By Theorem~\ref{fundamentallemma}, $v \simeq w$ implies $v \approx w$. \end{Rk}

Instead of inducting on the length of $w$, we begin most proofs in this section with replacing $w$ by some length-minimal word that can be obtained from it by operations (1) - (3), and then reach a contradiction every time we get a chance to cancel $t_i t_i$.

\begin{Rk} Let $t, u, v$ be words in the Ritt monoid. If $w' = tuv$ is length-minimal, then $u$ is length-minimal. \end{Rk}

We use the same interval-subscript notation for long sequences of $t_i$ as we did for long sequences of $f_i$.
These intervals may be increasing or decreasing, and open or closed on either end.

\begin{Def}
If $a < b$, then $t_{(a,b]} = t_{[a+1,b]} = t_{(a,b+1)} = t_{[a+1,b+1)} := t_{a+1} t_{a+2} \ldots t_b$. A word of this form is a \emph{left-to-right transit}.\\
If $a > b$, then $t_{(a,b]} = t_{[a-1,b]} = t_{(a,b-1)} = t_{[a-1,b-1)} := t_{a-1} t_{a-2} \ldots t_b$. A word of this form is a \emph{right-to-left transit}.\\
If $a=b$, then $t_{[a,b]} := t_a = t_b$, while $t_{(a,b]} = t_{[a,b)} = t_{(a,b)}$ is the empty word.\\
\end{Def}

Transits are so named because, for example, in $t_{(a,b]} \star \vec{f} = \vec{g}$ one factor $f_b$ ``travels'' left from its original $b$th position to become (in the sense of Remark~\ref{costumes}) the factor $g_a$ in $a$th position in $\vec{g}$.

The following generalizations of operation (3) are useful. The last one says that if two adjacent factors $f_{a-1}$ and $f_{a-2}$ travel some number of steps to the left and then switch places, they could just as well have switched places first, and traveled later.

\begin{lemma} \label{teenylem}
\begin{itemize}
\item  if $r+1 > r \geq s$, then
$t_r t_{[r+1,s]} \simeq t_{[r+1,s]} t_{r+1}$
\item if $p > r \geq s$, then
 $t_r t_{[p,s]} \simeq t_{[p,s]} t_{r+1}$
\item $t_{[b, a-1]} t_{[b,a]} \simeq t_{[b-1, a_1]} t_{[b,a]} t_{a-1}$
\end{itemize} \end{lemma}

\begin{proof}
 For (1), $t_r t_{r+1} t_r \simeq t_{r+1} t_r t_{r+1}$, and then $t_{r+1}$ commutes with $t_{r-1}$ through $t_s$.

 For (2), note that $t_r$ commutes with $t_p$ through $t_{r+2}$ and then (1) applies.

 We prove (3) by induction on $b-a$. The base case $b=a$ is exactly operation (3) above. For the induction step,
 $$t_{[b, a-1]} t_{[b,a]} = t_b t_{b-1} t_{[b-2, a-1]} t_b t_{[b-1, a]}
 \simeq t_b t_{b-1} t_b t_{[b-2, a-1]} t_{[b-1, a]} \simeq$$
$$ \simeq t_{b-1} t_b t_{b-1} t_{[b-2, a-1]} t_{[b-1, a]} = t_{b-1} t_b t_{[b-1, a-1]} t_{[b-1, a]} =: u$$
 Applying the inductive hypothesis to $t_{[b-1, a-1]} t_{[b-1, a]}$, we get
$$ u \simeq t_{b-1} t_b t_{[b-2, a-1]} t_{[b-1, a]} t_{a-1} \simeq t_{b-1}  t_{[b-2, a-1]} t_b t_{[b-1, a]} t_{a-1}.$$
\end{proof}

A sequence of Ritt swaps in the first canonical form is a sequence of right-to-left transits whose action resembles an insert-sort: having arranged $f_k$ through $f_{i+1}$ in the right order, this sequence \emph{inserts} $f_i$ in the required $a_i$th place among $f_k$ through $f_{i+1}$, and then proceeds to deal with $f_{i-1}$, and so on, until all factors are arranged as wanted.

\begin{Def}
A word $w \in \rittmonoid_k$ is in \emph{first canonical form} if it has the form
$w = t_{(a_1, 1]} t_{(a_2, 2]} \ldots t_{(a_{k-1}, k-1]}$ for some
 $a_1, a_2, \ldots, a_{k-1}$ such that $i \leq a_i \leq k$ for each $i$.

A word $w \in \rittmonoid_k$ is in \emph{reverse first canonical form} if it has the form
$w = t_{[a_k , k)}  t_{[a_3 , 3)} t_{[a_2 , 2)}$ for some
 $a_1, a_2, \ldots, a_{k-1}$ such that $1 \leq a_i \leq i$ for each $i$.
\end{Def}

\begin{Rk} \label{alt-1stcan}
Omitting those transits which are empty words gives an alternative formulation of first canonical for as
$ t_{[a_m, b_m]} t_{[a_{m-1}, b_{m-1}]} \ldots t_{[a_1, b_1]}$ with $a_i \geq b_i$ for all $i$, and $b_m < b_{m-1} < \ldots < b_1$. \end{Rk}

 The three syntactic operations in Remark~\ref{synoprk} are invariant under switching left and right, so anything that holds for the first canonical form also holds, mutatis mutandis, for the reverse first canonical form.

The following lemma straightens out two right-to-left transits that occur in the wrong order.

\begin{lemma} \label{lem-for-1stcan}
 If $a \geq b$ and $c \geq d$ and $w' := t_{[a,b]} t_{[c,d]}$ is length-minimal, then it is equivalent to a single right-to-left transit, or to $t_{[a',b']} t_{[c',d']}$ for some $a',b',c',d'$ such that $a' \geq b'$ and $c' \geq d'$ and $b' < d'$.
  \end{lemma}

\begin{proof}
If $b<d$, then $w'$ is already of the desired form. So assume $b \geq d$.

Now compare $b$ and $c$:\begin{itemize}
 \item If $b > c+1$, then $\hat{w} =  t_{[c,d]} t_{[a,b]}$ works, because in this case each $t_i$ in the first transit of $w'$ commutes with each $t_j$ in the second transit, and $b > c+1 >d$.
 \item If $b=c+1$, $w'$ is already a single transit as wanted.
 \item If $b=c$, operation (1) shortens the word $w'$ contradicting length-minimality.
 \item  This leaves the case where $c > b \geq d$ for which we use Lemma~\ref{teenylem} and another case-out.
\end{itemize}

 So $c > b \geq d$; compare $a$ and $c$: \begin{itemize}
\item If $a < c$, then Lemma~\ref{teenylem} can be applied to each $t_i$ for $a \geq i \geq b$ giving
  $w' \simeq t_{[c,d]} t_{[a+1, b+1]} =: \hat{w}$, of the desired form because $d \leq b$ implies $d < b+1$.
 \item If $a \geq c$, Lemma~\ref{teenylem} can still be applied to each $t_i$ for $c-1 \geq i \geq b$ giving
$$w' = t_{[a,b]} t_{[c,d]}= $$
$$= t_{[a,c]} t_{[c-1, b]} t_{[c,d]} \simeq $$
$$\simeq t_{[a,c]} t_{[c,d]} t_{[c,b=1]}$$
contradicting length-minimality, as $t_c t_c$ sits in the middle of $t_{[a,c]} t_{[c,d]}$.
\end{itemize}
\end{proof}

\begin{prop}
\label{firstcanprop}
 For every $w \in \rittmonoid_k$ there exists a unique $\hat{w} \in \rittmonoid_k$ in first canonical form which represents the same permutation as $w$.

 This $\hat{w}$ can be obtained from $w$ by operations (1), (2), and (3) above, so for any decomposition $\vec{f}$ such that $w \star \vec{f}$ is defined, $\hat{w}\star \vec{f} = w \star \vec{f}$.
\end{prop}

\begin{proof}
 First, replace $w$ by some length-minimal $w'$ obtained from $w$ by operations (1), (2), and (3) in Remark~\ref{synoprk}.
 Any word including $w'$ can be written as a sequence of right-to-left transits
 $ t_{[a_m, b_m]} t_{[a_{m-1}, b_{m-1}]} \ldots t_{[a_1, b_1]}$ with $a_i \geq b_i$ for all $i$.
 To achieve the additional requirement that $b_m < b_{m-1} < \ldots < b_1$ in Remark~\ref{alt-1stcan}, use Lemma~\ref{lem-for-1stcan} repeatedly to straighten out pairs of adjacent out-of-order $b_i$'s. Clearly, this process terminates.
\end{proof}

\begin{corollary}
\label{nearaction}
 If two words $w$ and $w'$ in the Ritt monoid $\rittmonoid_k$ represent the same permutation and both $w \star \vec{f}$ and $w' \star \vec{f}$ are defined, then $w \star \vec{f} = w' \star \vec{f}$.
\end{corollary}
\begin{proof}
 For every permutation there is a unique word in the first canonical form representing it.
\end{proof}

Another immediate consequence is a bound on the length of words and the number of (linear equivalence classes of) decompositions of a given polynomial; better bounds are obtained in~\cite{MZ}.

 \begin{corollary}
\label{1stcanbound}
 For any given polynomial $f$ and decomposition $(f_k, \ldots, f_1)$ of $f$, there are at most
$k!$ other decompositions $\vec{g}$ of $f$ (up to linear
equivalence, of course), and any one of them can be reached by a
sequence of at most $\frac{k(k-1)}{2}$ Ritt swaps.
\end{corollary}

Our main use of the first canonical form is to define and obtain our second canonical form.

 It is sometimes natural and often useful to break a decomposition into chunks before analyzing it. For example, in analyzing the commutative diagram $\pi^\sigma \circ f = g \circ \pi$, it is natural to start with a decomposition of $g \circ \pi$ which is a decomposition of $\pi$ followed by a decomposition of $g$. Clusterings in Section \ref{clustersec} are another example. Words in second canonical form (with respect to such a break-up into chunks) first shuffle factors within each chunk as much as necessary, and only then move factors between chunks. That is, these words perform a merge-sort.

\begin{Def}
Given integers $k=c_r > c_{r-1} > \ldots >c_1 >c_0=0$, a word $w \in \rittmonoid_k$ is in \emph{second canonical form with respect to $\vec{c}$} if it is of the form $w = v w_r w_{r-1} \ldots w_1$ and all of the following hold. \begin{itemize}
\item All $w_i$ and $v$ are in first canonical form.
\item For each $j$, only $t_i$ with $c_{j+1} > i > c_j$ appear in $w_j$; that is, $w_j$ only permutes factors in
$\vec{f}_{[c_{j+1}, c_j)}$.
\item For any two transits $t_{[a, b]}$ and $t_{[a', b']}$ in $v$ with both $b,b' \in [c_{j+1}, c_j)$ for some $j$, $a < a'$ if and only if $b <b'$. That is, $v$ does not change the order of two factors originating inside the same $\vec{f}_{[c_{j+1}, c_j)}$.
\end{itemize} \end{Def}

Note that $w_i$ in the definition above act on disjoint sets of factors, and therefore commute with each other.

\begin{prop}
\label{2ndcanprop}
For every word $w \in \rittmonoid_k$ and every tuple $\vec{c}$ of integers with $k=c_r > c_{r-1} > \ldots >c_1 >c_0=0$, there is a word $\hat{w} \in \rittmonoid_k$ in second canonical form with respect to $\vec{c}$ which represents the same permutation as $w$.

 This $\hat{w}$ can be obtained from $w$ by operations (1), (2), and (3) in Remark~\ref{synoprk}, so for any decomposition $\vec{f}$ such that $w \star \vec{f}$ is defined, $\hat{w}\star \vec{f} = w \star \vec{f}$.
\end{prop}

We first prove a special case $r=2$ of this proposition in the following lemma, and then prove the full proposition.

\begin{lemma} (Proposition~\ref{2ndcanprop} for $r=2$)
\label{llama2111}
 Fix $k > e > 0$. For every $w \in \rittmonoid_k$, there are $v, w_G, w_H \in \rittmonoid_k$ such that \begin{itemize}
\item  $v w_G w_H$ is obtained from $w$ by operations (1), (2), and (3) in Remark~\ref{synoprk};
\item  only $t_i$ with $i > e$ occur in $w_H$;
\item  only $t_i$ with $i < e$ occur in $w_G$;
\item   $v = t_{[a_m, b_m]} t_{[a_{m-1}, b_{m-1}]} \ldots t_{[a_1, b_1]}$ is in first canonical form, and
$$b_1 = e, b_2 = e-1, \ldots, b_m = e-m+1 \mbox{ and } a_1 > a_2 > \ldots > a_m.$$
\end{itemize} \end{lemma}

\begin{proof}
\textbf{First reductions:}
 Without loss of generality, we may assume that $w$ is already in first canonical form; that is,
  $$w = w_n w_{n-1} \ldots w_1 \mbox{ where each } w_i = t_{[c_i, d_i]}$$
 with $d_1 > d_2 > \ldots > d_n$ and $c_i \geq d_i$ for each $i$.

 We may further assume that $ e \leq d_1$. Indeed, otherwise let $j \leq n$ be the greatest such that $d_j >e$, and let
 $$\widetilde{w_H} := w_j w_{j-1} \ldots w_1 \mbox{ and } w' := w_n \ldots w_{j+1} \mbox{ so that } w = w' \widetilde{w_H}.$$
 Since $\widetilde{w_H}$ only involves $t_i$ with $i > e$, it clearly suffices to prove the proposition for $w'$.

So, we have
 $$w = w_n w_{n-1} \ldots w_1 \mbox{ where each } w_i = t_{[c_i,d_i]} $$
  in first canonical form with
 $$ e \geq d_1 > d_2 > \ldots > d_n \mbox{ and } c_i \geq d_i \mbox{ for each } i.$$
The first inequality above makes $d_n \leq e-n+1$.

\textbf{Claim:}
 There are $v$ and $w_G$ satisfying the last two requirements in the lemma with $m \leq n$, such that $v w_G \simeq w$.

 We obtain these $v$ and $w_G$ by induction on $n$.

 \textbf{Base case} When $n = 1$, consider the one and only chunk $w_n = t_{[c_n, d_n]}$ of  $w$. We know that $d_n \leq e$.\\
\textit{Case 1} If $c_n < e$, then $w_n \in \rittmonoid_e$, so $v := \emptyset$ and $w_G := w_n$ work.\\
\textit{Case 2} If $c_n \geq e$, then $w_n  = t_{[c_n, e]} t_{[e-1, d_n]}$, and $v := t_{[c_n, e]}$ and $w_G := t_{[e-1, d_n]}$ work.

\textbf{Induction step} We apply the induction hypothesis to the initial $(n-1)$ chunks of $w$ to get
$$ w_{n-1} \ldots w_{1} \simeq \widetilde{v} \widetilde{w_G}$$
 for some $$\widetilde{v} = t_{[a_m, e-m+1]} t_{[a_{m-1}, e-m+2]} \ldots t_{[a_1, e]}$$
  with $m \leq n-1$ and $a_1 > a_2 \ldots > a_m$,
  and some $\widetilde{w_G} \in \rittmonoid_e$ (so $\widetilde{w_G}$ only involves $t_i$ with $i < e$). So
$$ w = w_n w_{n-1} \ldots w_{1} \simeq w_n \widetilde{v} \widetilde{w_G}.$$
Let $w' := w_n \widetilde{v}$, so that $w = w' \widetilde{w_G}$.

\textbf{Subclaim:} There are words $v$ and $u_G$ such that $w' \simeq v u_G$, and $u_G \in \rittmonoid_e$ (so $u_G$ only involves $t_i$ with $i < e$), and $v$ has the required shape.

 Once we prove this Subclaim, $v$ and $w_G := u_G \widetilde{w_G}$ will satisfy the Claim.

 Proof of Subclaim: We have
  $$w' = t_{[c_n, d_n]} t_{[a_m, e-m+1]} t_{[a_{m-1}, e-m+2]} \ldots t_{[a_1, e]}$$
and $m \leq n-1$ and $a_1 > a_2 \ldots > a_m$ and $d_n \leq e-n+1$. So $d_n \leq e-m$, and so
 $$t_{[c_n, d_n]} = t_{[c_n, e-m]} t_{[e-m-1, d_n]},$$
 and, since the least index appearing in $\widetilde{v}$ is $e-m+1$ and the greatest in $t_{[e-m-1, d_n]}$ is $e-m-1$,
 $$w' = t_{[c_n, e-m]} t_{[e-m-1, d_n]} \widetilde{v} \simeq t_{[c_n, e-m]} \widetilde{v}  t_{[e-m-1, d_n]}.$$

Since $ t_{[e-m-1, d_n]}$ can be absorbed into $u_G$, it suffices to prove the subclaim for the special case where $d_n = e-m$. If now $c_n < a_m$, we may set $a_{m+1} := c_n$ and set $v := t_{[c_n, e-m]} \widetilde{v}$ and be done.

So, it suffices to prove the Subclaim for
 $$w' = t_{[c_n, e-m]} t_{[a_m, e-m+1]} t_{[a_{m-1}, e-m+2]} \ldots t_{[a_1, e]}$$
where $a_1 > a_2 \ldots > a_m$ but $c_n \geq a_m$. So
 $$w' = t_{[c_n, a_m+1]} t_{[a_m, e-m]} t_{[a_m, e-m+1]} t_{[a_{m-1}, e-m+2]} \ldots t_{[a_1, e]}$$
By Lemma~\ref{teenylem}, $t_{[a_m, e-m]} t_{[a_m, e-m+1]} \simeq t_{[a_m-1, e-m]} t_{[a_m, e-m+1]} t_{e-m}$, so
$$w'
\simeq t_{[c_n, a_m+1]} t_{[a_m-1, e-m]} t_{[a_m, e-m+1]} \mathbf{t_{e-m}} t_{[a_{m-1}, e-m+2]} \ldots t_{[a_1, e]}
\simeq$$
$$\simeq t_{[c_n, a_m+1]} t_{[a_m-1, e-m]} t_{[a_m, e-m+1]}  t_{[a_{m-1}, e-m+2]} \ldots t_{[a_1, e]} \mathbf{t_{e-m}} \simeq $$
$$\simeq  t_{[a_m-1, e-m]} \mathbf{t_{[c_n, e_m+1]}} t_{[a_m, e-m+1]}  t_{[a_{m-1}, e-m+2]} \ldots t_{[a_1, e]} t_{e-m}$$
Now if $c_n < a_{m-1}$, we are done, because $t_{m-e}$ may be absorbed into $u_G$, and the rest is already of the right form. Otherwise, we move $t_{[c_n, a_{m-1}+1]}$ one more step to the right in exactly the same way, and then compare $c_n$ to $a_{m-2}$. Since there is no requirement on $a_1$, this process ends in success after at most $m$ steps.
\end{proof}

The full version of Proposition~\ref{2ndcanprop} now follows by an easy induction.

\begin{proof} (This is the proof of Proposition~\ref{2ndcanprop}.)

 We induct on the number of chunks $r$. For $r=1$, this is just first canonical form. The case $r=2$ is Lemma \ref{llama2111}. For the induction step, we suppose that the proposition holds for $r=s$, and prove that it holds for $r = s+1$. Fix $w$ and $k= c_{s+1} > c_s > \ldots > c_1 > c_0 =0$.

 First, apply Lemma \ref{llama2111} to $\vec{d}$ where $k = d_2 > c_1 = d_1 > d_0 =0$. That is, replace $w$ by $v_Q \circ u_1 \circ u_2$, where $u_1$ only involves $t_i$ with $i < c_1$, $u_2$ only involves $t_i$ with $i > c_1$, and $v_Q$ does what it's supposed to.

 Then, apply the inductive hypothesis to $u_2$ and $k = c_{s+1} > c_s > \ldots > c_2 > c_0 =0$ to get
 $u_2
  \simeq v' w'_2 w'_3 \ldots w'_{s+1}$.
  So $w \simeq v_Q u_1 v' w'_2 w'_3 \ldots w'_{s+1}
 \simeq v_Q  v' u_1 w'_2 w'_3 \ldots w'_{s+1}$, the second equivalence because $u_1$ and $v'$ act on disjoint sets of factors.
 Finally, let $v :=v_Q v'$, $w_1 := u_1$, and $w_i := w'_i$ for $i \geq 2$.
 \end{proof}

 When the second canonical form is applied to a clustering, the characterization of $v$ can be substantially strengthened.

\begin{Rk} \label{canoclust}
 Suppose that $\vec{a}$ is a clustering of a decomposition $\vec{f}$, and fix $w \in \rittmonoid_k$ for which $w \star \vec{f}$ is defined. If $w = v w_1 w_2 \ldots w_t$ is in second canonical form with respect to $\vec{a}$, then each $w_j$ only permutes factors within the $j$th cluster, and $v$ only moves factors left from cluster to cluster but does not change the order of those originating in the same cluster.
 Since only quadratics can move between clusters, the Ritt swaps in $v$ can only move quadratics.
 These quadratics can go left or right, but cannot overtake each other because of gates, and cannot collide with each other because tautological Ritt swaps are not allowed.

By Proposition~\ref{swap-inside}, all Ritt swaps in $w_j$ may be witnessed by identity linear factors.
Now $\vec{a}$ is still a clustering of $\vec{g} := w_1 w_2 \ldots w_t \star \vec{f}$.
To illustrate what can happen in $v \star \vec{g}$, we describe in detail an initial chunk of $v$.

The rightmost symbol in $v$ is $t_{a_j}$ for some $j$, since $v$ does not swap factors inside a cluster.
By Lemma~\ref{extra-cluster-swap}, exactly one of $g_{a_j}$ and $g_{a_j +1}$ must be a wandering quadratic of $\vec{a}$ and $\vec{g}$. \begin{itemize}
\item If $g_{a_j}$ is a wandering quadratic, $v$ moves it to the left, leaves a one-way left-to-right gate at $j$, and then only permutes factors further right. That is, $v = v' t_{[b, a_j]}$ for some $b \geq a_j$ and $t_{[b, a_j]} \star \vec{g}$ has a one-way left-to-right gate between the $j$th and the $(j+1)$st clusters, and $v'$ is a word in $\{ t_i \sthat j < a_j -1\}$.
\item If $g_{a_j +1}$ is a wandering quadratic, $v$ moves it right to a new position, and the rest of $v$ cannot move other quadratics left past that new position. That is, $v = v' t_b t_{b+1} \ldots t_{a_j}$ and $v'$ is a word in $\{ t_i \sthat i < b -1\}$.
\end{itemize}

Applying the same analysis to the remaining $v'$ gives an inductive characterization of $v$ as a concatenation of (left-to-right or right-to-left) transits, each of which moves a wandering quadratic of $\vec{g}$ some number of steps (right or left).
\end{Rk}

 Our first use of canonical forms is to characterize those $(f,g)$-skew-invariant curves which have nothing to do with skew-twists. More precisely, we consider triples of polynomials $(f, g, \pi)$ satisfying $\pi^\sigma \circ f = g \circ \pi$, where $f$ and $\pi$ share no initial compositional factors, and $\pi^\sigma$ and $g$ share no terminal compositional factors. We continue to only consider disintegrated polynomials $f$ and $g$, that is $f$ and $g$ that are not skew-conjugate to monomials, Chebyshev polynomials, and negative Chebyshev polynomials.

 We begin by saying something about the conclusion of Proposition~\ref{notskewtwist}.

\begin{Def} \label{decomp-inout-def}
A decomposition $\vec{f}$ has \emph{in-degree (respectively, out-degree) divisible by $n$} if no $f_i$ is linearly related to $P_p$ for any $p$ that divides $n$, and every non-monomial factor of the decomposition is monic and has in-degree (respectively, out-degree) divisible by $n$. In particular, this forces all $f_i$ to be Ritt polynomials.\end{Def}

We first show that a skew-conjugacy class of disintegrated polynomials has at most one decomposition with non-trivial in- or out-degree, up to skew-conjugating by scalings.

\begin{Rk}
In- and out-degrees make sense for decomposable polynomials (see Definition~\ref{inoutdegdef}), and if $\vec{f}$ has in-degree divisible by $p$, then indeed the in-degree of $\vec{f}^\circ$ is divisible by $p$. A converse requires additional hypotheses and is proved in \cite{MZ}. \end{Rk}

\begin{Rk} \label{uni-inout-sca}
It is clear that if $\vec{h}$ is skew-linearly equivalent to $\vec{f}$ via scalings,
 then $\vec{f}$ and $\vec{h}$ have the same in- and out-degrees.
\end{Rk}

\begin{lemma} \label{uni-inout-tra}
If $\vec{f}$ is a decomposition of a disintegrated polynomial, $\vec{g}$ is skew-linearly equivalent to $\vec{f}$ via translations, and each of the two has non-trivial in- or out-degree, then $\vec{g} = \vec{f}$. \end{lemma}

\begin{proof}

\textit{Setup: } There is a translation $M$ such that $\vec{g}$ is linearly equivalent to $\vec{e} := (M^\sigma \circ f_k, f_{k-1}, \ldots, , f_2, f_1 \circ M^{-1})$. We show that $M = \id$, after which repeated applications of Lemmas~\ref{ccluster-unique} and~\ref{nccluster-unique} finish the proof.

We now use Lemmas~\ref{precl-linrel} and~\ref{cl-linrel} and the ideas in their proofs.

Let $(\vec{h}, \vec{L})$ be a cleanup of a clustering $\vec{a}$  of $\vec{f}$;
then $\vec{a}$ is also a clustering of $\vec{e}$ and $\vec{g}$ and
 $(\vec{h}; (M^\sigma \circ L_k, L_{k-1}, \ldots, , L_1, L_0 \circ M^{-1}) )$
is a cleanup of $\vec{e}$ and $\vec{g}$ (with respect to $\vec{a}$).

For each $j$, the $j$th clusters of $\vec{f}$ and $\vec{g}$ are of the same kind, both \textsf{C} or both \textsf{C}-free.

Since $\vec{f}$ and $\vec{g}$ have non-trivial in- or out-degree, all factors $f_i$ and $g_i$ are already Ritt polynomials.

\textit{Case 1:} Suppose that the rightmost cluster $\vec{f}_{[a_1, 0)}$ of $\vec{f}$ is a \textsf{C}-free cluster. Since $\vec{f}$ and $\vec{g}$ have non-trivial in- or out-degree, all factors $f_i$ and $g_i$ are already Ritt polynomials, and then by the uniqueness of cleanups and the proof of the existence of cleanups $L_0 = \id$. But then the rightmost cluster $\vec{g}_{[a_1, 0)}$ of $\vec{g}$ is a \textsf{C}-free cluster, so for the same reason $L_0 \circ M^{-1} = \id$. So $M = \id$ as desired.

\textit{Case 2:} Suppose that the rightmost cluster $\vec{f}_{[a_1, 0)}$ of $\vec{f}$ is a \textsf{C} cluster, so the rightmost cluster $\vec{g}_{[a_1, 0)}$ of $\vec{g}$ is also a \textsf{C} cluster. Since all factors of $f$ and $g$ are already Ritt polynomials, $\vec{f}$ cannot be a single \textsf{C} cluster.
Intuitively, our strategy is to pull the linear $M$ through the rightmost cluster, and then obtain a contradiction, as in Case 1 if the second rightmost cluster is \textsf{C}-free, and otherwise by forcing the two clusters to fuse.

Since type \textsf{C} Ritt polynomials have in- and out-degrees $1$ and $2$, in this case the non-trivial in- or out-degrees of $f$ and $g$ must be $2$, so there are no quadratic factors, so the degree $N$ of this \textsf{C} cluster $\vec{f}_{[a_1, 0)}$ is odd. Now
$$\vec{f}_{[a_1, 0)}^\circ = (\cdot \lambda_f^{-N}) \circ T_f^{-1} \circ C_N \circ T_f \circ (\cdot \lambda_f)$$
for some non-zero scalar $\lambda_f$ and $T_f = \id$ or $T_f = (+2)$, and
$$\vec{g}_{[a_1, 0)}^\circ = (\cdot \lambda_g^{-N}) \circ T_g^{-1} \circ C_N \circ T_g \circ (\cdot \lambda_g).$$ In order for $\vec{g}$ to be linearly equivalent to $\vec{e}$, there must be a linear $L$ such that
 $L \circ \vec{f}_{[a_1, 0)}^\circ \circ M^{-1} = \vec{g}_{[a_1, 0)}^\circ$.
 That is,
 \begin{equation} \label{uni-inout-eq}
 L \circ (\cdot \lambda_f^{-N}) \circ T_f^{-1} \circ C_N \circ T_f \circ (\cdot \lambda_f) \circ M^{-1} =
 (\cdot \lambda_g^{-N}) \circ T_g^{-1} \circ C_N \circ T_g \circ (\cdot \lambda_g).
 \end{equation}

Thus by Corollary~\ref{oneCRk},
$$T_g \circ (\cdot \lambda_g^N) \circ L \circ (\cdot \lambda_f^{-N}) \circ T_f^{-1} = (\cdot \pm 1)
\mbox{  and  }
T_f \circ (\cdot \lambda_f) \circ M^{-1} \circ (\cdot \lambda_g^{-1}) \circ T_g^{-1} = (\cdot \pm 1).$$
Since $T_f$, $T_g$, and $M$ are translations, the second equation implies that $\lambda_f = \pm \lambda_g$.
Since $N$ is odd, $\frac{\lambda_f^N}{\lambda_g^N} = \frac{\lambda_f}{\lambda_g}$, so the first equation makes $L$ a translation.

In equation~\ref{uni-inout-eq}, bring the scalings by $\lambda_f$ and $\lambda_f^{-N}$ to the outside of the left-hand side conjugating $L$ and $M^{-1}$ to $\tilde{L}$ and $\tilde{M}$; then bring them to the other side to cancel with the scalings by $\lambda_g$ and $\lambda_g^{-N}$, leaving a scaling $S$ by $\pm 1$; and finally bring this $S$ inside the translations on the right-hand side, conjugating $T_g$ to $\tilde{T}_g$; note that now in the middle of the right-hand side $S \circ C_n \circ S = C_N$, so we get
$$\tilde{L} \circ T_f^{-1} \circ C_N \circ T_f \circ \tilde{M} = \tilde{T}_g^{-1} \circ C_N \circ \tilde{T}_g.$$
 Applying Corollary~\ref{oneCRk} to this new equation makes $\tilde{L} = \tilde{M}^{-1} = \tilde{T}_g^{-1} \circ T_f$.
 Thus, if $M \neq \id$, then $L \neq \id$ also. This is what I meant by ``pulling $M$ through the rightmost cluster''.

Recall that $L$ was a witness to the linear equivalence of $\vec{g}$ and $\vec{e}$, so chopping off the rightmost cluster, we have linearly equivalent decompositions
$(M^\sigma \circ f_k, f_{k-1}, \ldots, , f_2, f_{a_1+1} \circ L^{-1})$ and $\vec{g}_{[k, a_1)}$, non-empty since we noted above that $\vec{f}$ has at least one more cluster. The same reasoning, for both cases, applies to these new decomposition. In Case 1, we immediately obtain $L = \id$. In Case 2, the same analysis relates $L$ to the translations $\hat{T}_f$, $\hat{T}_g$ and scaling by $\hat{\lambda_f}$, $\hat{\lambda_g}$ that make
$$\vec{f}_{[a_2, a_1)}^\circ = (\cdot \hat{\lambda}_f^{-N}) \circ \hat{T}_f^{-1} \circ C_N \circ \hat{T}_f \circ (\cdot \hat{\lambda}_f)$$
and the same for $g$. This relation among $\lambda_f$, $T_f$, $\hat{T}_f$, and $\hat{\lambda}_f$ and the corresponding data for $g$ forces the first two clusters to fuse into a single cluster, contradicting the fact that $\vec{a}$ is a clustering.

\end{proof}

\begin{prop} \label{uni-inout}
Suppose that $\vec{f}$ is a decomposition of a disintegrated polynomial, $\vec{g}$ is skew-linearly equivalent to $\vec{f}$, and each has non-trivial in- or out-degree. Then there is some $\lambda$ such that $g_i = \lambda^{m_i} \ast f_i$, where $m_i := \deg( \vec{f}_{(i, 1]}^\circ )$. In particular, their in- and out-degrees are the same.\end{prop}

\begin{proof}
As usual, we can deal with scalings and translations separately. Separating scalings from translations in the linear factors witnessing skew-linear equivalence, we find an intermediate decomposition $\vec{h}$ which is skew-linearly equivalent to $\vec{f}$ via scalings, and skew-linearly equivalent to $\vec{g}$ via translations. By Remark~\ref{uni-inout-sca}, $\vec{h}$ has the same in- and out-degrees as $\vec{f}$. By Lemma~\ref{uni-inout-tra}, $\vec{h} = \vec{g}$.\end{proof}

The last proposition essentially says that for decompositions of disintegrated polynomials, in- and out-degrees are invariant under skew-linear equivalence.

\begin{prop}
\label{notskewtwist} \label{notskewtwrk}
  If two disintegrated polynomials $f$ and $g$ satisfy $g \circ \pi = \pi^\sigma \circ f$, and
  $f$ and $\pi$ share no initial compositional factors, and $\pi^\sigma$ and $g$ share no terminal compositional factors, then there are linear $L$ and $M$ such that $M \circ \pi \circ L$ is a monomial whose degree divides the
  in-degree of some decomposition of $ (M^\sigma)^{-1} \circ  g \circ M$ and the out-degree of some decomposition of $L^\sigma \circ f \circ L^{-1}$. \end{prop}

The following slight weakening of this proposition, which does not refer to decompositions, is an immediate consequence of Lemma 2.8 of \cite{MZ}, and the full version follows from other results in that paper.

\begin{corollary}
If two non-linear polynomials $f$ and $g$ satisfy $g \circ \pi = \pi^\sigma \circ f$,
  and $f$ and $\pi$ share no initial compositional factors,
  and $\pi^\sigma$ and $g$ share no terminal compositional factors,
  then there are linear $L$ and $M$ such that
   either $L^\sigma \circ f \circ L^{-1}$ and $(M^\sigma)^{-1} \circ  g \circ M$ are both monomials, Chebyshev polynomials, or negative Chebyshev polynomials (and then we say nothing about $\pi$);
   or $M \circ \pi \circ L (x) = x^n $ is a monomial, $L^\sigma \circ f \circ L^{-1} (x) = x^k \cdot u(x^n)$, and $(M^\sigma)^{-1} \circ  g \circ M (x) = x^k \cdot u(x)^n$ for some polynomial $u$.\end{corollary}

Our slightly stronger statement is the one we use to characterize skew-invariant curves.
The rest of this section constitutes the proof of Proposition~\ref{notskewtwist}. The next proposition translates it into the language of decompositions and canonical forms.

\begin{prop} \label{notsktw-translate}
(Translating Proposition~\ref{notskewtwist})\\
 Suppose that polynomials $f$, $g$, and $\pi$ satisfy $g \circ \pi = \pi^\sigma \circ f$, and
  that $f$ and $\pi$ share no initial compositional factors, and $\pi^\sigma$ and $g$ share no terminal compositional factors.
 Let $m$ be the number of factors in (any) decomposition of $\pi$, and let $l$ be the number of factors in (any) decomposition of $f$ (or $g$).
 Then there are decompositions $\vec{\pi}$ of $\pi$, $\vec{f}$ is $f$, $\vec{g}$ of $g$, and $\vec{\rho}$ of $\pi^\sigma$ (which $\vec{\rho}$ need not be $(\vec{\pi})^\sigma$) such that
$$ ( t_{[\ell, 1]} \ldots t_{[\ell+m-2, m-1]} t_{[l+m-1, m]} ) \star \vec{g}\vec{\pi} = \vec{\rho}\vec{f}$$
\end{prop}

\begin{proof}

 Let $(\pi_m, \ldots, ,\pi_1)$ be a decomposition of $\pi$, and $(g_l, \ldots, g_1)$ be a decomposition of $g$. Let $w = v w_1 w_2$ be the word in the second canonical form that yields a decomposition of $f$ followed by a decomposition of $\pi^\sigma$. Since we were free to choose the decompositions of $\pi$ and $g$, we may assume, losing this freedom, that $w_i$ are empty. So we get decompositions as above and
 $$v = t_{[a_k, b_k]} t_{[a_{k-1}, b_{k-1}]} \ldots t_{[a_1, b_1]}$$
 with $a_i \geq b_i -1$ for all $i$ ( $a_i = b_i-1$ means that the word $(t_{a_i}, \dots t_{b_i})$ is empty);  $b_i = \operatorname{length}(\vec{\pi})+1-i$,
  and $a_k < \ldots < a_2 < a_1$; and
 $$v \star \vec{g}\vec{\pi} = \vec{\rho}\vec{f}$$
  Now it follows immediately that $k = \operatorname{length}(\vec{\pi})$, for otherwise $t_1$ does not occur in $v$, so
the rightmost factor $\pi_1$ in $\vec{g}\vec{\pi}$ is untouched by
the action of $v$, so it is a shared initial factor of $\pi$ and
$f$, contradicting a hypothesis of the proposition.

For exactly the same reasons, unless
 $a_i = \operatorname{length}(\vec{g}) + \operatorname{length}(\vec{\pi}) - i$ for all $i$,
  $\rho$ and $g$ will share a terminal factor, which is also not supposed to happen.

 So $v =  t_{[\ell, t_1]} \ldots t_{[\ell+m-2, m-1]} t_{[\ell+m-1, m]} $ as wanted.
\end{proof}

 We have $v := v_1 v_2 \ldots v_m$ where $v_i := t_{[\ell+1-i, i]}$, and we have $v \star \vec{g}\vec{\pi}$ defined.

The next lemma shows that $\vec{g}\vec{\pi}$ is a cluster or one of $\pi$ and $g$ has degree 2. The following two lemmas handle these two cases.

\begin{lemma}
 Suppose that the conclusion of Proposition~\ref{notsktw-translate} holds, but $\vec{g}\vec{\pi}$ is not a cluster.
 Then one of $\pi$ and $g$ is quadratic. \end{lemma}

\begin{proof}
 Suppose that $\vec{g}\vec{\pi}$ is not a cluster, and fix a clustering $\vec{a}$ of it.

 If the leftmost cluster $(g_k, \ldots, g_e)$ of $\vec{a}$ does not contain all of $\vec{g}$, then $(g_k, \ldots, g_1, \pi_m)$ is not a cluster, so Proposition~\ref{one-traverse} applied to $ v_m \star \vec{g}\vec{\pi} =: (\pi'_m, \vec{g}', \pi_{m-1}, \pi_{m-2}, \ldots, \pi_1) $ makes $\pi_m$ quadratic and leaves one-way left-to-right gate
  between $g'_e$ and $g'_{e-1}$. This prevents $\pi_{m-1}$ from crossing into the leftmost cluster, contradicting the conclusion of Proposition~\ref{notsktw-translate}, so there must be no $\pi_{m-1}$, so the whole of $\pi = \pi_m$ is quadratic.

 Symmetrically, if the rightmost cluster of $\vec{a}$ does not contain all of $\pi$, then $g$ must be quadratic.

 So suppose that the two clusters of $\vec{a}$ are exactly $\vec{g}$ and $\vec{\pi}$. Since $(g_1, \pi_m)$ is swappable, one of them must be a wandering quadratic of this clustering. If $g_1$ is the wandering quadratic, then $\pi_m$ cannot move further left, so there must no more left for it to go, i.e. $g = g_1$ is quadratic. If $\pi_m$ is the wandering quadratic, then it leaves a one-way left-to-right gate between the two clusters, so $\pi_{m-1}$ cannot enter the left cluster, so there must be no  $\pi_{m-1}$, so the whole of $\pi = \pi_m$ is quadratic.
\end{proof}

\begin{lemma} \label{quadpig}
Proposition~\ref{notskewtwist} holds when one of $\pi$ and $g$ has degree $2$. \end{lemma}

\begin{proof}
If $\pi$ is quadratic, skew-conjugate $f$ and $g$ to make $\pi = P_2$. Let $\vec{a}$ be a clustering of $\vec{g}$.
Applying Proposition~\ref{one-traverse} to $v_m \star (\vec{g} P_2)$ shows that there must be a right-to-left gate
 between any two clusters of $\vec{a}$.
 In order for $\pi = P_2$ to enter, $\vec{a}$ must also have a right-to-left gate at $0$.
At the far left end, $\vec{a}$ must also have a right-to-left gate at $k$, since after all the Ritt swaps are performed, the now leftmost quadratic factor must exit as $\pi^\sigma = P_2$ with no  
additional linear factors.
It is now routine to verify that having all these gates is sufficient for the conclusion of Proposition~\ref{notskewtwist}.

If $g$ is quadratic, the in/out degrees of the factors of $\pi$ that are not type \textsf{C} are irreversably changed by the traversing quadratic factor, so \textsf{C}-free clusters in $\pi$ must be purely monomial. Similarly, if $\pi$ is not a single cluster, then the gates in $\pi$ (before the quadratic gets across to become $f$) and in $\pi^\sigma$ (after the quadratic has gotten across) do not match up. Thus we can skew-conjugate $f$ and $g$ to make $\pi$ a monomial or a Chebyshev polynomial. It is not routine to verify that this forces $f$ and $g$ to be monomials or Chebyshev polynomials, contradicting disintegratedness.
\end{proof}

\begin{lemma} Proposition~\ref{notskewtwist} holds when $\vec{g}\vec{\pi}$ is a cluster. \end{lemma}

\begin{proof}
 We may and do assume that neither $g$ nor $\pi$ are quadratic, since Lemma \ref{quadpig} takes care of those cases.
 The statement of Proposition~\ref{notskewtwist} is invariant under skew-conjugating $f$ and $g$ (and changing $\pi$ accordingly), so we may do so during the proof without loss of generality.

\textit{First reductions:}
Since $\vec{g}\vec{\pi}$ is a single cluster, there are linear $A$ and $B$ and Ritt polynomials $\tilde{g_i}$ and $\tilde{\pi}_i$ such that $\vec{g}\vec{\pi}$ is linearly equivalent to
 $(A \circ \tilde{g}_\ell, \tilde{g}_{\ell -1}, \ldots, \tilde{g}_1, \tilde{\pi}_m, \ldots, \tilde{\pi}_2, \tilde{\pi}_1 \circ B)$. Skew-conjugating $f$ by $B$, we may assume that $B = \id$.

Now there is a linear $C$ such that $( C^{-1} \circ \tilde{\pi}_m, \ldots, \tilde{\pi}_2, \tilde{\pi}_1 )$ is linearly equivalent to $\vec{\pi}$ and $(A \circ \tilde{g}_\ell, \tilde{g}_{\ell -1}, \ldots, \tilde{g}_1 \circ C)$ is linearly equivalent to $\vec{g}$. Skew-conjugating $g$ by $C$, that is replacing $\pi$ by $C \circ \pi$ and replacing $g$ by
$ C^\sigma \circ g \circ C^{-1}$, and replacing $A$ by $C^\sigma \circ A$, we may assume that $C = \id$.

Thus, replacing $\vec{\pi}$ and $\vec{g}$ by linearly equivalent decompositions, we may assume that all $\pi_i$ are Ritt polynomials $\tilde{\pi}$, that $g_i$ for $i \neq \ell$ are Ritt polynomials $\tilde{g_i}$, and that $g_\ell =
A \circ \tilde{g}_\ell$ for the linear $A$ and Ritt polynomial $\tilde{g_\ell}$.

\textit{Case 1:} If this is a \textsf{C} cluster, then all $\tilde{g}_i$ and $\pi_i$ are Chebyshev polynomials, so $\pi = \pi^\sigma$ is a Chebyshev polynomial and $g = A \circ C_{\deg(g)}$. Now
$v \star \vec{g}\vec{\pi} = (A \circ \pi_m, \pi_{m-1}, \ldots, \pi_1, g_{l}, \ldots, g_2, g_1)$ is linearly equivalent to $\vec{\rho}\vec{f}$ for some decompositions $\vec{f}$ of $f$ and some decomposition $\rho$ of the Chebyshev polynomial $\pi^\sigma$ of degree greater than $2$. In particular, there is a linear $D$ such that $A \circ C_{\deg(\pi)} \circ D = C_{\deg(\pi)}$, so by Lemma~\ref{ccluster-unique} $A = (\cdot \pm 1)$, contradicting the hypothesis that $g$ is disintegrated.

\textit{Case 2:} If $\vec{g}\vec{\pi}$ is a \textsf{C}-free cluster, at least one of $g$ and $\pi$ must be linearly related to a monomial.

Suppose towards contradiction that $\pi$ is not monomial, so at least one factor $\pi_i$ is not monomial, nor type \textsf{C}. In $v \star \vec{g}\vec{\pi} = \vec{\rho} \vec{f}$, each such factor $\pi_i(x) = x^{k_i} \cdot u(x^{\ell_i})^{n_i}$ becomes, in the sense of Remark~\ref{costumes}, $\tilde{\rho}_i$ with new in-degree $\ell_i \cdot \deg(f)$ and new out-degree $\frac{n_i}{\deg(f)}$. For some linear $D$, the decomposition $\vec{\rho}$ of $\pi^\sigma$ is linearly equivalent to $(\tilde{\rho}_m, \ldots \tilde{\rho}_2, \tilde{\rho}_1 \circ D)$. (Recall that the action $\star$ is only defined up to linear equivalence.) Since $\sigma(\vec{\pi})$ is another decompositions of $\pi^\sigma$, it must be possible to obtain $\sigma(\vec{\pi})$ from $(\tilde{\rho}_m, \ldots \tilde{\rho}_2, \tilde{\rho}_1 \circ D)$ by a sequence of Ritt swaps. Recall that in order for $v \star \vec{g}\vec{\pi}$ to be defined, the degrees of the monomial factors of $\pi$ must be relatively prime to $\deg(f)$, so Ritt swaps within $\vec{\rho}$ cannot undo the changes to in- and out-degrees. None of $\rho_i$ are type \textsf{C}, and all have non-trivial in-degree, so inserting linear factors also cannot undo those changes. This is a contradiction.

Thus, all $\pi_i$ and the whole $\pi$ are monomials, so $v \star \vec{g}\vec{\pi} = (A \circ \pi_m, \pi_{m-1}, \ldots, \pi_1, f_{l}, \ldots, f_2, f_1)$. Thus, there is a linear $D$ such that $A \circ P_{\deg(\pi)} \circ D = P_{\deg(\pi)}$, so $A$ must be a scaling. As we are working over a difference-closed field, we can get rid of $A$ by skew-conjugating $g$ by an appropriate scaling. All Ritt swaps within a cluster can be witnessed by identity linear factors, so $f_i$ have the requisite in- and out-degrees.
\end{proof}

Having finished the proof of Proposition~\ref{notskewtwist}, we note two  
consequences of it.

\begin{corollary} \label{breaknotsktw}
If $f$, $g$, and $\pi$ are as in Proposition~\ref{notskewtwist}, then there are indecomposable $\pi_i$ for $i \leq m$,
and polynomials $f = f_0, f_1, \ldots, f_m = g$ such that $\vec{\pi}^\circ = \pi$, and for each $i$, $f_{i-1}$, $\pi_i$, and $f_i$ are also as in Proposition~\ref{notskewtwist}.\end{corollary}

\begin{corollary} \label{notsktw-degbound}
If $f$, $g$, and $\pi$ are as in Proposition~\ref{notskewtwist}, then the degree of $\pi$ is bounded by the degree of any indecomposable factor of $f$, so a fortiori bounded by the degree of $f$. \end{corollary}

\section{Classification of skew-invariant curves from clusterings}
\label{puttingtogether}

In this last technical Section~\ref{puttingtogether}, we bring together clusterings, the action by the skew-twist monoid $\skewmonoid_k$ on skew-linear-equivalence classes of decompositions, and the characterization in Proposition~\ref{notskewtwist} of indecomposable curves that do not come from skew-twists in order to finally state and prove our classification of irreducible plane curves which are $(f,g)$-skew-invariant for a given pair $(f,g)$ of disintegrated polynomials.

In order to describe how correspondences not coming from skew-twists interact with skew-twists, we bring them into our monoid-action formalism via new monoid generators. To characterize correspondences coming from skew-twists, we describe the interaction between clusterings and skew-twists. Finally, we put it all together to write out the final characterization, and then state a few special cases and more readable weakenings.

\subsection{Augmented skew-twist monoid} \label{more-gen-sec}
By Corollary~\ref{breaknotsktw}, skew-invariant correspondences arising from Proposition~\ref{notskewtwist} (rather than from skew-twists) can also be broken down into indecomposable factors which are graphs of monomials of prime degree $p$. For each prime $p$, the graph $P_p$ and its converse relation will be encoded by the new generators $\delta_p$ and $\epsilon_p$, respectively. Most of the time, the action of these new generators will be undefined, since Proposition~\ref{notskewtwist} forces the decomposition to have a non-trivial in- or out-degree, which is rare. Proposition~\ref{uni-inout} essentially shows that the action is well-defined on skew-linear equivalence classes.

We now do for correspondences coming from Proposition~\ref{notskewtwist} what we did for skew-twists in Section~\ref{sktwsec}: we define a new monoid with more generators, its action on decompositions, witnessing sequences, encoded correspondences, its action on skew-linear equivalence classes of decompositions, a notion of equivalence for words in the new monoid. As for skew-twists, we then show that the correspondence encoded is essentially well-defined, and prove enough equivalences for words to get enough control on degrees to characterize invariant curves.

\begin{Def}
The free monoid
 generated by the generators of $\skewmonoid_k$
 together with countably many new symbols $\epsilon_p$ and
 $\delta_p$ as $p$ ranges through the primes  is denoted by $\skewmonoid^+_k$.

If $\vec{f}$ has non-trivial in-degree divisible by $p$, then $\epsilon_p \star \vec{f} := \vec{g}$ where $g_i := f_i$ whenever $f_i$ is a monomial, and for all other $i$ there are monic non-constant polynomials $u_i$ and integers $k_i \geq 1$ such that $f_i(x) = x^{k_i} \cdot u_i(x^{p \ell_i})^{n_i}$ and $g_i := x^{k_i} \cdot u_i(x^{\ell_i})^{p n_i}$.

To undo what $\epsilon_p$ does, we define $\delta_p \star \vec{g} := \vec{f}$ as above when $\vec{g}$ has non-trivial out-degree divisible by $p$.
\end{Def}

\begin{lemma} \label{delep-welldef}
 Suppose that $\vec{f}$ and $\vec{h}$ are skew-linearly equivalent decompositions of disintegrated polynomials, both have non-trivial in-degree divisible by $p$, and $\vec{g} = \epsilon_p \star \vec{f}$ and $\vec{ \tilde{g}} = \epsilon \star \vec{h}$.
 Then there is some $\lambda$ such that $h_i = \lambda^{m_i} \ast f_i$, where $m_i := \deg( \vec{f}_{(i, 1]}^\circ )$, and $\tilde{g}_i = \mu^{m_i} \ast g_i$ for $\mu = \lambda^p$.

 The corresponding result holds for $\delta$ in place of $\epsilon$.
\end{lemma}

\begin{proof}
The first conclusion is Proposition~\ref{uni-inout}, and the second follows immediately. \end{proof}

\begin{Def} \label{curv-enc-plus}
Suppose that $\vec{f}$ is a decomposition of a polynomial $f$ and $w := w_n \ldots w_2 w_1 \in \skewmonoid^+_k$ where each $w_j$ is a single generator: a Ritt swap $t_i$, a single skew-twist $\phi$ or $\beta$, or one of the new generators $\epsilon_{p_j}$ or $\delta_{p_j}$ for some prime $p_j$.

A sequence of decompositions $\vec{f} = \vec{f}^0, \vec{f}^1, \ldots, \vec{f}^n$ is \emph{a witnessing sequence for $w \star \vec{f}$} if for each $j$,\begin{itemize}
 \item if $w_j$ is $t_i$, $\phi$, or $\beta$, see Definition~\ref{correncdef};
 \item if $w_j = \epsilon_p$, then $\vec{f}^j$ is skew-linearly equivalent to some $\vec{h}$ which has non-trivial in-degree divisible by $p_j$, and $\vec{f}^{j+1} = \epsilon_p \star \vec{h}$;
 \item is $w_j = \delta_p$, switch the roles of $\vec{f}^j$ and $\vec{f}^{j+1}$ above.
\end{itemize}

If such a sequence exists, we write $w \star [[\vec{f}]] = [[\vec{f}^n]]$; otherwise, $w \star [[\vec{f}]] = \infty$.\\
\vspace{.1cm}

\emph{The correspondence $\kriva$ encoded by this witnessing sequence} is again the composite of curves $\krivb_j$; for $1 \leq j \leq n$ \begin{itemize}
\item if $w_j$ is $t_i$, $\phi$, or $\beta$, the curve $\krivb_j$ is exactly as in Definition~\ref{correncdef};
\item if $w_j = \epsilon_p$, $C_j$ is the graph of $P_p \circ T_j$, where $T_j$ is the outside linear factor witnessing that $\epsilon_p \star [[\vec{f}^j]]$ is defined;
\item if $w_j = \delta_p$, $C_j$ is the converse of this graph. \end{itemize}

We also say that $\kriva$ is \emph{a correspondence encoded by $w \star \vec{f}$}.
\end{Def}

\begin{Rk}  \label{256257forplus}
Note that it may well be that $\epsilon \star [[\vec{f}]]$ is defined but $\epsilon \star \vec{f}$ is not.

By Lemma~\ref{delep-welldef}, the witnessing sequence is well-defined up to skew-linear equivalence, so $w \star [[\vec{f}]]$ is well-defined, and gives an action of $\skewmonoid^+_k$ on skew-linear equivalence classes of decompositions.

Further, by Lemma~\ref{delep-welldef} and the fact that $\lambda \ast P_p = P_p$, the curve encoded by $\epsilon \star [[\vec{f}]] = [[\vec{g}]]$ does not depend on the choice of the witnessing $\vec{h}$, up to a terminal linear factor exactly as Lemma~\ref{newhapppycorrs}. The same holds for $\delta$, so Lemma~\ref{newhapppycorrs} holds for words in $\skewmonoid^+_k$, along with Remark~\ref{concat-compose} that says that concatenation of words corresponds to composition of encoded correspondences.
\end{Rk}

We define equivalence for words in $\skewmonoid^+_k$ exactly as in Definition~\ref{equicorrdef} for words in $\skewmonoid_k$.

\begin{Def}
Given $v, w \in \skewmonoid^+_k$  and a decomposition $\vec{f} = (f_k,\ldots,f_1)$.
We say that $v$ and $w$ are \emph{equivalent with respect to $\vec{f}$} and write $v \approx_{\vec{f}} w$ if
$v \star [[\vec{f}]] = w \star [[\vec{f}]]$ and there are  witnessing sequences
$(\vec{g}^j)$ and $(\vec{h}^j)$ for $v \star \vec{f}$ and $w \star \vec{f}$, respectively so
that the final $\vec{g}^n$ and $\vec{h}^n$ are decompositions of the same polynomial $g$, and
$(\kriva_v)_\per =(\kriva_{w})_\per$ for the curves $\kriva_v$ and $\kriva_w$ encoded by $v$ (respectively, $w$) via  $(\vec{g}^j)$
 (respectively, $(\vec{h}^j)$).

 When $v \approx_{\vec{f}} w$ for all $\vec{f}$, we write $v \approx w$ and say that the two words are \emph{equivalent}.
\end{Def}

\begin{Rk} \label{plus-concat-compose}
It is  clear that this equivalence again respects concatenation: if $u \approx u'$ and $v \approx v'$, then $uv \approx u'v'$.\end{Rk}

\begin{lemma} \label{new-gen-equivs}
\begin{itemize}
\item For any of the old generators $x = t_i, \phi, \psi$ and any prime $p$, $\epsilon_p x \approx x \epsilon_p$ and $\delta_p x \approx x \delta_p$.
\item For any $p \neq q$ and any $x,y \in \{ \epsilon, \delta\}$, $x_p y_q \approx y_q x_p$.
\item $\delta_p \epsilon_p \approx \id$.
\end{itemize}
\end{lemma}

\begin{proof}
It is clear that single skew-twists do not change the in- and out-degrees of a decomposition. To see that Ritt swaps also do not change them, recall that a decomposition with non-trivial in- or out-degree divisible by $p$ may not have the monomial $P_p$ among its factors; when $p=2$, this precludes wandering quadratics and makes clusterings completely rigid. This also takes care of the second part. The last part is obvious.\end{proof}

\begin{Rk}
The $(f,g)$-(skew-)invariant curve encoded by $\epsilon_p \delta_p \star f = g$, defined by $x^p = y^p$, is the union of $p$ lines whose slopes are $p$th roots of unity. Because $P_p$ is not a compositional factor of $f$, components other than the diagonal may be skew-periodic, unlike in the case of skew-twists, so it is not true that $\epsilon_p \delta_p \approx \id$. However, composing the curve defined by $x^p = y^p$ with the one defined by $y^p = z^p$ does not give anything new, so $\epsilon_p \delta_p \epsilon_p \delta_p \approx \epsilon_p \delta_p$.\end{Rk}

The next corollary together with Corollary~\ref{notsktw-degbound} bounds the degrees of the correspondence coming from Theorem~\ref{notskewtwist}.

\begin{corollary} \label{wtilde-degbound}
For any word $w$ consisting entirely of $\epsilon_p$ and $\delta_p$ for various $p$, there are words $u$ and $v$ such that $w \approx uv$, and $u$ consists entirely of $\epsilon_p$ for various $p$, and $v$ consists entirely of $\delta_p$ for various $p$.

Thus, the degrees of the two monomials encoded by $u$ and $v$ are bounded by $\deg(f)$.
\end{corollary}

\subsection{Clusterings and skew-twists}
\label{skew-clust-sec}

The interaction between clusterings and skew-twists is the key to finishing the characterization of curves encoded by words in $\skewmonoid_k$.
Recall ( Remark~\ref{circledance} ) that in the context of skew-twists one should imagine the factors of a decomposition standing in a circle, rather than in a line, with only a faint marker between the ``first'' and ``last'' factors to remind one to add $\sigma$ or $\sigma^{-1}$ to factors passing the marker. In this vein, recall that the skew-twist monoid $\skewmonoid_k$ acts on skew-linear equivalence classes of decompositions, that is on decompositions up to linear equivalence and skew-conjugacy.

\begin{Def}
If $\vec{f}$ is itself a cluster, then the sequence $(k)$ is a \emph{skew-preclustering} of $\vec{f}$. In general,
  a sequence $k \geq a_r > \ldots > a_1 >0 $ with $r \geq 2$ is a \emph{skew-preclustering} of a decomposition $\vec{f}$ if $f_{[a_i, a_{i-1})}$ is a cluster for each $r \geq i > 1$, and $(f_{a_1}^\sigma, f_{a_i -1}^\sigma, \ldots, f_1^\sigma, f_k, \ldots, f_{a_r})$ is also a cluster.

 A skew-preclustering is a \emph{skew-clustering} if no cluster (including $(f_{a_1}^\sigma, f_{a_1 -1}^\sigma, \ldots, f_1^\sigma, f_k, \ldots, f_{a_r+1})$) consists of two wandering quadratics, and the concatenation of two adjacent clusters is never a cluster, including the concatenations $(f_{a_2}^\sigma, f_{a_2 -1}^\sigma, \ldots, f_1^\sigma, f_k, \ldots, f_{a_r+1})$ and $(f_{a_1}^\sigma, f_{a_i -1}^\sigma, \ldots, f_1^\sigma, f_k, \ldots, f_{a_{r-1}+1})$ that wrap around the end of the polynomial.

 A skew-clustering with $a_r = k$, that is with a cluster boundary at the edge of the polynomial, is a \emph{robust skew-clustering}. The corresponding clustering $(a_r, \ldots, a_1, 0)$ is a \emph{robust clustering}.
\end{Def}

The next lemma collects a number of immediate observations that connect skew-clusterings to clusterings, and uses the new freedom of skew-conjugacy to improve cleanups.

\begin{lemma} \label{skew-cluster-basics}
\begin{enumerate}
\item Suppose that $\vec{a}$ is a skew-clustering of a decomposition $\vec{f}$.
 Then there is a clustering $\vec{b}$ of the decomposition $\beta^m \star \vec{f}$ such that $i$ is a cluster boundary of $\vec{a}$ if and only if $(i + m) \mod k$, that is, the remainder of
  $(i+m)$ upon division by $k$,
  is a cluster boundary of $\vec{b}$.
 Similarly, there is a clustering $\vec{c}$ of the decomposition $\phi^m \star \vec{f}$ such that $i$ is a cluster boundary of $\vec{a}$ if and only if $(i - m) \mod k$ is a cluster boundary of $\vec{c}$.
\item If $\vec{a}$ is a skew-clustering of a decomposition $\vec{f}$ and $a_r =k$, then
 $(a_r, \ldots, a_1, 0)$ is a clustering of $\vec{f}$. In this case, we call both $(a_r, \ldots, a_1)$ and $(a_r, \ldots, a_1, 0)$ a \emph{robust clustering} of $\vec{f}$.
\item For any preclustering $(a_r, \ldots, a_1, a_0)$ of any decomposition, $(a_r, \ldots, a_1)$ is a skew-preclustering of the same decomposition.
\item Any skew-preclustering can be refined to a skew-clustering; in particular, every decomposition admits a skew-clustering.
\item Every decomposition admits a skew-clustering.
\item Every decomposition has a plain skew-twist which has a robust clustering.
\item Every robust clustering admits (up to skew-conjugacy!) a cleanup with $L_k = \id$.
\item If $(\vec{L}, \vec{h})$ is a cleanup of a robust clustering $\vec{a}$ of $\vec{f}$ with $L_k = \id$, then the clustering of $\phi^{a_i} \star \vec{f}$ obtained in part (1) is also robust, and reindexing $L_i$ and $h_i$ and applying $\sigma$ as necessary produces a cleanup of it.
\end{enumerate} \end{lemma}

\begin{Rk}
The notion of robust clustering is necessary in that parts (2) and (8) become much more complicated without this extra hypothesis, because a plain skew-twist might break a \textsf{C} cluster into two pieces, and things become complicated if one of the pieces does not contain an odd-degree factor, and particularly complicated if that piece has degree two.
\end{Rk}

\begin{proof}
For part (4), induct on the number of clusters exactly as in the proof of the existence of clusterings, Lemma~\ref{clusteringsexist}.

Only part (7) merits detailed explanation. Take some cleanup $(\vec{L}, \vec{h})$ of a robust skew-clustering. Skew-conjugating, we may absorb the translation part of $L_k$ into $L_0$ and assume without loss of generality that $L_k$ is a scaling.

If there is a \textsf{C} cluster, skew-conjugate to move scaling $L_k$ into $L_1$ and then move it left as in the proof of the existence of cleanups until it sits to the right of a \textsf{C} cluster, where it may stay without violating the definition of ``cleanup''.

If there are no \textsf{C} clusters, skew-conjugating by $\cdot \frac{1}{\lambda}$ and moving the new scaling left as in the proof of the existence of cleanups replaces $L_k$ by $L_k \circ ( \cdot \frac{\lambda^\sigma}{\lambda^{ \deg(f)}})$.  Here again, because we are working over a difference closed field, there is no problem to find $\lambda$ such that $L_k \circ (\cdot \frac{\sigma(\lambda)}{\lambda^{\deg(f)}}) = \operatorname{id}$.   If one is interested only in the case of an algebraic dynamical system, then it suffices to find a $(\deg(f) - 1)^\text{st}$ root of the leading coefficient of $f$ inside the fixed field of $\sigma$.
 \end{proof}

For a robust skew-clustering and a cleanup with $L_k = \id$, it is clear what ``gate at $k$'' should mean, except maybe when $\vec{f}$ is itself a cluster. Recall that the factors are now standing in a circle, so a gate at $k$ and a gate at $0$ are intuitively the same thing.

\begin{Def}
 Let $\vec{a}$ be a robust skew-clustering of a decomposition $\vec{f}$ and let $(\vec{L}, \vec{h})$ be a cleanup of it with $L_k = \id$.

 If $\vec{a}$ has more than one cluster, then $\vec{f}$ with $\vec{a}$ has \emph{a (left-to-right or right-to-left) gate at $k$} if and only if the clustering and cleanup of $\phi^{a_i} \star \vec{f} =  ( f_{a_1}^\sigma, f_{a_1 -1}^\sigma, \ldots, f_1^\sigma, f_k, f_{k-1}, \ldots, f_{a_{1}+1})$ obtained in (8) above has a gate in that direction between $f_1^\sigma$ and $f_{k}$.

 If $\vec{a} = (k)$ has exactly one cluster, then $\vec{f}$ has \emph{a (left-to-right or right-to-left) gate at $k$} if and only if the preclustering $(2k, k, 0)$ of $\vec{f}^\sigma \vec{f}$ has a gate in that direction at $k$.
\end{Def}

Recall that two adjacent clusters with a two-way gate between them can be fused into a single cluster; this has the following convenient consequence.

\begin{lemma} \label{one-cluster-lemma}
 Suppose that $\vec{f}$  is a decomposition of a disintegrated polynomial $f$, and that a skew-clustering of it
 $(a_r, \ldots, a_1)$ has a two-way gate at some $j$.
 Then $\vec{f}$ is a single \textsf{C}-free cluster with at least one non-monomial factor, and admits a cleanup with no linear factors.
\end{lemma}

\begin{proof}
 Lemma~\ref{twodoork} forbids two-way gates between distinct clusters of a clustering, so $j = 1$ and $a_1 = k$.
 If $\vec{f}$ has more than one cluster, this makes $(f_{a_1}^\sigma, f_{a_i -1}^\sigma, \ldots, f_1^\sigma, f_k, \ldots, f_{a_{r-1} +1})$ into a cluster, contradicting the definition of skew-clustering. If $\vec{f}$ is a single cluster, then the two-way gate at $k$ means that (up to skew-conjugacy) $\vec{f}$ admits a cleanup with $L_k = \id$ and
 $L_0 = (\cdot \pm 1)$.  Since $f$ is disintegrated, $\vec{f}$ cannot consist of a single \textsf{C} cluster or a single \textsf{C}-free cluster with only monomial factors.
\end{proof}

The next remark gathers the results we have proved about the interaction of (skew-)clusterings, Ritt swaps, and skew-twists.

\begin{Rk}
Since robust skew-clusterings correspond precisely to robust clusterings, different robust clusterings of skew-linearly equivalent decompositions obey Proposition~\ref{unique-clustering-ntech}: they have the same number of clusters, the same cluster boundaries with gates (possibly in different directions), and cluster boundaries may only differ by one, and then only by one wandering quadratic (this is the case when the gate changes direction). It is clear that the two robust skew-clusterings have the same gates at $k$.

Skew-twists act on skew-clusterings (and their gates) via the first part of Lemma~\ref{skew-cluster-basics}: given a skew-clustering $\vec{a}$ of $\vec{f}$, the corresponding skew-clustering $\vec{b}$ of $\phi^i \star \vec{f}$ has cluster boundaries at $(a_j -i) \mod k$,
and similarly for $\beta_i$. This $\vec{b}$ is robust of and only if $i = a_j$ for some $j$. It is clear that $\vec{f}$ has a gate at the cluster boundary at $a_j$ if and only if $\phi^i \star \vec{f}$ has a gate at the corresponding cluster boundary at $(a_j -i) \mod k$.

As noted in Lemma~\ref{imp-swap-inside}, a Ritt swap inside a cluster (that is, $t_i$ for some $i \neq a_j$ for all $j$) does not affect the clustering or its gates. By skew-twisting until the boundary is inside, it is clear that the new notion of gates at $k$ for a skew-clustering is also unaffected.

Recall (Lemma~\ref{extra-cluster-swap}) that a Ritt swap across clusters (that is, $t_{a_j}$ for some $j$) always involves a wandering quadratic that leaves on cluster and enters the other, changing the direction of the gate at the boundary.

It should now be clear that the number of clusters in a robust skew-clustering, and the presence of a gate between particular clusters, are invariant under skew-linear equivalent, Ritt swaps, and skew-twists by whole clusters, though the indexing of the clusters changes in this last case.
\end{Rk}

The conclusion of the next Lemma~\ref{goodplaintwist} is used in Lemma~\ref{happyw1w2lem} to bound the number of consecutive $\psi$'s or $\gamma$s in a word from the border guard monoid acting on $\vec{g}$.

\begin{lemma} \label{goodplaintwist}
Every decomposition $\vec{f}$ of a disintegrated polynomial has a plain skew twist $\vec{g} := \phi^i \star \vec{f}$ which has a robust clustering $\vec{a}$ with one of the following properties:\begin{enumerate}
\item $\vec{a}$ has no gates at $k$
\item $\vec{a}$ has a one-way
 gate at $k$
\item $\vec{g}$ is a single \textsf{C}-free cluster, $\vec{a}$ has a two-way gate at $k$ and $g_k$ is not a monomial.
\end{enumerate} \end{lemma}

\begin{proof}
 We know that $\vec{f}$ has a skew-clustering $\vec{b}$.
 If some $b_j$ satisfies one of the first two items in the conclusion, let $i := b_j$.
 Otherwise, Lemma~\ref{one-cluster-lemma} says that any plain skew-twist of $f$ is a single cluster, and one of the factors $f_i$ is not a monomial. In any case, the plain skew-twist $\phi^i \star \vec{f}$ or $\beta^{k-i} \star \vec{f}$ with the corresponding skew-clustering works. \end{proof}

\begin{lemma} \label{happyw1w2lem}
 If $\vec{g}$ and $\vec{a}$ satisfy one of the three conclusions of Lemma~\ref{goodplaintwist}, and $w \in \borgarmonoid_k$ such that $w \star \vec{g} =: \vec{h}$ is defined, then $\vec{h}$ satisfies the same conclusion. If furthermore $w$ contains no instances of $\beta$ (respectively, $\gamma$), then the number of instances of $\gamma$ (respectively, $\beta$) in $w$ is bounded by $0$ in the first case of Lemma~\ref{goodplaintwist}, by $1$ in the second case, and by the degree of $g_k$ in the last case. \end{lemma}

These bounds are useful because in any word in the border guard monoid, the $\psi$s can be separated from the $\beta$s in the following sense.

\begin{lemma}
\label{w1w2lemma}
For any word $w \in \borgarmonoid_k$, there are $w_i \in \borgarmonoid_k$ such that $w \approx w_2 w_1$
and $\gamma$ does not appear in $w_1$
and $\psi$ does not appear in $w_2$.
\end{lemma}

\begin{proof}
 Given $w \in \borgarmonoid_k$, we find an equivalent word $w'$ that has no substrings of the form $\psi u \gamma$ for some $u \in \rittmonoid_{k-1}$. Clearly, $w'$ is the desired word. To construct $w'$, we prove a

 \emph{Claim:} for any $u \in \rittmonoid_{k-1}$ there is a word $v' \in \rittmonoid_{k-1}$ such that $\psi u \gamma \approx v'$ or $\psi u \gamma \approx \gamma t_{k-2} \psi v'$.

  Then replacing a substring $\psi u \gamma$ by one of these does not increase the number of instances of $\psi$ and $\gamma$ in a word, and straightens out one $\psi$, $\gamma$ pair in the wrong order. Thus, after finitely many such operations we obtain the desired $w'$.

\emph{Proof of Claim:}
 Without loss of generality, we may assume that $u \in \rittmonoid_{k-1}$ is in reverse first canonical form, i.e. either $u =v$ or $u= t_1 v$ where $t_1$ does not appear in $v \in \rittmonoid_{k-1}$. Then $\psi u \gamma \approx v'$ in the first case, and $\psi u \gamma \approx \gamma t_{k-2} \psi v'$ in the second, for some $v' \in \rittmonoid_{k-1}$.
\end{proof}

\subsection{Characterization of skew-invariant curves}
\label{triumph-sec}

Finally, we show that every skew-invariant curve is encoded by some word in the augmented skew-twist monoid $\skewmonoid^+_k$, give a normal form for such words, and thereby obtain a normal form for the skew-invariant curves.

\begin{prop}
For any disintegrated polynomials $f$ and $g$ and any irreducible $(f,g)$-skew-invariant plane curve $\krivb$, there are a decomposition $(f_k, \ldots, f_1)$ of $f$, a word $w \in \skewmonoid^+_k$, and a curve $\kriva$ encoded by $w \star \vec{f}$ such that $\krivb \subset \kriva$.
\end{prop}

\begin{proof}
By Proposition~\ref{curvetocomposition}, $\krivb = \rho \circ (\pi^{-1})$ for some polynomials $\pi$ and $\rho$; here $\pi^{-1}$ denotes the converse relation to the graph of $\pi$ and $\circ$ is curve composition in the sense of Definition~\ref{curve-compose-def}.

By successively factoring out single skew-twists from both sides of $\pi$, it is clearly possible to write $\pi := \pi_3 \circ \pi_2 \circ \pi_1$ so that $\pi_1$ is a skew-twist from $h$ to some polynomial $h_f$, and $\pi_3$ is a skew-twist from some polynomial $f_h$ to $f$, and the middle diagram $\pi_2^\sigma \circ h_f = f_h \circ \pi_2$ satisfies the hypotheses of Theorem~\ref{notskewtwist}: that is, $h_f$ and $\pi_2$ share no initial compositional factors, and $f_h$ and $\pi_2^\sigma$ share no terminal compositional factors. In like manner, we may write $\rho = \rho_3 \circ \rho_2 \circ \rho_1$.

Now the graphs of $\rho_i$ and the converses of graphs of $\pi_i$ are encoded by words in $\skewmonoid^+_k$, so the composition $\krivb = \rho_3 \circ \rho_2 \circ \rho_1 \circ (\pi_1^{-1}) \circ (\pi_2^{-1}) \circ (\pi_3^{-1})$ is encoded by the concatenation of these words.
\end{proof}

\begin{prop}
For any word $w \in \skewmonoid^+_k$, there are integers $M, N \in \mathbb{N}$ with $M < k$, words $w_2, w_1 \in \borgarmonoid_k$ with no instances of $\gamma$ in $w_1$ and no instances of $\psi$ in $w_2$, and a word $\tilde{w}$ consisting only of $\epsilon_p$ and $\delta_p$ for various $p$ such that $w \approx \phi^N \tilde{w} w_2 w_1 \phi^M$
or $w \approx \beta^N \tilde{w} w_2 w_1 \phi^M$. \end{prop}

\begin{proof}
By Lemma~\ref{goodplaintwist}, we find $M < k$ such that $\vec{g} := \phi^M \star \vec{f}$ satisfies one of the three possible conclusions of that lemma. By Lemma~\ref{biggercorr}, $w \approx w \beta^M \phi^M$.

By Lemma~\ref{new-gen-equivs}, there are $\hat{w} \in \skewmonoid_k$ and $\tilde{w}$ consisting only of $\epsilon_p$ and $\delta_p$ for various $p$ such that $w \beta^M \approx \tilde{w} \hat{w}$, and by Remark~\ref{plus-concat-compose}
$w \approx \tilde{w} \hat{w} \phi^M$.

By Proposition~\ref{border_guard_prop} and Lemma~\ref{w1w2lemma}, there are $N$ and $w_i$ as desired such that $\hat{w} \approx \phi^N w_2 w_1$ or $\hat{w} \approx \beta^N w_2 w_1$. Again by Remark~\ref{plus-concat-compose},
$w \approx \tilde{w} \phi^N w_2 w_1 \phi^M$ or $w \approx \tilde{w} \beta^N w_2 w_1 \phi^M$

One last application of Lemma~\ref{new-gen-equivs} finishes the proof. \end{proof}

\begin{Def} \label{mono-curve-def}
A \emph{monomial curves} is plane curve $C \subseteq \AA^2$
 defined by $x^n = y^m$ for some $m, n \in \mathbb{N}$.
 \end{Def}

Recall (Remark~\ref{reintroduce-linear}) that curves encoded by words in $\skewmonoid^+_k$ are only defined up to a linear terminal compositional factor $L$ which must be added manually.

\begin{theorem} \label{realchar}
For any disintegrated polynomials $f$ and $g$, any irreducible $(f,g)$-skew-invariant plane curve $\krivb$ is an irreducible component of $\hat{g} \circ \tilde{g} \circ \kriva \circ \krivd \circ \tilde{f}$ where
\begin{itemize}
\item[$\tilde{f}$] is (the graph of) an initial compositional factor of $f$ or linear $L$;
\item[$\krivd$] is a monomial curve encoded by a word in the border guard monoid whose degrees are bounded by Lemma~\ref{happyw1w2lem}, and in any case by the degree of $f$;
\item[$\kriva$] is a monomial curve encoded by a word in $\epsilon_p$ and $\delta_p$ whose degrees are bounded the product of in- and out-degrees of a compositional factor of $f$, and in any case by the degree of $f$;
\item[$\hat{g}$ and $\tilde{g}$] are one of the following, for some $N \in \mathbb{N}$: \begin{itemize}
 \item $\tilde{g}$ is the converse of the graph of an initial compositional factor of $g^{\sigma^N}$ or linear, and $\hat{g}$ is the converse of the graph of $g^{\lozenge N}$;
 \item $\tilde{g}$ is the graph of an initial compositional factor of $g^{\sigma^{-N}}$ or linear, and $\hat{g}$ is the graph of $( g^{(\sigma^{-N})} )^{\lozenge N}$.
\end{itemize}
\end{itemize}
\end{theorem}

\begin{Rk} \label{simpler-answerk}
If some factor of $\vec{f}$ is unswappable, or if some robust clustering of $\vec{f}$ has no gate at some cluster boundary, then $\krivd$ and $\kriva$ above must be diagonals. Then one of $\hat{g} \circ \tilde{g}$ and $\tilde{f}$ cancels with part of the other, and the whole $(f,g)$-skew-invariant plane curve $\krivb$ is the graph of something like $\hat{g} \circ \tilde{g}$. More generally, by Proposition~\ref{notskewtwist} $\kriva$ is the diagonal unless $\vec{f}$ is skew-conjugate to a polynomial of the form $x^k \cdot u(x^\ell)^n$ for some integers $k \geq 1$ and $\ell$, $n$ such that $\ell n \geq 2$.\end{Rk}

The characterization of $(f,g)$-\emph{invariant} curves becomes particularly simple when the two polynomials are the same.

\begin{theorem} \label{hhprop}
Fix an algebraic dynamical system $(\AA^2, (h, h))$ for a disintegrated polynomial $h$. Any irreducible $(h,h)$-invariant plane curve $\krivb$ is the graph, or the converse of the graph, of $L \circ \tilde{h}^\ell$ for some linear linear $L$ that commutes with some compositional power of $h$
and some $\tilde{h}$ such that $\tilde{h}^{\circ r} = h$ for some $r$.\end{theorem}

\begin{proof}
Let $f := g := h$, and let $\hat{g}$, $\tilde{g}$ $\kriva$ and $\krivd$ and  $\tilde{f}$ be as in the conclusion of that theorem so that $\krivb$ is an irreducible component of $\hat{g} \circ \tilde{g} \circ \kriva \circ \krivd \circ \tilde{f}$.

Nontrivial $\kriva$ irreparably damages in- and out-degrees of factors $f_i$ in a decomposition of $f=h$. This cannot be fixed by $\krivd$ or by Ritt swaps inside the decomposition because the monomials of $\kriva$ are not among the $f_i$ (see Definition~\ref{decomp-inout-def}). If the same monomial occurs in both directions in $\kriva$, then $\kriva$ is reducible and its factors are given by replacing $x^p = y^p$ in its definition by $x = \xi y$ for various $p$th roots of unity $\xi$.

Non-trivial $\krivd$ irreparably damages the gates of a clustering of $\vec{h}$ in the second case of Lemma~\ref{goodplaintwist}, and irreparably damages the in- and out-degrees of the factor $f_i$ guarding the border in the third case of Lemma~\ref{goodplaintwist}.

Now, as in Remark \ref{simpler-answerk}, one of $\hat{g} \circ \tilde{g}$ and $\tilde{f}$ cancels part of the other, leaving behind the graph of a ``fractional compositional power of $h$'', since $f = g  =h$ is defined over the fixed field of $\sigma$. That is, $\krivb$ is the graph (or the converse of the graph) of $h_0 \circ h^{\circ s}$ for some $s$, where $h_0$ is linear or the identity $h \circ h_0 = h_0 \circ h$ is a plain skew-twist. That is, for some $h_1$, both $h_0 \circ h_1 = h$ and $h_1 \circ h_0 = h$.

The theorem now follows from Ritt's theorem on commuting rational functions~\cite{Ritt-permutable}.
\end{proof}

Recall (see Corollary~\ref{nonorthotopoly}) that for every $f$ and $g$ there is some $h$ so that all $(f,g)$-invariant curves can be understood in terms of $(h,h)$-invariant curves.

\begin{lemma} \label{reducetohh}
Suppose that $f$ and $g$ are disintegrated polynomials and there is an irreducible $(f,g)$-invariant curve.
Then there are polynomials $\pi$, $\rho$, and $h$ such that $\pi \circ h = f \circ \pi$ and $\rho \circ h = g \circ \pi$, and any irreducible $(f,g)$-invariant curve $\kriva$ is of the form $(\pi, \rho) (\krivb_0)$ for some irreducible $(h,h)$-periodic curve $\krivb_0$.
\end{lemma}

\begin{proof}
Proposition~\ref{curvetocomposition} produces polynomials $\pi$, $\rho$, and $h$ as required. Proposition~\ref{pushpullinv} applied to the map of algebraic dynamical systems
$(\pi, \rho) : (\AA^2, (h, h)) \rightarrow (\AA^2, (f, g))$ finishes the proof: any irreducible component $\krivb_0$ of $( (\pi, \rho)^{-1} (\kriva) )_\per$ works.
\end{proof}

An $(h,h)$-periodic curve is $(h^{\circ m},h^{\circ m})$-invariant for some $m$, so Theorem~\ref{hhprop} almost applies to the conclusion of Lemma~\ref{reducetohh}.

\begin{theorem}
\label{centralthm}
Suppose that $f$ and $g$ are disintegrated polynomials and there is an irreducible $(f,g)$-invariant curve.
Then there are polynomials $\pi$, $\rho$, and $h$ such that $\pi \circ h = f \circ \pi$ and $\rho \circ h = g \circ \pi$, and any irreducible $(f,g)$-invariant curve $\kriva$ is of the form $(\pi, \rho) (\krivb_0)$ where $\krivb_0$ is the graph, or the converse of the graph, of $L \circ \tilde{h}^\ell$ for some linear $L$ that commutes with some compositional power of $h$
and some $\tilde{h}$ such that $\tilde{h}^{\circ r} = h^{\circ m}$, for some $r$ and $m$.
\end{theorem}

\section{Applications}
\label{Applications}
In this section we use  our characterization of skew-invariant varieties to answer some open questions about the model theory of difference fields and the arithmetic of algebraic dynamical systems.

\subsection{Disintegrated minimal sets in ACFA}
\label{newsectionstrmin}

In this subsection we address some fine structural questions about minimal sets
in difference closed fields of characteristic zero.
Specifically,  we consider minimal sets of the form $(\AA^1,f)^\sharp$.
We show that if the isomorphism class of $({\AA}^1, f)$ is defined over
 the fixed field of some power of the
distinguished automorphism, then nonorthogonality to $({\AA}^1, f)^{\sharp}$
is definable.  Conversely, whenever the moduli point of $({\AA}^1,
f)$ is transcendental over the fixed field, it is undefinable.
We close out this section by showing that when $f$ is disintegrated,
$({\AA}^1, f)^{\sharp}$ has Morley rank one.

\begin{nota}
We fix a difference closed field $(\UU,\sigma)$ of characteristic zero.  All of the objects we discuss,
such as polynomials, varieties, definable sets, \emph{etc.}, are defined over $\UU$.  Sometimes, we abuse notation writing
expressions like ``$a \in (X,f)^\sharp$'' to mean that $X$ is an algebraic variety over $\UU$, $f:X \to X^\sigma$ is a dominant regular map,
and $a \in (X,f)^\sharp(\UU,\sigma)$ is a $(\UU,\sigma)$-rational point of the $\sigma$-variety $(X,f)$.
\end{nota}

\begin{nota}
When we speak of properties of polynomials being definable, we are considering the polynomial ring as an ind-definable set.  More
concretely, we say that some class ${\mathcal K}$ of $n$-tuples of polynomials is definable if there is some natural number $d$
and definable set ${\mathsf K} \subseteq M_{(d+1) \times n}(\UU)$ for which
$$
{\mathcal K} = \{ (\sum_{i=0}^d a_{i,1} x^i, \ldots, \sum_{i=0}^d a_{i,n} x^i) ~:~ (a_{i,j}) \in {\mathsf K} \} \text{ .}
$$
\end{nota}

\begin{lem}
\label{skewconjdef}
For any given natural number $d$ the relation that two polynomials of degree $d$ are skew-conjugate is  definable.
\end{lem}

\begin{proof}
The action of the group of linear polynomials by skew-conjugation is definable with respect to our presentation of the space of degree $d$ polynomials
as a constructible subset of $\AA^{d+1}(\UU)$ .
\end{proof}

\begin{lem}
\label{modulidef}
If $f$ is a polynomial which is skew-conjugate to $f^\sigma$, then there is a polynomial $g$ which is
skew-conjugate to $f$ and satisfies $g^\sigma = g$.
\end{lem}

\begin{proof}
  By hypothesis, there is some linear $\lambda$ with $f^{\sigma} =
  \lambda^{\sigma} \circ f \circ \lambda^{- 1}$. As pre-composition with
  $\lambda^{- 1}$ defines an automorphism of the space of degree $d$
  polynomials, from the geometric axiom for difference closed fields, there is
  some linear $\mu$ satisfying $\mu^{\sigma} \circ \lambda = \mu$. Set $g : =
  \mu^{\sigma} \circ f \circ \mu^{- 1}$. Then $g^{\sigma} = \mu^{\sigma^2}
  \circ f^{\sigma} \circ (\mu^{\sigma})^{- 1} = \mu^{\sigma^2} \circ
  \lambda^{\sigma} \circ f \circ \lambda^{- 1} \circ (\mu^{\sigma})^{- 1} =
  (\mu^{\sigma} \circ \lambda)^{\sigma} \circ f \circ (\mu^{\sigma} \circ
  \lambda)^{- 1} = \mu^{\sigma} \circ f \circ \mu^{- 1} = g$.
\end{proof}

\begin{Rk}
The above lemmata hold more generally.  For example, if $X = X^\sigma$ is an algebraic variety which descends to the fixed field, $\operatorname{Aut}(X)$ is represented by a connected algebraic group, and $(X,f)$ is a $\sigma$-variety on $X$ which is isomorphic as a
$\sigma$-variety to $(X,f^\sigma)$, then there is a map $g:X \to X$ for which $g^\sigma = g$ and $(X,g)$ is isomorphic to $(X,f)$.
\end{Rk}

With these observations in place, let us prove a theorem on definability of
nonorthogonality.

\begin{theorem} \label{hruit-thm}
  For a nonconstant polynomial $f$, the set of polynomials $g$ with $({\AA}^1,
  g) \not \perp ({\AA}^1, f)$ is definable if and only if  $f$ is not skew-conjugate to $f^{\sigma^n}$
  for every positive integer $n \in \ZZ_+$.
\end{theorem}

\begin{proof}
If $f$ is linear, then the set of polynomials orthogonal to $f$
 is precisely the
set of linear polynomials, which is clearly definable. Likewise, if $f$ is
skew-conjugate to $P_n$ or $C_n$ for $n = \deg (f)$, then $({\AA}^1, f)
\not \perp ({\AA}^1, g)$ if and only if $g$ is skew-conjugate to $P_n$ or
$C_n$. As this class is also definable, we may restrict to the study of
disintegrated $f$. By Theorem~\ref{centralthm}, $g$ is non-orthogonal to $f$ just in
case the skew-conjugacy class of some decomposition of $g$ is in the
image of some (any) decomposition of $f$ under the action of the augmented skew twist monoid $ST^+_k$.
As this monoid is countable, there are at most countably
many skew-conjugacy classes of polynomials nonorthogonal to $f$. The map
$f^{\lozenge n} : ({\AA}^1, f) \to ({\AA}^1, f^{\sigma^n})$ witnesses
nonorthogonality between $f$ and $f^{\sigma^n}$. Thus, if $f$ is not
skew-conjugate to any of its images under $\sigma^n$, we see that there are
exactly $\aleph_0$ skew-conjugacy classes of polynomials nonorthogonal to $f$.
As every infinite definable set must be uncountable, we conclude that
nonorthogonality to such an $f$ is not definable.

Finally, consider the case when $f$ is disintegrated and is skew-conjugate to
some $f^{\sigma^n}$, then by Lemma~\ref{modulidef} we my assume that $f$ itself is equal to
$f^{\sigma^n}$.  Considering Theorem~\ref{centralthm} again we see that if $g$ were
nonorthogonal to $f$, then this nonorthogonality would be witnessed by the
composition of skew-twists, monomial correspondences (with degrees bounded by
$\deg(f)$), another skew-twist, and graphs of skew-composites $f^{\lozenge m}$ with $m < n$.
As the set of such curves is finite (up to isomorphism), the set
of polynomials skew-conjugate to $f$ is definable.
\end{proof}

\begin{Rk}
Curiously, if $f$ is a polynomial defined over some small difference subfield $K$ of $\UU$, then
model theoretic algebraic closure defines a locally finite closure operator on $(\AA^1,f)^\sharp \smallsetminus \operatorname{acl}(K)$
 just in case the skew-conjugacy class of $f$ is transcendental over the fixed field.   If we further assume that $K$ is
 finitely generated as a difference field, then it is not unreasonable to guess that $(\AA^1,f)^\sharp \cap \operatorname{acl}(K)$ is
 finite. (This is plainly false when $f$ is skew-conjugate to $f^{\sigma^n}$ for some $n \in \ZZ_+$.)
\end{Rk}

Because quantifier elimination fails for $\text{ACFA}$, definable sets of $D$-rank one need not
have Morley rank one.  Indeed, it is easy to see that the fixed field is minimal, but its induced structure is unstable.  More sophisticated examples
of stable minimal sets of infinite multiplicity constructed from Hecke correspondences appear in~\cite{CH-ACFA1}.  Here we show that if $f$ is a disintegrated
polynomial, then $(\AA^1,f)^\sharp$ has Morley rank one and is ``usually'' strongly minimal.

\begin{lem}
\label{skewonto}
Let $X$ be an algebraic variety for which $X = X^\sigma$.  Let $f:X \to X$ and $g:X \to X$ be two self-maps.  Define
$\phi := f \circ g$ and $\psi:= g^\sigma \circ f$.  Then $g:(X,\phi)^\sharp \to (X,\psi)^\sharp$ is onto.

$$
\begin{CD}
X @>{g}>> X @>{f}>> X \\
@V{\phi}VV @VV{\psi}V  @VV{\phi^\sigma}V \\
X @>{g^\sigma}>>  X @>{f^\sigma}>>  X
\end{CD}
$$

\end{lem}

\begin{Rk}
This proof works for any inversive difference field.  It is not necessary to work with a difference closed field.
\end{Rk}

\begin{proof}
Let $P \in (X,\psi)^\sharp$.  Set $Q := \sigma^{-1} f(P)$.  Let us note that the identity $\phi^\sigma \circ f = f^\sigma \circ g^\sigma \circ f = f^\sigma \circ \psi$ shows that $f:(X,\psi) \to (X,\phi^\sigma)$ is a map of $\sigma$-varieties so that
$f(P) \in (X,\phi^\sigma)^\sharp$.  Applying $\sigma^{-1}$, we have $Q \in (X,\phi)^\sharp$.  We compute
$g (Q) = g \sigma^{-1} f (P) = \sigma^{-1} g^{\sigma} f (P) = \sigma^{-1} \psi(P) = \sigma^{-1} \sigma (P) = P$.
\end{proof}

\begin{lem}
\label{iteronto}
If $(X,f)$ is any $\sigma$-variety and $n \in \ZZ_+$, then the map $f^{\lozenge n}:(X,f)^\sharp \to (X^{\sigma^n},f^{\sigma^n})$ is
onto.
\end{lem}
\begin{proof}
Working by induction on $n$ it suffices to consider the case of $n = 1$.  Let $a \in (X^{\sigma},f^{\sigma})^\sharp$.  Set $b := \sigma^{-1}(a)$.
Then $f(b) = f(\sigma^{-1}(a)) = \sigma^{-1} f^{\sigma}(a) = \sigma^{-1} \sigma (a) = a$.
\end{proof}

With the next lemma we say we characterize the image of a power map.

\begin{lem}
\label{powerimage}
Given a nonconstant polynomial $u$, positive integer $k$, and a prime $\ell$, we set $f(x) := x^k u(x^\ell)$ and $g(x) := x^k u(x)^\ell$.
 In general, $(\AA^1,g)^\sharp$ is the image of $P_\ell$ on
$\bigcup_{\zeta \in \mu_\ell} (\AA^1,\zeta f)^\sharp$.
If $\sigma$ does not act on $\mu_\ell$, the group of $\ell^\text{th}$ roots of unity,
 by raising to the $k^\text{th}$ power, then $P_\ell:(\AA^1,f)^\sharp \to
(\AA^1,g)^\sharp$ is surjective.
\end{lem}
\begin{proof}
Let $a \in (\AA^1,g)^\sharp$.  Let $b \in AA^1(\UU)$ be any solution to $P_\ell(b) = a$.  From the equation
$P_\ell \circ f = g \circ P_\ell$, we see that $P_\ell (f(b)) = g(a) = \sigma(a)$ while we also know that $P_\ell(\sigma(b)) = \sigma(a)$.
Hence, there is some $\xi \in \mu_\ell$ for which $\xi f(b) = \sigma(b)$.

We assume now that $\sigma$ does not act by raising the
$k^\text{th}$ power on $\mu_\ell$. Thus, the map $\mu_\ell \to \mu_\ell$ given by $\zeta \mapsto \sigma(\zeta)/\zeta^k$ is onto. Thus, we
may choose $\zeta \in \mu_\ell$ with $\sigma(\zeta)/\zeta^k = \xi$. We compute  $f(\zeta b) = (\zeta b)^k u( (\zeta b)^\ell) =
\zeta^k b^k u(b^\ell) = \zeta^k f(b) = \zeta^k \xi \sigma(b) = \zeta^k \xi \sigma(\zeta)^{-1} \sigma(\zeta b) = \sigma(\zeta b)$.
Thus, $\zeta b \in (\AA^1,f)^\sharp$ and $P_\ell(\zeta b) = a$.
\end{proof}

Combining the above lemmata we conclude that disintegrated sets of the form $(\AA^1,f)^\sharp$
have Morley rank one.

\begin{theorem}
\label{finiteRM}
If $f$ is a disintegrated polynomial, then $(\AA^1,f)^\sharp$ has Morley rank one.
\end{theorem}
\begin{proof}
The quantifier-elimination to bounded existential quantifiers for $\text{ACFA}$ together with the work around finite $\sigma$-stable extensions in \cite{CH-ACFA1} imply that every infinite definable subset of $(\AA^1,f)^\sharp$ is (up to a finite sets) is a finite union of sets of the form $h (\AA^1,g)^\sharp$ where $h:(\AA^1,g) \to (\AA^1,f)$ is a map of $\sigma$-varieties.

By Theorem~\ref{centralthm}, $h$ may be expressed as a composite of a sequence of skew-twists, power maps of degree bounded by $\deg(f)$ and maps of the
form $k^{\lozenge n}$.  By Lemmata~\ref{skewonto} and~\ref{iteronto}, the maps of the first and third type are always onto.  By Lemma~\ref{powerimage}, there are at most $\deg(f)$ many distinct sets arising from the power maps.
\end{proof}

\subsection{Density of dynamical orbits}
\label{ddo}
In this subsection we apply Theorem~\ref{hhprop} to deduce a version of a conjecture of Zhang on the density of dynamical orbits.  Let us recall Zhang's conjecture.

\begin{conj}[Conjecture 4.1.6 of~\cite{Zh}]
\label{Zhconj}
Let $f:X \to X$ be a polarizable dynamical system over a number field $k$. Then there is point $a \in X(k^{alg})$ algebraic over $k$ whose forward orbit ${\mathcal O}_f(a) := \{ f^{\circ n}(a) : n \in {\mathbb Z}_+ \}$ is Zariski dense in $X$.
\end{conj}

The dynamical systems we have been considering, namely, $(\AA^n,\Phi)$ given by coordinatewise univariate polynomials as above, do not fit Conjecture~\ref{Zhconj} as stated for a couple of reasons.  First, as $\AA^n$ is affine, no dynamical system on $\AA^n$ can be polarized.   More seriously, even if we pass to a projective closure, the hypothesis of polarizability forces all of the polynomials involved to have the same degree.  We shall prove that there are dense orbits without these restrictions.

In light of our results and a geometric version of Conjecture~\ref{Zhconj} due to Amerik and Campana~\cite{AC}, we propose a more general conjecture on the density of dynamical orbits.

\begin{conj}
\label{genZhconj}
Let $K$ be an algebraically closed field of characteristic zero, $X$ an irreducible algebraic variety over $K$, and $\Phi:X \to X$ a \emph{rational} self-map.  We suppose that there does not exist a positive dimensional algebraic variety $Y$  and dominant rational map $g:X \to Y$ for which $g \circ \Phi = g$ generically.  Then there is some point $a \in X(K)$ with a Zariski dense forward orbit.
\end{conj}

\begin{Rk}
In~\cite{ABR}, Amerik, Bogomolov and Ravinsky prove some instances of Conjecture~\ref{genZhconj},
without imposing any polarizability hypotheses, but instead arguing from the local behaviour of the dynamical system.
\end{Rk}

We shall prove the instance of Conjecture~\ref{genZhconj} in which $X$ is affine space and $\Phi$ is given by a sequence of univariate polynomials.

\begin{theorem}
\label{Zhangconjdense}
Let $K$ be a field of characteristic zero, $f_1, \ldots, f_n \in K[x]$ nonconstant polynomials over $K$ in one variable.  Suppose that the linear polynomials amongst the $f_i$'s are independent in the sense
Definition~\ref{independentpoly}.  Let $\Phi:\AA^n_K \to \AA^n_K$ be given by
$(x_1, \ldots, x_n) \mapsto (f_1(x_1), \ldots, f_n(x_n))$.  Then there is a point $a \in \AA^n(K)$ for which $\cO_\Phi(a)$ is Zariski dense.
\end{theorem}

\begin{Rk}
As one sees from the proof, in some sense almost every point in $\AA^n(K)$ has a Zariski dense orbit.  We do not pursue the issue of giving a quantitative treatment of this observation.
\end{Rk}

\begin{Rk}
As the reader will see, the notion of \emph{independence} is exactly what is required so that there is no dominant
map from $(\AA^n,\Phi)$ to a positive dimensional trivial algebraic dynamical system.  We do not pretend that the inclusion of linear polynomials in this statement is deep, but we have included them as there is little extra work involved in doing so and they round out the statement.
\end{Rk}

\begin{Rk} \label{Rk67}
Theorem~\ref{Zhangconjdense} may be read as saying that there are points $a \in \AA^n(K)$ having the property that for no positive integer $N$ is $\Phi^{\circ N}(a)$ contained in any proper $\sigma$-subvariety of $(\AA^n,\Phi)$ when $K$ is treated as a difference field with $\sigma = \operatorname{id}_K$.  In fact, we will prove Theorem~\ref{Zhangconjdense} by explicitly describing the irreducible $\sigma$-subvarieties of $(\AA^n,\Phi^{\circ M})$ for all $M \in \ZZ_+$ and then observing that there are points in $\AA^n(K)$ whose forward orbits miss all such $\sigma$-subvarieties.
\end{Rk}

We prove Theorem~\ref{Zhangconjdense} as a consequence of a number of simple lemmata.

\begin{lemma}
\label{invardense}
Let $f:X \to X$ be an algebraic dynamical system over some field $K$ with $X$ being irreducible.  A point $a \in X(K)$ has a Zariski dense forward orbit if and only if there is no natural number $m$ and proper $f$-invariant subvariety (not necessarily irreducible) of $X$ containing $f^{\circ m}(a)$.
\end{lemma}
\begin{proof}
For any point $a \in X(K)$, as $f(\cO_f(a)) = \cO_f(f(a)) \subseteq \cO_f(a)$, for $m \gg 0$ the variety $\overline{ \cO_f (f^{\circ m}(a))}$ is an $f$-invariant subvariety of $X$.  Hence, if $\cO_f(a)$ is not Zariski dense in $X$, then $\overline{\cO_f(f^{\circ m}(a))}$ is a proper $f$-invariant subvariety of $X$.  Conversely, if $f^{\circ m}(a) \in Y \subsetneq X$ and $Y$ is $f$-invariant, then $\cO_f(a) \subseteq Y(K) \cup \{ f^{\circ i}(a) : 0 \leq i \leq m \}$ so that $\overline{\cO_f(a)} \subseteq Y \cup \{ f^{\circ i}(a) : 0 \leq i \leq m \} \subsetneq X$.
\end{proof}

\begin{lemma}
\label{iterdense}
If $f:X \to X$ is an algebraic dynamical system over some field $K$, $X$ is irreducible, and $a \in X(K)$ has a Zariski dense forward orbit, then for any $m \in {\mathbb Z}_+$, $X = \overline{\cO_{f^{\circ m}}(a)}$
\end{lemma}

\begin{proof}
For $i = 0, \ldots, m-1$, let $Z_i := \overline{\cO_{f^{\circ m}}(f^{\circ i}(a))}$.  Then as $\cO_f(a) = \bigcup_{i=0}^{m-1} \cO_{f^{\circ m}}(f^{\circ i}(a))$, we have $X = \bigcup_{i=0}^{m-1} Z_i$.  Hence, $X = Z_i$ for some $i$.  As $X$ has a dense $f$-orbit, the map $f:X \to X$ is necessarily dominant (otherwise, $\overline{\cO_f(a)} \subseteq \{ a \} \cup \overline{f(X)} \subsetneq X$).  As $f$ maps $Z_j$ to $Z_{j+1 \mod{m}}$, we must have $X = Z_j$ for all $j$. In particular, $X = Z_0 = \overline{\cO_{f^{\circ m}}(a)}$.
\end{proof}

\begin{lemma}
\label{orthdense}
Suppose that $f:X \to X$ and $g:Y \to Y$ are algebraic dynamical systems over the field $K$, $(X,f) \perp (Y,g)$, and that there are rational points $a \in X(K)$ and $b \in Y(K)$ with $\overline{\cO_f(a)} = X$ and $\overline{\cO_g(b)} = Y$.  Then $\overline{\cO_{(f,g)}(a,b)} = X \times Y$.
\end{lemma}
\begin{proof}
Let $Z := \overline{\cO_{(f,g)}(a,b)}$ be the Zariski closure of the forward $(f,g)$-orbit of $(a,b)$.  As $(f,g) (\cO_{(f,g)}(a,b)) \subseteq \cO_{(f,g)}(a,b)$, the variety $Z$ is $(f,g)$-invariant.  As $(X,f) \perp (Y,g)$, $Z$ must be a finite union of varieties of the form $A \times B$ where $A \subseteq X$ is $f$-invariant and $B \subseteq Y$ is $g$-invariant. Let $A \times B$ be a component containing $(a,b)$.  By Lemma~\ref{iterdense}, $X = \overline{\cO_{f}(a)} \subseteq A \subseteq X$ and $Y = \overline{\cO_{g}(b)} \subseteq B \subseteq Y$.  Hence, $X \times Y = \overline{\cO_{(f,g)}(a,b)}$.
\end{proof}

\begin{lemma}
\label{integrality}
Let $K$ be a field of characteristic zero and $f$ and $g$ two disintegrated polynomials over $K$.  Then there is a
 point $(a,b) \in \AA^2(K)$ for which $\cO_{(f,g)}(a,b)$ is Zariski dense in $\AA^2_K$.
\end{lemma}

\begin{proof}
By Theorem~\ref{nonorthotopoly} there are a natural number $m$, a polynomial $h$ and dominant maps of dynamical systems
$\rho:(\AA^1,h) \to (\AA^1,f^{\circ m})$ and $\pi:(\AA^1,h) \to (\AA^1,g^{\circ m})$.  It follows from Ritt's theorem
on polynomials with common iterates~\cite{Ritt-iteration} that there is a maximal $k$ for which we may write
$h = \widetilde{h}^{\circ k}$.  Let $R$ be a finitely generated subring of $K$ over which $\widetilde{h}$, $\pi$, $\rho$, and all symmetries
of $\widetilde{h}$ are defined and which contains the multiplicative inverse of the leading coefficient
of each of these polynomials.   Let $\widetilde{R}$ be the integral closure of $R$ in its field of
fractions, regarded as a subfield of $K$.   Let $\hat{a} \in \widetilde{R}$ and $\hat{b} \in K \smallsetminus \widetilde{R}$ so that
  neither $\hat{a}$ nor $\hat{b}$ is $h$-pre-periodic.  Set $a := \pi(\hat{a})$ and $b := \rho(\hat{b})$.

We claim that there is no (possibly reducible) weakly $(h,h)$-invariant curve $C$ with $(\hat{a},\hat{b}) \in C(K)$. As
neither $\hat{a}$ nor $\hat{b}$ is preperiodic, we see that we may assume that each component of $C$ projects dominantly in
both directions.  If $(\hat{a},\hat{b}) \in C(K)$, then for $n \gg 0$, we would have $(h^{\circ n}(\hat{a}),h^{\circ n}(\hat{b})) \in C_\per(K)$ so
 that $(h^{\circ n}(\hat{a}),h^{\circ n}(\hat{b}))$ would lie on an irreducible $(h^{\circ nN},h^{\circ nN})$-invariant
  curve for some $N \gg 0$.  By Theorem~\ref{hhprop}, such a curve is defined by $y = L \circ \widetilde{h}^\ell(x)$ or
  $x = L \circ \widetilde{h}^\ell(y)$.  Neither such curve can contain a $K$-rational point of the form $(c,d)$ with $c \in R$ and
  $d \notin \widetilde{R}$, which is exactly the form of $(h^{\circ n}(\hat{a}),h^{\circ n}(\hat{b}))$ as the polynomial $h$ maps
   $R$ to $R$ and cannot map a non-integral point to an integral point as its leading coefficient is a unit.

It follows that $(a,b) = (f,g)(\hat{a},\hat{b})$ cannot lie on any weakly $(f,g)$-invariant curve as the pullback of such a curve would
be weakly $(h,h)$-invariant.   Thus, $\overline{\cO_{(f,g)}(a,b)} = \AA^2$.
\end{proof}

\begin{lemma}
\label{trivdense}
Let $K$ be a field of characteristic zero and $f_1, \ldots, f_n \in K[x]$ a sequence of nonconstant polynomials over $K$.  We assume that each $f_i$ has degree at least two and is not conjugate to a monomial,
Chebyshev polynomial or negative Chebyshev polynomial. Then there is a rational point $a = (a_1, \ldots, a_n) \in \AA^n(K)$ with a dense $(f_1,\ldots,f_n)$-orbit.
\end{lemma}

\begin{proof}
Let $R \subseteq K$ be some finitely generated subring over which complete decompositions of each $f_i$ are defined and the leading coefficient of each indecomposable factor is a unit.  We argue by induction on $i$ that we can find some finitely generated ring $B$ containing $R$ and contained in $K$ for which there is a point $(a_1, \ldots, a_i) \in \AA^i(B)$ with $\cO_{(f_1,\ldots,f_i)}(a)$ Zariski dense in $\AA^i$.  In the case of $i = 1$, the result follows by  height considerations (for example, by embedding $R \subseteq {\mathbb C}$ if we take $a \in R$ with $|a| \gg 0$, then $\lim_{m \to \infty} f_1^\circ(a) = \infty$ so that, in particular, $a$ is not preperiodic).

In the inductive case, we have $(a_1,\ldots,a_i) \in \AA^i(B)$ with a Zariski dense $(f_1, \ldots, f_i)$-orbit.
Let $a_{i+1} \in K$ be any element of $K$ which is not integral over $B$.   Then for every $m$, $f^{\circ m}(a_{n+1})$ is also non-integral so by Lemma~\ref{integrality} $(f^{\circ m}(a_j),f^{\circ m}(a_{i+1}))$ does not belong to any $(f_j^{\circ m},f_{i+1}^{\circ m})$-invariant curve.   By triviality, it follows that $(f_1^{\circ m}(a),\ldots,f_{i+1}^{\circ m}(a))$ does not belong to any $(f_1^{\circ m},\ldots,f_{i+1}^{\circ m})$-invariant variety.
\end{proof}

Let us now combine these results to complete the proof of Theorem~\ref{Zhangconjdense}.

\begin{proof}
Reordering the indices if need be, we may express $(\AA^n,\Phi)$ as a product
$(\AA^{n_\ell},\lambda) \times (\AA^{n_G},\gamma) \times (\AA^{n_t},\tau)$ where $\lambda$ is given by a sequence of univariate linear polynomials,
$\gamma$ is given by a sequence of polynomials of degree at least two each conjugate to a monomial, Chebyshev polynomial, or
negative Chebyshev polynomial, and $\tau$ is
given by a sequence of disintegrated polynomials.  By Proposition~\ref{lineardense} there is some $a \in \AA^{n_\ell}(K)$ with $\cO_\lambda(a)$
Zariski dense in $\AA^{n_\ell}$, by Proposition~\ref{chebpowdense} there is some $b \in \AA^{n_G}(K)$ with $\cO_{\gamma}(b)$ Zariski
dense in $\AA^{n_G}$, and by Lemma~\ref{trivdense} there is some $c \in \AA^{n_t})(K)$ with $\cO_{\tau}(x)$ Zariski dense in $\AA^{n_t}$.
By Lemma~\ref{orthdense}, $\cO_\Phi((a,b,c))$ is Zariski dense in $\AA^n$.
\end{proof}

\subsection{Difference equations for Frobenius lifts}
\label{Froblift} \label{frobsect}
In this section we observe that for dynamical systems lifting the Frobenius, one can capture the periodic points with a difference equation.  Consequently, our results on the structure of difference varieties imply strong restrictions on
the algebraic relations amongst the periodic points of such dynamical systems.

\begin{nota}
In what follows, $K$ is a field with a  valuation $v$, ring of integers $R := \{ x \in K : v(x) \geq 0 \}$,  maximal ideal $\fm := \{ x \in R : v(x) > 0 \}$, and residue field $k := R/\fm$ of characteristic $p > 0$.  We assume that $\sigma:K \to K$ is an automorphism lifting the $p$-power Frobenius in the sense that $v(\sigma(x)) = v(x)$ for all $x \in K$ and $\sigma(x) \equiv x^p \mod \fm$ for $x \in R$.  We assume moreover that $K$ is maximally complete and algebraically closed.  The results we prove about periodic points descend from $K$ to subfields, so the reader may comfortably drop these last two hypotheses, but some of our intermediate results require at least completeness.  Ultimately, we shall assume that $K$ has characteristic zero, but for now, this is not necessary.
\end{nota}

\begin{nota}
If $X$ is a scheme over $R$, then we write $X_0$ for the base change of $X$ to $k$ and $X_\eta$ for the base change of $X$ to $K$.  We write $\pi:X(R) \to X_0(k)$ for the natural reduction map.
\end{nota}

With Theorem~\ref{frobeq} we show that difference equations given by liftings of the Frobenius give dynamical Teichm\"{u}ller maps.  Towards the end of this section we specialize to the case of dynamical systems given by sequences of univariate polynomials and thereby deduce form our earlier work that algebraic relations amongst periodic points of such systems are highly restricted.

\begin{theorem}
\label{frobeq}
Let $X$ be a separated scheme of finite type over $R$.  We assume that $X$ is smooth over $R$.  Suppose that $\Gamma \subseteq X \times X^{\sigma}$ is a closed subscheme of $X \times X^\sigma$ for which the projection $\Gamma \to X$ is \'{e}tale.  Suppose moreover that $q = p^n$ is a power of $p$ and $\Gamma$ lifts the Frobenius in the sense that some component of the special fibre $\Gamma_0$ is the graph of the geometric $q$-power Frobenius morphism $F:X_0 \to X_0^{(q)}$.  Then the reduction map $\pi:X(R) \to X_0(k)$ restricts to a bijection between $(X,\Gamma)^\sharp(R,\sigma^n)$ and $X_0(k)$.
\end{theorem}

\begin{proof}
To ease notation let us write $\rho := \sigma^n$.

Let us first show that $\pi:(X,\Gamma)^\sharp(R,\rho) \to X_0(k)$ is surjective.  Let $a \in X_0(k)$ be any $k$-rational point on $X_0$.  Pick any point $\widetilde{a} \in X(R)$ with $\pi(\widetilde{a}) = a$.  From the hypothesis that $X$ is smooth over $R$, we may fix an \'{e}tale covering $f:U \to \AA^m_R$ where $\widetilde{a} \in U(R)$, $U \subseteq X$ is an affine open subset and $f(a) = {\boldsymbol{0}}$.  Note that $f^\sigma:U^\sigma \to \AA^m_R$ gives analytic coordinates on $X^\sigma$ near $\sigma(\widetilde{a})$.

As $\Gamma \to X$ is \'{e}tale, the set $(f, f^\sigma)(\Gamma(R) \cap \pi^{-1}\{a\} \times (\pi^\sigma)^{-1}\{F(a)\})$ is the graph of an analytic function $g:\fm^m \to \fm^m$ where $g(x_1,\ldots,x_m) = (x_1^q,\ldots,x_m^q) \mod \fm \cdot R[[x_1,\ldots,x_m]]$.  That we can find a solution to $g({\boldsymbol x}) = \sigma({\boldsymbol x})$ follows from Newton's method (see~\cite{ScICM} in this context).

That is,
if for some $\gamma > 0$
we have a solution to $g(x) \equiv \sigma(x) \mod I_\gamma$ where $I_\gamma := \{ x \in R : v(x) \geq \gamma \}$, we can find some $x'$ with $x \equiv x' \mod I_\gamma$ but $g(x) \equiv \sigma(x) \mod I_{\gamma^+} := \{ x \in R : v(x) > \gamma \}$ and then taking limits we find a true solution with in the given neighborhood.  In our case, we already know that $g({\boldsymbol 0}) = {\boldsymbol 0} \mod \fm = I_{0^+}$.  Given an approximate solution $x$, suppose that $g(x) \equiv \sigma(x) \mod I_\gamma$ with $\gamma > 0$.  Let $\epsilon \in R$ with $v(\epsilon) = \gamma$.  We seek to find $x' = x + c \epsilon$ with $c = (c_1,\ldots,c_m)$ and $v(c_i) \geq 0$ for each $i$.  We have $g(x + c \epsilon) = g(x) + \sum_{i=1}^m \frac{\partial g}{\partial X_i}(x)  c \epsilon + \epsilon^2 \ast \equiv g(x) \mod I_{\gamma^+}$ as $\frac{\partial g}{\partial X_i}(X) \equiv  q X_i^q \mod \fm R[[X_1,\ldots,X_m]]$.  On the other hand, $\sigma(x + c \epsilon) = \sigma(x) + \sigma(c) \sigma(\epsilon) \equiv \sigma(x) + (c_1^q,\ldots,c_m^q) \sigma(\epsilon) \mod I_{\gamma^+}$.  Subtracting, we need only solve $\sigma(\epsilon)(c_1^q,\ldots,c_m^q) \equiv g(x) - \sigma(x) \mod I_{\gamma^+}$ .  By hypothesis, each component of $g(x) - \sigma(x)$ has valuation at least $\gamma = v(\sigma(\epsilon))$.  As $k$ is perfect, we may solve these equations.

These calculations demonstrate that the restriction of $\pi$ to $(X,\Gamma)^\sharp(R,\rho)$ is injective as well since the solution $c = (c_1,\ldots,c_n)$ is uniquely determined modulo $\fm$. Since we know the residue of the solution, this shows that the reduction map is injective.
\end{proof}

\begin{corollary}
\label{compositecor}
With $X$ and $\Gamma$ as in Theorem~\ref{frobeq}, for any natural number $N$ one has $(X,\Gamma)^\sharp(R,\rho) = (X,\Gamma^{\lozenge N})^\sharp(R,\rho^N)$.
\end{corollary}
\begin{proof}
A composite of \'{e}tale extensions is \'{e}tale.  Hence, the hypothesis of Theorem~\ref{frobeq} apply to $X$, $\Gamma^{\lozenge N}$, and $mN$.  So, $\pi:(X,\Gamma^{\lozenge N})^\sharp(R,\rho^N) \to X_0(k)$ is also a bijection.  As $(X,\Gamma)^\sharp(R,\rho) \subseteq (X,\Gamma^{\lozenge N})^\sharp(R,\rho^N)$, these sets must be equal.
\end{proof}

Specializing $\Gamma$ somewhat, we may use Theorem~\ref{frobeq} to find a difference equation for periodic points.

\begin{theorem}
\label{periodeq}
Let $X$ be a separated scheme of finite type over $R$, smooth over $R$ and $f:X \to X$ a morphism lifting the $q = p^n$-power Frobenius.  Let $\rho := \sigma^n$.  We assume that $f = f^\rho$ and $X = X^\rho$.  Then every $f$-periodic $R$-rational point belongs to $(X,f)^\sharp(R,\rho)$.
\end{theorem}
\begin{proof}
Let $b \in X(R)$ be an $f$-periodic point of order $M$.  There are only finitely many solutions to $f^{\circ M}(x) = x$ (as, for instance, this is true on the special fibre).  Hence, $\rho^N(b) = b$ for some $N > 0$.  Thus, $b$ satisfies $\rho^{MN}(x) = f^{\circ MN}(x)$.  That is, $b \in (X,f^{\circ MN})^\sharp(R,\rho^{MN})$ which is $(X,f)^\sharp(R,\rho)$ by  Corollary~\ref{compositecor}.
\end{proof}

\begin{Rk}
Theorem~\ref{periodeq} holds for $f$ analytic.  This observation yields interesting information in the case that $X$ is a moduli space of abelian varieties, $\Gamma \subseteq X \times X$ is a $p$-power Hecke correspondence, and $f:X \to X$ (or, really, $f$ is defined on some dense open subset) is a branch of $\Gamma$ lifting the Frobenius.  In this case, the difference equation captures the canonical lifts. (See~\cite{ScAO} for more details.)
\end{Rk}

\begin{Rk}
If in Theorem~\ref{periodeq} we assume that $k = {\mathbb F}_p^{\text{alg}}$, then as every point in $X(k)$ is $f$-periodic, every point in $(X,f)^\sharp(R,\rho)$ is $f$-periodic.
\end{Rk}

\begin{Rk}
This method of obtaining interesting difference equations for periodic points by lifting equations on the Frobenius has been used in the study of Manin-Mumford questions~\cite{Hr-MM,PR}.  When more structure (for instance, a group) is available, then more complicated equations beyond simply $f(x) = \sigma(x)$ may be used to give deeper information.  We expect that these equations in the more general dynamical context will be useful, but we have not pursued this issue.
\end{Rk}

Let us conclude by specializing to the case of sequences of univariate polynomials.

\begin{theorem}
\label{frobpolyper}
Let $q = p^\ell$ be a power of $p$.  We suppose that $K$ has characteristic zero.  Let $f_1,\ldots,f_n \in R[x]$ be polynomials with $f_i(x) \equiv x^p \mod \fm R[x]$ for each $i \leq n$.  We suppose that for some $m > 0$ each
$f_i = f_i^{\sigma^m}$
for each $i$.  If $X \subseteq \AA^n_K$ is an irreducible subvariety containing a Zariski dense set of points of the form $(\zeta_1,\ldots,\zeta_n)$ where $\zeta_i \in R$ is $f_i$-periodic, then $X$ is a difference subvariety of $(\AA^n,(f_1^{\lozenge m},\ldots,f_n^{\lozenge m}))$ and has the shape described in Theorem \ref{centralthm}.
Moreover, if $\deg(f_i) = q$ for each $i$, then we may replace the hypothesis ``$\zeta_i \in R$'' by ``$\zeta_i \in K$.''
\end{theorem}
\begin{proof}
By Theorem~\ref{periodeq}, the $(f_1,\ldots,f_n)$-periodic points in $\AA^n(R)$ are all contained in $(\AA^n,(f_1^{\lozenge m},\ldots,f_n^{\lozenge m}))^\sharp(R,\sigma^{\ell m})$.  Hence, if $X$ contains a Zariski dense set of periodic points from $\AA^n(R)$, then $X \cap (\AA^n,(f_1^{\lozenge m},\ldots,f_n^{\lozenge m}))^\sharp(R,\sigma^{\ell m})$ is Zariski dense in $X$ implying that $X$ is a difference subvariety of $(\AA^n, (f_1^{\lozenge m}, \ldots, f_n^{\lozenge m}) )$.  The description of $X$ now follows from our description of such difference varieties.

For the ``moreover'' clause observe that if $\deg(f_i) = q$, then every $f_i$-periodic point is integral over $R$, and, hence, actually an element of $R$ as $R$ is integrally closed in $K$.
\end{proof}

\begin{Rk}
Further specializing Theorem~\ref{frobpolyper} one obtains statements about algebraic relations amongst the periodic points of polynomial without reference to valuations as announced in the introduction.
For example, let $q$ be a power of a prime number $p$.  Suppose that $f(x) = x^q + p g(x)$ where $g(x) \in \ZZ[x]$ and $\deg(g) \leq q$.  Suppose moreover that $f$ is not linearly conjugate to a monomial or a Chebyshev polynomial.  Then every irreducible variety $X \subseteq \AA^n_\CC$ which contains a Zariski dense set of $n$-tuples of $f$-periodic points is defined by a sequence of equations of the form $f(x_i) = x_j$ or $g(x_\ell) = a$ for $a$ some fixed $f$-periodic point and $g$ a polynomial which commutes with $f$.
\end{Rk}

\bibliography{dynbibfinal}{}
\bibliographystyle{plain}

\end{document}